\documentclass[11pt]{amsart}

\usepackage{amscd,amssymb,amsmath,latexsym,enumerate}
\usepackage[utf8]{inputenc}
\usepackage{amscd,amssymb,amsmath,amsthm,bbm}
\usepackage[mathscr]{euscript}
\usepackage{mathrsfs}

\usepackage[breaklinks=true,colorlinks=true,
linkcolor=black,urlcolor=black,citecolor=black,
bookmarks=true,bookmarksopenlevel=2]{hyperref}

\DeclareMathAlphabet{\mathpzc}{OT1}{pzc}{m}{it}

\usepackage{xcolor}

\newtheorem*{definition*}{Definition}
\newtheorem{definition}{Definition}[section]
\newtheorem{theorem}[definition]{Theorem}
\newtheorem{proposition}[definition]{Proposition}
\newtheorem{lemma}[definition]{Lemma}

\newtheorem{corollary}[definition]{Corollary}

\theoremstyle{definition}
\newtheorem{remark}[definition]{Remark}
\newtheorem{example}[definition]{Example}

\newcommand{\CC}{{\mathbb C}}
\newcommand{\NN}{{\mathbb N}}

\newcommand{\RR}{{\mathbb R}}
\newcommand{\ZZ}{{\mathbb Z}}

\newcommand{\Cc}{{\mathcal C}}

\newcommand{\tS}{{\widetilde{S}}}

\newcommand{\tg}{{\tilde{g}}}

\newcommand{\supp}{{\operatorname{supp}}}			

\renewcommand{\subset}{\subseteq}

\allowdisplaybreaks

\begin{document}
\title{Linear repetitivity beyond abelian groups}
\author{Siegfried Beckus, Tobias Hartnick, Felix Pogorzelski}

\address{Department of Mathematics\\
University of Potsdam\\
Potsdam, Germany}
\email{beckus@uni-potsdam.de}

\address{Institute of Algebra and Geometry \\
KIT, Karlsruhe, Germany}
\email{tobias.hartnick@kit.edu}

\address{Department of Mathematics and Computer Science\\
University of Leipzig, Leipzig, Germany}
\email{felix.pogorzelski@math.uni-leipzig.de}

\begin{abstract} We show that linearly repetitive weighted Delone sets in groups of polynomial growth have a uniquely ergodic hull. This result applies in particular to the linearly repetitive weighted Delone sets in homogeneous Lie groups constructed in the companion paper \cite{BHP21-prim} using symbolic substitution methods. More generally, using the quasi-tiling method of Ornstein-Weiss, we establish unique ergodicity of hulls of weighted Delone sets in amenable unimodular lcsc groups under a new repetitivity condition which we call tempered repetitivity. For this purpose, we establish a general sub-additive convergence theorem, which also has applications concerning the existence of Banach densities and uniform approximation of the spectral distribution function of finite hopping range operators on Cayley graphs. 
\end{abstract}


\maketitle

\section{Introduction} \label{sec:intro}

The theory of aperiodic order is concerned with (generalizations of) Delone sets displaying some long range order while at the same time being far from periodic. Until recently, the theory has mostly been studied in Euclidean space or abelian locally compact groups - we refer the reader to \cite{BaakeGrimm13} for an extensive bibliography in the abelian setting. With any (weighted) Delone set in an abelian locally compact group one can associate a topological dynamical system, the so-called hull system of the Delone set, and these systems play a central role in the dynamical approach to the study of aperiodic order. A crucial question in this context is how properties of the dynamical system are reflected in the structure of the underlying Delone set and vice versa. For example, it is a famous result of Lagarias and Pleasants in \cite{LaPl03} that the hull of a linearly repetitive Delone set in $\RR^n$ is uniquely ergodic.\medskip

Recently, the scope of the theory of aperiodic order was extended to the realm of arbitrary (i.e.\@ not necessarily abelian) locally compact groups and, even more generally, proper homogeneous metric spaces, cf.\@
\cite{BjHa18,BjHaPo16,BjHaPoII, BjHaPo17} in collaboration with Bj\"orklund.
Further recent developments in this area can be found in \cite{BjHa3,BjHa2,FishExtensionsOfSchreiber,BjHaSt,BP18,BjHa2, Machado1, CordesHTonic, Hrushovski, Machado3, Machado2}. A particular focus has been on (aperiodic) Delone sets in locally compact second countable (lcsc) groups, and more specifically in constructing interesting examples of such sets. 
While the original focus of this program was on Delone sets arising from a non-abelian version of Meyer's cut-and-project construction, our most recent work \cite{BHP21-prim} introduces symbolic substitution techniques to a non-abelian setting. Using these techniques we are able to construct the first examples of aperiodic linearly repetitive Delone sets in non-abelian Lie groups. By analogy with the Lagarias-Pleasants theorem mentioned above, it is then natural to ask whether the hull dynamical systems of these Delone sets are uniquely ergodic. We will see that the answer to this question is positive for all examples provided in \cite{BHP21-prim}, but the proof of this fact will depend on a very peculiar property of our examples, namely exact polynomial growth of the ambient Lie group.
Our construction in the companion paper \cite{BHP21-prim} produces linearly repetitive Delone sets inside a certain class of Lie groups called \emph{homogeneous Lie groups} (see \cite{FoSte82} or \cite[Chapter~3]{FR16} for a more modern treatment). Lie groups from this class turn out to be nilpotent (hence in particular amenable) and admit a canonical quasi-isometry class of left-invariant metrics called \emph{homogeneous metrics}, which induce the given topology.

If $G$ is a homogeneous Lie group with homogeneous metric $d$, then a subset $\Lambda \subset G$ is called a \emph{Delone set} if there exist $R > r > 0$ such that $d(x,y) \geq r$ for all distinct $x,y \in \Lambda$ and such that every $x \in G$ is at distance at most $R$ from $\Lambda$. It is called \emph{linearly repetitive} if there exists a constant $c_\Lambda\geq 1$ such that
every pattern of radius $\rho$ of $\Lambda$ occurs in every ball of radius $c_\Lambda\rho$ (this is made precise in Definition~\ref{DefLinRep} below). By definition, the \emph{hull} of a linearly repetitive Delone set is its orbit closure with respect to the Chabauty--Fell topology on the space of closed subsets of $G$. We then have the following generalization of the Lagarias-Pleasents theorem, which we expect to draw interesting connections to questions of spectral convergence via dynamical systems, cf.\@ \cite{BBdN18,BP18}.

\begin{theorem}[Unique ergodicity from linear repetitivity, homogeneous case]\label{Intro1}
	Let $G$ be a homogeneous Lie group with a homogeneous metric $d$.
	Then the hull of every Delone set in $G$ that is linearly repetitive {with respect to $d$} is minimal and uniquely ergodic. 
\end{theorem}

We emphasize that there are plenty of examples to which Theorem~\ref{Intro1} applies. Indeed, in \cite{BHP21-prim} we were able to construct non-periodic repetitive Delone sets,  which are linearly repetitive with respect to a homogeneous metric on the ambient group, in 133 of the 149 families of indecomposible $1$-connected nilpotent $7$-dimensional real Lie groups classified in \cite{Gong98}. In fact these examples were a key motivation for us to establish Theorem~\ref{Intro1}.

We will actually prove a slightly stronger result than Theorem~\ref{Intro1}, which applies to all \emph{weighted} Delone sets which are \emph{almost linearly repetitive} (this is made precise below). Minimality of the hull holds in an even wider context, see Proposition~\ref{prop:minmin} below, and can be established by adapting ideas from \cite{FrRi14} to the setting of weighted Delone sets. The main contribution of the present article lies in proving unique ergodicity, and it is in this part of the proof that we use the assumption that $G$ is a homogeneous Lie group with homogeneous metric $d$. As explained in \cite[Proposition~3.36]{BHP21-prim}, these assumptions imply that the pair $(G,d)$ has \emph{exact polynomial growth} in the sense of the following definition (Definition~\ref{defi:exactgrowth} below):

\begin{definition*}
Let $G$ be an lcsc group with left-Haar measure $m_G$ and let $d$ be an adapted (i.e. left-invariant, proper and locally bounded) metric on $G$. We say that $G$ has \emph{exact polynomial growth with respect to $d$} if there are constants $C > 0$ and $q \geq 0$ such that 
	\[
	\lim_{t \to \infty} \frac{m_G(B_t)}{C t^{q}} = 1,
	\]
	where $B_t$ denotes the open ball of radius $t$ around the identity of $G$.
\end{definition*}

Every compactly-generated lcsc group which has exact polynomial growth with respect to some metric $d$ as above obviously has polynomial growth in the sense of geometric group theory.
Conversely, by a celebrated results of Breuillard \cite{Bre14}, every compactly-generated lcsc group of polynomial growth admits an adapted metric of \emph{exact} polynomial growth, see Theorem~\ref{Bre2}. If $G$ is moreover a connected Lie group, then this metric can be chosen to be continuous (in fact, Riemannian and in particular real-analytic), but it seems to be an open problem whether this is the case in general. It turns out that exact polynomial growth is actually sufficient to establish Theorem~\ref{Intro1}:

\begin{theorem}[Unique ergodicity from linear repetitivity, exact case]\label{Intro1a}
	If an lcsc group $G$ has exact polynomial growth with respect to an adapted metric $d$,
	then the hull of every (weighted) Delone set in $G$ that is (almost) linearly repetitive {with respect to $d$} is minimal and uniquely ergodic. 
\end{theorem}

We establish Theorem~\ref{Intro1a} in Theorem~\ref{LRConvenient} below; note that by the previous remarks, Theorem~\ref{Intro1} is merely a special case. The assumption that $G$ has exact polynomial growth with respect to $d$ implies that metric balls around the identity form a F\o lner sequence in $G$, and even a \emph{strong F\o lner exhaustion sequence} in the sense of Definition~\ref{DefStrongFES} below. Linear repetitivity with respect to $d$ can then be reformulated in terms of this F\o lner sequence, and this leads to the notion of \emph{(almost) tempered repetitivity} of a Delone set with respect to a F\o lner sequence (see Definition~\ref{DefATR} below). Using this notion, one can reformulate Theorem~\ref{Intro1a} in terms of F\o lner sequences rather than metrics, and this metric-free formulation actually holds in much greater generality:

\begin{theorem}[Unique ergodicity from tempered repetitivity]\label{Intro2} Let $(T_m)$ be a strong F\o lner exhaustion sequence in an amenable unimodular lcsc group $G$ and let $\Lambda$ be a weighted Delone set in $G$. If $\Lambda$ is almost tempered repetitive with respect to $(T_m)$, then the hull of $\Lambda$ is minimal and uniquely ergodic.
\end{theorem}

Theorem~\ref{Intro1a} is indeed a special case of Theorem~\ref{Intro2}, since linear repetitivity with respect to an adapted metric of exact polynomial growth is equivalent to tempered repetitivity with respect to the corresponding strong F\o lner exhaustion sequence of balls (see Proposition~\ref{prop:Temp-LR} below).

It seems to us that Theorem~\ref{Intro2} is the most natural generalization of the Lagarias--Pleasants theorem to the general amenable case: By Proposition~\ref{prop:strongfolnerexist} below, strong F{\o}lner exhaustion sequences exist in all unimodular, amenable lcsc groups, but it is not clear at all in which cases (beyond the case of exact polynomial growth) they can be chosen to be balls with respect to a suitable metric, hence the need for a metric free formulation. Theorem~\ref{Intro2} indicates that, in this metric-free context, \emph{tempered repetitivity} is the correct replacement for linear repetitivity, hence we would like to advertise this notion.

While Theorem~\ref{Intro2} is formulated in terms of Delone dynamical system, it is easy to deduce a corresponding unique ergodicity result in the symbolic setting. To formulate such a result, let $\mathcal A$ be a finite alphabet and $\Gamma$ be a countable amenable group. We refer to an element of $\mathcal{A}^{\Gamma}$ as a \emph{coloring} of $\Gamma$ by $\mathcal A$ and equip $\mathcal A^\Gamma$ with its natural compact metrizable topology, given by the product topology. Given a coloring $C \in \mathcal{A}^{\Gamma}$ we refer to its orbit closure in $\mathcal A^\Gamma$ as the \emph{hull} of $C$. Similarly to the case of Delone sets, one can also define the notions of a symbolically tempered repetitive coloring (with respect to a F\o lner sequence in $\Gamma$) or a symbolically linearly repetitive coloring (with respect to a left-invariant metric on $\Gamma$), see Definition~\ref{Def-SymbLR} below.
In fact, the situation is slightly simpler than in the Delone setting, since for countable groups the notions of F{\o}lner sequence and strong F{\o}lner sequence are equivalent by \cite[Lemma~2.7~(d)]{PS16}. In any case, the following result can then be deduced easily from Theorem~\ref{Intro2}, using the fact that colorings of $\Gamma$ can be identified with certain weighted Delone sets in $\Gamma$ (see Section~\ref{sec:symbolic} for details):

\begin{theorem}[Unique ergodicity in the symbolic setting]
	\label{thm:LRsymbolic}
		Let $\Gamma$ be a countable amenable group and let $\mathcal{A}$ be a finite set. If $\mathcal{C} \in \mathcal{A}^{\Gamma}$ is symbolically tempered repetitive with respect to a F{\o}lner exhaustion sequence, then the hull of $\mathcal{C}$ is minimal and uniquely ergodic. 
		In particular, if $\Gamma$ has exact polynomial growth with respect to an adapted metric $d_{\Gamma}$ and if  $\mathcal{C}$ is symbolically linearly repetitive with respect to $d_{\Gamma}$, then the hull of $\mathcal{C}$ is minimal and uniquely ergodic.  
\end{theorem}

Let us now discuss some of the ingredients of the proof of Theorem~\ref{Intro2}. As in the cases of  Theorems~\ref{Intro1} and~\ref{Intro1a}, the minimality result of Theorem~\ref{Intro2} can be obtained by adapting arguments from \cite{FrRi14}, whereas the unique ergodicity result requires new ideas. 
Our approach is based on the observation by Damanik and Lenz \cite{DaLe01} that in the Euclidean setting, linear repetitivity implies the validity of a uniform sub-additive convergence theorem for certain functions defined on the collection $\mathcal{B}(\RR^n)$ of all $n$-dimensional boxes in $\RR^n$, which in turn can be used to ensure unique ergodicity. This observation has been refined in many different directions, see e.g.\@ \cite{Len02b, Len03, LMV08, LSV11, BBL13, Pog13, Pog14, FrRi14, PS16}.
A  different approach is via strongly almost periodic translation bounded measures, see \cite{LLRSS}. In Theorem~\ref{thm:abstr} below we establish a general version of such a convergence theorem for a large class of amenable groups, which is general enough to allow us to deduce Theorem~\ref{Intro2} (and thereby also Theorems~\ref{Intro1},~\ref{Intro1a} and~\ref{thm:LRsymbolic}). While the precise statement of the convergence theorem is rather technical, let us highlight two of the main features of the theorem and its proof:

\begin{itemize}
\item The theorem works for general amenable unimodular lcsc groups and for a rather general class of weight functions. To achieve this, the study of $n$-dimensional boxes in $\RR^n$ is replaced by the Ornstein-Weiss machinery of $\varepsilon$-quasi-tilings \cite{OW87} and more specifically its variant from \cite{PS16}.
\item There is no requirement that the Delone set under consideration has finite local complexity. A consequence of this is that Theorem~\ref{Intro1a} also holds for \emph{almost} linearly repetitive Delone sets (and Theorem~\ref{Intro2} also holds for \emph{almost} tempered repetitive Delone sets). This is achieved using ideas from Frettl\"oh-Richard \cite{FrRi14}. Roughly speaking, to incorporate non-FLC situations, one has to identify not only patches which are equal up to a translation, but also patches which are close up to a translation. The fact that being $\delta$-close for a small $\delta>0$ is not an equivalence relation creates all kinds of technical difficulties.
\end{itemize}

\medskip

It is well-known in the abelian setting that uniform sub-additive convergence theorems have applications beyond unique ergodicity of Delone dynamical systems, and this is no different in our more general setting. We demonstrate the strength of our convergence theorem by two such applications one to Banach densities and one to the integrated densities of states (IDS). We deduce the corresponding theorems along the lines of the corresponding abelian situations.\medskip

 Upper and lower Banach densities appear as an important 
combinatorial quantity in many mathematical areas, see 
for instance a recent example \cite{DHZ19} for countable amenable groups. 
For sets with enough symmetry, the upper and the lower Banach 
density may coincide and one can just refer to
the Banach density of a set. In analogy to the abelian realm, this situation 
occurs for almost tempered repetitive weighted Delone sets. 
For the definition of the upper and lower Banach density
of a weighted Delone set we refer to Section~\ref{sec:applications}. Given a weighted Delone set $\Pi$, we denote by $\delta_\Pi$ the associated Dirac comb which is a Radon measure on the group.

\begin{corollary} \label{cor:Banachdensities}
Let $G$ be an amenable unimodular lcsc group with Haar measure $m_G$. Suppose that $\Lambda$ is a  weighted Delone set that is almost tempered repetitive with respect to a strong F{\o}lner exhaustion sequence $(T_m)$ of $G$. Then the {\em Banach density} 
	\[
	b_{\Lambda} := \lim_{m \to \infty} \frac{\delta_\Pi(T_m^{-1})}{m_G(T_m)} 
	\]
	exists uniformly in $\Pi \in \mathcal{H}_\Lambda$.
\end{corollary}

\medskip

The implications of our convergence theorem concerning the IDS are a bit more technical to state. We give a sample result in the setting of finitely generated amenable groups $\Gamma$ whose elements are labeled by colors taken from a finite set $\mathcal{A}$. Since each coloring $\mathcal{C} \in \mathcal{A}^{\Gamma}$ can be interpreted as a weighted Delone set in $\Gamma$, it makes sense to define a version of tempered repetitivity for a coloring. In this context, our convergence theorem allows us to verify a criterion from \cite{PS16} which then shows that the IDS for certain pattern-equivariant operators 
on graphs can be uniformly approximated by finite volume analogs. The precise class of considered operators is defined in Definition~\ref{defi:operators}.

\begin{corollary} \label{cor:IDS}
Let $\Gamma$ be a finitely generated amenable group and let $\mathcal{C} \in \mathcal{A}^{\Gamma}$
be a coloring of the group by a finite set  $\mathcal{A}$. Suppose that $\mathcal{C}$ is symbolically tempered repetitive as a weighted Delone set with respect to some strong F{\o}lner exhaustion sequence $(T_m)$. Then for every $\mathcal{C}$-invariant self-adjoint operator $H: \ell^2(\Gamma) \to \ell^2(\Gamma)$
and for its integrated density of states $N_H: \RR \to [0,1]$, we get 
\[
\lim_{m \to \infty} \big\| N_H - N_{H_m} \big\|_{\infty} =0,
\]
where the $N_{H_m}$ denote the empirical eigenvalue distribution functions of restrictions 
$H_m$ of $H$ to a subspace of $\ell^2(T_m^{-1})$.
\end{corollary}

This article is organized as follows. In Section~\ref{prelim} we introduce Delone dynamical systems and recall
the equivalence of minimality of these systems and (almost) repetitivity of the underlying set. Section~\ref{ergodictheorem}
is devoted to the main  convergence theorem, cf.\@ Theorem~\ref{thm:abstr}, which is stated and proven 
for so-called admissible weight functions defined on hulls of almost tempered weighted Delone sets.
From this we derive in Section~\ref{UE} the unique ergodicity of the hulls of almost tempered repetitive weighted Delone sets, as stated in Theorem~\ref{Intro2}. Moreover, we derive
Theorems~\ref{Intro1a} and~\ref{Intro1} from Theorem~\ref{Intro2}, cf.\@ Theorem~\ref{LRConvenient}.
Section~\ref{sec:symbolic} is devoted to symbolic systems over a finite set and the proof of Theorem~\ref{thm:LRsymbolic}. The aforementioned applications including the proofs of the Corollaries~\ref{cor:Banachdensities} and~\ref{cor:IDS}
are carried out in Section~\ref{sec:applications}. The Appendices~A and~B contain proofs of some technical lemmas on almost sub-additive weight 
functions and topological aspects of dynamical systems induced by weighted Delone sets.

\medskip

\textbf{Acknowledgments}.
We thank Daniel Lenz for an inspiring discussion on linear repetitive Delone sets and related sub-additive ergodic theorems. Moreover, we are thankful to Christoph Richard for conversations on almost linear repetitivity. We are grateful to Michael Bj\"orklund and Amos Nevo for discussions on F{\o}lner conditions for balls in groups of polynomial growth. We thank Hung Ninh Ngoc for carrying out explicit volume computations in the Heisenberg group with respect to the Cygan-Kor{\'a}nyi metric.

\section{Preliminaries on Delone dynamical systems} \label{prelim}

\subsection{Delone dynamical systems}\label{SecDelDynSys}
We recall that a subset $P$ of a metric space $(X,d)$ is called 
\begin{enumerate}
\item \emph{uniformly discrete} if there exists $r>0$ such that $d(x,y) \geq r$ for all distinct $x,y \in P$;
\item \emph{relatively dense} if there exists $R>0$ such that $X$ is an $R$-neighbourhood of $P$;
\item \emph{Delone} if it is uniformly discrete and relatively dense.
\end{enumerate}
If $P$ is a Delone set, $\sigma \geq 1$ and $\alpha: P \to [\sigma^{-1}, \sigma]$ is a function, then the pair $(P, \alpha)$ is called a \emph{weighted Delone set}. We consider Delone sets as weighted Delone sets with constant weight $1$.

Let $G$ be an lcsc group and let $d$ be a metric on $G$. Following \cite{CdlH}, we say that a metric $d$ on $G$ is 
\begin{itemize}
\item \emph{left-invariant}, if $d(gh, gk)  = d(h,k)$ for all $g,h, k \in G$;
\item \emph{proper} if $B_t := \{g \in G \mid d(g,e) < t\}$ is relatively compact for all $t > 0$;
\item \emph{locally bounded} if every $g\in G$ admits a neighbourhood of finite diameter with respect to $d$;
\item \emph{adapted} if it is left-invariant, proper and locally bounded.
\end{itemize}
By \cite[Prop. 2.A.9]{CdlH} any continuous proper left-invariant metric generates the topology on $G$. We observe that if $d$ and $d'$ are continuous metrics on $G$ which are coarsely equivalent (in the sense of  \cite[p.~11]{CdlH})), then they define the same class of Delone sets.
We will need the following facts  concerning such metrics (see \cite[Section 1.D]{CdlH}):

\begin{proposition}
\label{AdaptedMetrics} 
Every lcsc group admits a continuous proper left-invariant metric $d$. Any such metric is automatically adapted and generates the underlying topology. Moreover, any two adapted metrics on $G$ are coarsely equivalent (even if they are not continuous).\qed
\end{proposition}

In view of the proposition we say that a subset $P \subset G$ is uniformly discrete, relatively dense or Delone if it has the corresponding property with respect to some (hence any) continuous adapted metric. For the remainder of this section we fix an adapted metric $d$ on $G$. Later on (see Section~\ref{UE}  below) we will have to choose an adapted metric with very peculiar properties, but for the moment the precise choice of metric will not matter to us. In view of Proposition~\ref{AdaptedMetrics} and the previous considerations, we assume as well that $d$ is continuous, and we make this assumption throughout this and the next section. We then denote by \[
B_t := \{g \in G \mid d(g,e) < t\} \quad \text{and} \quad \overline{B}_t := \{g \in G \mid d(g,e) \leq t\}
\]
the open, respectively closed ball  around the identity so that $B_s \subset \overline{B}_s \subset B_t$ for all $s < t$. 
The open and closed balls around arbitrary centers $g \in G$ are defined as $B_t(g) = gB_t$, and $\overline{B}_t(g) = g \overline{B}_t $, respectively.
We emphasize that $\overline{B}_r$ may be larger than the closure $\overline{B_r}$ of $B_r$. For example, in the non-Archimedean case $\overline{B_r} = B_r$ 
since the latter is compact. 
Due to our continuity assumption on $d$ we know that $B_r$ is open and $\overline{B}_r$ is closed; neither of these statements would be true for a general (i.e.\@ possibly discontinuous) adapted metric.

By \cite{BjHa18}, a subset $P \subset G$ is uniformly discrete with respect to $d$ (or any other continuous adapted metric on $G$) if and only if there exists an open relatively-compact subset $U \subset G$ such that $\#(gU \cap P)\leq 1$ for all $g \in G$ and relatively dense if and only if there exists a compact subset $K \subset G$ such that $P$ is left-$K$-syndetic in the sense that $P. K = G$, where for two subsets $A,B\subseteq G$, the notation $A.B$ refers to the Minkowski product of $A$ and $B$. With a slight abuse of notation we write $gA$ for $\{g\}.A$, where $g\in G$ and $A\subseteq G$. We then say that $P$ is \emph{$U$}-uniformly discrete and $K$-relatively-dense respectively, and refer to a set with both of these properties as a \emph{$(U,K)$-Delone set}. Finally, we denote by
\[
\operatorname{Del}(U,K,\sigma) := \Big\{ (P,\alpha):\, P \mbox{ is } (U,K)\mbox{-Delone and } \alpha:P \to [\sigma^{-1},\sigma] \Big\}
\]
the collection of weighted Delone sets with uniform parameters $U, K, \sigma$. Note that $G$ acts on $\operatorname{Del}(U,K,\sigma)$ by $g.(P, \alpha) := (gP, g_*\alpha)$, where $g_*\alpha(q) := \alpha(g^{-1}q)$.

\medskip

We will identify each $(P, \alpha) \in \operatorname{Del}(U,K,\sigma)$ with the associated Dirac comb
\[
\delta_{(P, \alpha)} := \sum_{x \in P} \alpha(x) \cdot \delta_x,
\]
and thereby think of $\operatorname{Del}(U,K,\sigma)$ as a $G$-invariant subset of the space of Radon measures on $G$. We will always equip the latter space with the weak-$*$-topology. Then by \cite{BaLe04, BjHaPoII} the subspace $\operatorname{Del}(U,K,\sigma)$ is compact, and the $G$-action on this space is jointly continuous. The topology induced on the space  $\operatorname{Del}(U,K,1)$ of (unweighted) Delone sets is precisely the Chabauty--Fell topology (see \cite{BjHaPoII} and also \cite[Section 2.1]{MR13} and the references therein).\medskip

We will need yet another description of the topology on $\operatorname{Del}(U,K,\sigma)$: Given a relatively compact subset $S \subset G$ and two Radon measures $\mu$ and $\nu$ on $G$ we define
\begin{equation} \label{eqn:formuladistMeasure}
\quad d_S(\mu, \nu) := \inf\big\{ \delta > 0:\, \big| \mu\big( B_{\delta}(y) \big) -\nu\big( B_{\delta}(y) \big)  \big| < \delta  \mbox{ for all } y \in S \cap (\mathrm{\supp}(\mu)\cup \mathrm{\supp}(\nu)) \big\}.
\end{equation}
In particular, using our identification of weighted Delone sets and their corresponding Dirac combs, we define
\begin{equation} \label{eqn:formuladist}
d_S(\Lambda, \Pi) := d_S(\delta_\Lambda, \delta_\Pi)
\end{equation}
for weighted Delone sets $\Lambda, \Pi$. Then, given  $\Lambda \in \operatorname{Del}(U,K,\sigma)$, the sets 
\[
\mathcal{U}_{S, \delta}(\Lambda) := \big\{ \Pi  \in \operatorname{Del}(U,K,\sigma):\, d_S(\Lambda, \Pi) < \delta  \big\}
\]
form a neighbourhood basis of $\Lambda$ in $\operatorname{Del}(U,K,\sigma)$, as $\delta$ and $S$ vary over all positive real numbers and all relatively compact subsets of $G$ respectively (cf.\@ Proposition~\ref{prop:neighborhood} in Appendix~\ref{AppendixB}). In particular, Delone sets $\Pi_n$ converge to $\Lambda$ in $\operatorname{Del}(U,K,\sigma)$ if and only if $\lim_{n \to \infty} d_S(\Pi_n, \Lambda) = 0$ for all relatively compact $S$.\medskip

Given $\Lambda \in \operatorname{Del}(U,K,\sigma)$, the orbit closure $\mathcal{H}_{\Lambda} :=  \overline{\big\{ g.\Lambda :\, g \in G\big\}}$ is compact an the $G$-action on $\mathcal{H}_{\Lambda}$ is jointly continuous, i.e.\  $G\curvearrowright \mathcal{H}_{\Lambda}$ is a topological dynamical system over $G$. We refer to $\mathcal H_\Lambda$ as the \emph{hull} of $\Lambda$ and to $G\curvearrowright \mathcal{H}_{\Lambda}$ as a  {\em Delone dynamical system}.

\subsection{Almost repetitivity and minimality} \label{sec:arepmin}
In this subsection we give a combinatorial characterization of minimality of Delone dynamical systems, generalizing results of \cite{FrRi14} in the unweighted case. Throughout this subsection we fix a weighted Delone set $\Lambda = (P, \alpha) \in \operatorname{Del}(U,K,\sigma)$ with associated Dirac comb $\delta_\Lambda$. If $S \subset T$ are two relatively compact subsets of $G$ and $\nu$ is a Radon measure on $T$, then we denote by $\nu|_S$ the restriction of $\nu$ to the Borel $\sigma$-algebra of $S$.
 \medskip

Given a relatively compact subset $S \subset G$ we refer to the pair $\Lambda_S:=({\delta_\Lambda}|_{S}, S)$ as the \emph{$S$-patch} of $\Lambda$. If $p = (\mu, S)$ is a patch of $\Lambda$ we refer to $\mu$ as the \emph{underlying measure} and $S$ as the \emph{support} of $p$ (which is not to be confused with the support of $\mu$). For $T \subset S$ we define the \emph{restriction} of $p$ to $T$ by $p|_{T} := (\mu|_{T}, T)$.

\medskip

Two patches $p = (\mu, S)$ and $q = (\nu, T)$ of $\Lambda$ are called \emph{equivalent} if there exists $g \in G$ with $g.p := (g_*\mu, gS) = (\nu, T)$, and an equivalence class of patches is called a \emph{pattern}. A pattern is said to be of \emph{size $S$} for a relatively compact set $S$ if the supports of its patches are translates of $S$.

\medskip
 
Given $\delta \geq 0$, we say that two patches $p = (\mu, S)$ and $q = (\nu, T)$ of $\Lambda$ are $\delta$-similar if there exists $g \in G$ such that $gS = T$ and $d_T(g_*\mu, \nu) < \delta$. By definition, two patches are $0$-similar if and only if they are equivalent. In particular, $0$-similarity is an equivalence relation, whereas $\delta$-similarity is not a transitive relation for $\delta>0$.

\medskip

Given a relatively compact subset $T \subset G$ and $\delta \geq 0$ we say that a pattern $[p]$ of $\Lambda$ with representative $p = (\mu, S)$ \emph{$\delta$-occurs in $T$} if there is some $g \in G$ such that $gS \subset T$ and $p$ is $\delta$-similar to $\Lambda_{gS}$. If $\delta=0$ we simply say that $[p]$ \emph{occurs in $T$}. 
Similarly, a patch $p=(\mu,S)$ ($\delta$-)occurs in a patch $q=(\nu,T)$ if there is some $g\in G$ such that $gS\subseteq T$ and $p$ is ($\delta$-)similar to $q|_{gS}$.

 \medskip

\begin{definition} 
A function $\varrho: (0,1) \times [1, \infty) \to [1, \infty)$ is called an {\em almost repetitivity function} for  $\Lambda$ if for every $\delta \in (0,1)$ and all $R \geq 1$ every pattern of $\Lambda$ of size $B_R$ $\delta$-occurs in $B_{\varrho(\delta, R)}(h)$ for all $h \in G$. If such a function exists, then $\Lambda$ is called {\em almost repetitive}. 
\end{definition}

Note that $\varrho(\delta,R)\geq R$ for all $\delta>0$.

\begin{remark}
Some authors define linear repetitivity by asking the above condition to hold for all $R\geq r_0$, where $r_0>0$ is an arbitrary constant, whereas we insist here that $r_0 =1$. While the two possible definitions are a priori different, this difference is irrelevant for us. Namely, we can always rescale the metric by a constant to achieve that the condition holds with $r_0=1$. The reader is invited to check that all our results are invariant under such rescalings, and hence all our theorems hold true for either definition. In view of this fact, we will always assume $r_0 = 1$ to keep the notation simple.
\end{remark}

\begin{proposition} 
\label{prop:minmin}
Let $\Lambda$ be a weighted Delone set in $G$. Then $\Lambda$ is almost repetitive 
if and only if the associated Delone dynamical system $G \curvearrowright \mathcal H_\Lambda$ is minimal. 
\end{proposition}

\begin{proof} We closely follow the proof of \cite[Theorem~3.11]{FrRi14}.

\medskip

	Suppose the hull of $\Lambda = (P, \alpha)$ is minimal. Let $R \geq 1$ and fix $\Pi= (Q,\beta) \in \mathcal{H}_{\Lambda}$. Assume that $0 <\varepsilon < R$. We further choose $R^{\prime} \geq R$ such that $B_{R^{\prime}} \cap Q = \overline{B}_R \cap Q$. We set 
	\[
	\mathcal{U}_{R,\varepsilon} \big( \Pi \big) = 
	\big\{ \Theta \in \mathcal{H}_{\Lambda}:\, 
		d_{B_{R^{\prime}}}(\Theta,\Pi)<\varepsilon
	\big\}.
	\]
	Then $\mathcal{U}_{R,\varepsilon} \big( \Pi \big)$ is a neighborhood of $\Pi$ in the weak-$\ast$-topology, cf.\@ Proposition~\ref{prop:neighborhood}. By minimality of $\Lambda$, every element in $\mathcal{H}_{\Lambda}$ has a dense orbit. This gives
	\[
	\mathcal{H}_{\Lambda} = \bigcup_{x \in G} x \, \mathcal{U}_{R,\varepsilon} \big( \Pi \big).
	\]
	By compactness of the hull, we find finitely many $x_1, x_2, \dots, x_{\ell}$ such that $\mathcal{H}_{\Lambda} \subseteq \bigcup_{j=1}^{\ell} x_j \, \mathcal{U}_{R,\varepsilon} \big( \Pi \big)$. We now choose $\varrho_1(R) > 0$ such that $x_j \in B_{\varrho_1(R)}$ for all $1 \leq j \leq \ell$. This means that for all $h \in G$ there is some $h^{*} \in B_{\varrho_1(R)}$ such that $h.(P,\alpha) = h^{*}.(D,\gamma)$ and $(D,\gamma)$ is some element in $\mathcal{U}_{R,\varepsilon}(\Pi)$. This also gives $D \cap B_{R^{\prime}} = h^{*\, -1} h\, P \cap B_{R^{\prime}}$. Since the patches $\big(\delta_{(Q \cap B_{R^{\prime}}, \beta_{| Q \cap B_{R^{\prime}}})}, B_{R^{\prime}} \big)$ and  $\big( \delta_{(D \cap B_{R^{\prime}}, \gamma_{| D \cap B_{R^{\prime}}})}, B_{R^{\prime}} \big)$ are $\varepsilon$-similar the left-invariance of the metric tells us that also
	$\big( \delta_{(Q \cap B_{R^{\prime}}, \beta_{| Q \cap B_{R^{\prime}}})}, B_{R^{\prime}} \big)$ is $\varepsilon$-similar to the $R^{\prime}$-patch of $P$ centered at $h^{-1}h^{*}$. We will now show that $h^{-1}h^{*} B_R B_{\varepsilon} \subseteq B_{\varrho_1(R) + 2R}(h^{-1})$. The condition $\varepsilon < R$ gives $B_R B_{\varepsilon} \subseteq B_{2R}$. Further, we have $h^{*}B_{2R} \subseteq B_{\varrho_1(R)}B_{2R} \subseteq B_{\varrho_1(R) + 2R}$. Hence, we obtain $h^{-1}h^{*}B_RB_{\varepsilon} \subseteq B_{\varrho_1(R) + 2R}(h^{-1})$, as claimed. Since $\Pi$ was chosen arbitrarily this shows that for all possible $B_R$-patterns $[p]$ of $\Lambda$ and for each $h \in G$ there is some $R$-patch of $\Lambda$ contained in $B_{\varrho_1(R) + 2R}(h)$ which is $\varepsilon$-similar to a representative of $[p]$. A straight forward compactness argument shows that there is a finite number of patterns $[p_k]$ such that every $B_{R}$-patch of $\Lambda$ is $\varepsilon$-similar to a representative of one of the $[p_k]$. Hence, repeating the above procedure finitely many times, we can set $\varrho(R) := \max_{k} \varrho_k(R) + 2R$ and we find that $\Lambda$ is almost repetitive with almost repetitivity function $\varrho$.
	
	\medskip
	
	Conversely, assume by contradiction that $\Lambda = (P,\alpha)$ is almost repetitive but $\mathcal{H}_\Lambda$ is not minimal. Let $\Pi = (Q,\beta) \in \mathcal{H}_{\Lambda}$ and $\{g.\Pi:\,g\in G\}$ is not dense in $\mathcal{H}_{\Lambda}$. This implies $\Lambda \notin \mathcal{H}_{\Pi}$ as well as $\mathcal{H}_{\Pi} \subsetneq \mathcal{H}_{\Lambda}$. We fix a compact neighborhood $\mathcal{V} \subseteq \mathcal{H}_{\Lambda}$ of $\Lambda$ containing the set $\{\Theta \in \mathcal{H}_{\Lambda}:\, d_{B_{1/\varepsilon}} \big( \Theta,\, \Lambda \big) < \varepsilon \} $ for some small $\varepsilon > 0$ such that $\mathcal{V} \cap \mathcal{H}_{\Pi} = \emptyset$. We now take a compact set $K \subseteq G$ with $B_{1/\varepsilon}(e) \subseteq \mathring{K}$ and define $T_{K, \varepsilon}\big( \Lambda \big) = \{g \in G:\, d_{\mathring{K}}\big(g.\Lambda, \Lambda \big) < \varepsilon \}$. 
	
	\medskip
	
	{\bf Claim:} The set $T_{K, \varepsilon}\big( \Lambda \big)$ is right-relatively dense. 
	
	\medskip
	
	Assuming the validity of the claim for a moment, we see 
	\[
	G.\Lambda =  \big( K^{\prime}T_{K, \varepsilon}(\Lambda)\big).\Lambda \subseteq K^{\prime}.\mathcal{V}
	\]
	for some compact set $K^{\prime} \subseteq G$. By compactness of $\mathcal{V}$ and the continuity of the $G$-action, we find that the set $K^{\prime}.\mathcal{V}$ is compact. Hence we arrive at $\mathcal{H}_{\Lambda} \subseteq K^{\prime}.\mathcal{V}$. We get that
	$\Pi = x.\Theta$ for some $x \in K^{\prime}$ and some $\Theta \in \mathcal{V}$. Thus, $\Theta = x^{-1}\Pi$  contradicting the fact that $\mathcal{V} \cap \mathcal{H}_{\Pi} = \emptyset$.
	
	\medskip
	
	It remains to show the claim. For this we proceed along the lines of the proof of Lemma~3.6, implication~(iii) $\Rightarrow$ (i) in \cite{FrRi14}. To this end, find some $r \geq 1$ such that $\mathring{K} \subseteq B_r$ and set $R:= \varrho(\varepsilon, r)$, where $\varrho$ is an almost repetitivity function of $\Lambda$. We fix a countable cover of $G$ of the form $G = \bigcup_{i=1}^{\infty} B_R(g_i)$. By almost repetitivity, for all $i \in \NN$, we find $x_i \in G$ such that $x_i^{-1}B_r \subseteq B_{R}(g_i)$ and $d_{B_r}\big( x_i\Lambda, \Lambda \big) = d_{x_i^{-1}B_r}\big( \Lambda, x_i^{-1}\Lambda \big) < \varepsilon$. Now to show the claim it clearly suffices to show that $T_{r,\varepsilon}:= \{x_i:\, i \in \NN\} \subseteq T_{K, \varepsilon}(\Lambda)$ is right-relatively dense in $G$. For each $b \in B_r$, we have $x_i^{-1}b \subseteq B_R(g_i)$ for all $i \in \NN$. This shows that $G = \bigcup_{i=1}^{\infty} B_{2R}(x_i^{-1}b) = \bigcup_{i=1}^{\infty} x_i^{-1} B_{2R}(b)$. For some fixed $b \in B_r$, we define 
	\[
	K^{\prime} := \overline{\big\{ g  \in G:\, g^{-1}b \in B_{2R}(b) \big\}}.
	\]
	This set is clearly compact. We will show $K^{\prime} T_{r,\varepsilon} = G$. Take $h \in G$. This gives 
	$h^{-1}b \in x_i^{-1}B_{2R}(b)$ for some $i \in \NN$. This yields
	\begin{eqnarray*}
		h &\in& \big\{ g \in G:\, g^{-1}b \in x_i^{-1}B_{2R}(b) \big\} = \big\{ g \in G:\, x_i g^{-1}b \in B_{2R}(b) \big\} \\
		&=& \big\{ y \in G:\, y^{-1}b \in B_{2R}(b) \big\} \cdot x_i \subseteq K^{\prime}  T_{r, \varepsilon}. 
		\end{eqnarray*} 
	Since $h$ was chosen arbitrarily, the proof of the claim is finished. 
\end{proof}\medskip

While almost repetitivity arises naturally in the study of minimality of Delone dynamical systems, a more classical notion is the following one:

\begin{definition} \label{defi:rep}
A function $\varrho: [1,\infty) \to [1, \infty)$ is called a {\em repetitivity function} for  $\Lambda$ if for all $R \geq 1$, every pattern of $\Lambda$ of size $B_R$ occurs in $B_{\varrho(R)}(h)$ for all $h \in G$. If such a function exists, then $\Lambda$ is called \emph{repetitive}.
\end{definition}

\begin{remark}
Similarly as for the almost repetitivity function, we have $\varrho(R)\geq R$ and we assume $R\geq 1$.
\end{remark}

We say that $\Lambda = (P, \alpha)\in \operatorname{Del}(U,K,\sigma)$ has {\em finite local complexity} (FLC) if for ever compact set $S$, the set $\{(x^{-1}.\Lambda)|_S :\, x\in P\}$ is finite.

\begin{proposition}
\label{prop:AlmRep-Rep}
Let $\Lambda\in \operatorname{Del}(U,K,\sigma)$ be of finite local complexity. If $\Lambda$ is almost repetitive then $\Lambda$ is repetitive.
\end{proposition}
\begin{proof}
Without loss of generality, we assume that $e\in U$. Since $\Lambda = (P,\alpha)$ is of finite local complexity, the set
$$
P(r):= \big\{  (x^{-1}.\Lambda)|_{B_{r}} :\, x\in P \big\}
$$
is finite for each $r>0$. In particular, if $r>0$ is such that the topological boundary $\overline{B_r}\setminus B_r$ is not intersecting $x^{-1}P$ for any $x\in P$, then there is a $\delta>0$ such that $x^{-1}P\cap \partial_{B_{\delta_1}}\big(B_{r}\big) = \emptyset$ holds for all $x\in P$ where 
$$
\partial_{B_{\delta_1}}(B_{r}) := \{g \in G:\, B_{\delta_1}g \cap B_{r} \neq \emptyset \, \wedge\, B_{\delta_1}g \cap (G \setminus B_{r}) \neq \emptyset\}
$$
is the $B_{\delta_1}$ boundary of the topological boundary of $B_r$.

Let $R>0$ be fixed. Since $\Lambda$ is $K$-relatively dense and of finite local complexity, we can choose an $R_1>0$ and a $\delta_1>0$ such that 
\begin{itemize}
\item[(a)] for each $x\in G$, there is an $y_x\in P$ satisfying $x\in B_{R_1}(y_x)$,
\item[(b)] for all $y\in P$, we have $y^{-1}P\cap \partial_{B_{\delta_1}}\big(B_{R+R_1}\big) = \emptyset$.
\end{itemize}

Consider an $R$-patch $p:=\Lambda_{B_R(x)}$ for $x\in G$. Let $y_x\in P$ be such that $x\in B_{R_1}(y_x)$. Now consider the patch
$$
p':=(y_x^{-1}.\Lambda)_{B_{R+R_1}} = y_x^{-1}. \big(\Lambda|_{B_{R+R_1}(y_x)}\big)
$$
By construction $p$ occurs in $p'$.  Let $0<\delta<\min\{\delta_1,\sigma^{-1}\}$ be such that $B_\delta\subseteq U$ and 
$$
q_1,q_2\in P(R+R_1) \text{ with } d_{B_{R+R_1}}(\delta_{(q_1,B_{R+R_1})},\delta_{(q_2,B_{R+R_1})})\leq\delta
	\quad\Longrightarrow\quad q_1=q_2.
$$ 
The latter is possible as $P(R+R_1)$ is finite.

\medskip

Let $h\in G$. Since $\Lambda$ is almost repetitive, $p'$ $\delta/2$-occurs in $B_{\varrho(\delta/2,R+R_1)}(h)$,
where $\varrho$ is an almost repetitivity function for $\Lambda$. 
Thus, there is a $g\in G$ such that $gB_{R+R_1}\subseteq B_{\varrho(\delta/2,R+R_1)}(h)$ and $d_{B_{R+R_1}}\big( \delta_{p'}, \delta_{g^{-1}.q_h} \big) <\delta/2$ where $q_h:=\Lambda_{B_{R+R_1}(g)}$. Thus,
\begin{equation*}
d_{B_{R+R_1}}\big( p', g^{-1} q_h \big)
	= \inf\left\{ 
			\varepsilon>0 :\!\!
			\begin{array}{c}
			\big|\delta_{y_g^{-1}.\Lambda}|_{B_{R+R_1}}\big(B_{\varepsilon}(y)\big) 
				- g^{-1}_{\ast}\delta_{\Lambda}|_{B_{R+R_1}(g)}\big(B_\varepsilon(y)\big)\big| <\varepsilon\\
			\mbox{ for all } y\in (y_x^{-1}P\cap B_{R+R_1})\cup (g^{-1}P\cap B_{R+R_1})
			\end{array}			
		\right\}	
	<\delta/2
\end{equation*}
holds.
By construction $\delta_{p'}$ has support on $\{e\}$, namely, $\delta_{p'}(e)>\sigma^{-1}$. Hence, there is a unique $y_g\in P$ and $g'\in B_{\delta/2}$ such that $g=y_g g'$ using that $\delta<\sigma^{-1}$, $P$ is 
$B_\delta$-uniformly discrete and $d_{B_{R+R_1}}\big( \delta_{p'}, \delta_{g^{-1}.q_h} \big)<\delta/2$. Due to condition (b), we conclude $\delta_{y_g^{-1}.\Lambda}|_{B_{R+R_1}} = \delta_{y_g^{-1}.\Lambda}|_{B_{R+R_1}(g')}$ and so
\begin{align*}
g^{-1}. q_h
	= g'^{-1} \big( \delta_{y_g^{-1}.\Lambda}|_{B_{R+R_1}(g')}, B_{R+R_1}(g')\big)
	= g'^{-1} \big( \delta_{y_g^{-1}.\Lambda}|_{B_{R+R_1}}, B_{R+R_1}(g') \big).
\end{align*}
By construction, $(y_g^{-1}.\Lambda)_{B_{R+R_1}}$ occurs in $B_{\varrho(\delta/2,R+R_1)+\delta/2}(h)$ since $g'\in B_{\delta/2}$.
Invoking that $\delta<\sigma^{-1}$, $P$ is $B_\delta$-uniformly discrete, $g'\in B_{\delta/2}$ and $d_{B_{R+R_1}}\big(\delta_{p'}, \delta_{g^{-1}.q_h} \big)<\delta/2$, we derive
$$
d_{B_{R+R_1}}\big(\delta_{p'}, \delta_{(y_g^{-1}.\Lambda)_{B_{R+R_1}}}\big)\leq\delta.
$$
Thus, $(y_g^{-1}.\Lambda)_{B_{R+R_1}}=p'$ follows by the choice of $\delta$ and since $(y_g^{-1}.\Lambda)|_{B_{R+R_1}},p'\in P(R+R_1)$. Since $p$ occurs in $p'$ and $(y_g^{-1}.\Lambda)_{B_{R+R_1}}$ occurs in $B_{\varrho(\delta/2,R+R_1)+\delta/2}(h)$, we conclude that $p$ occurs in $B_{\varrho(\delta/2,R+R_1)+\delta/2}(h)$. 

\medskip

Hence, the map $\varrho':(0,\infty)\to(0,\infty)$ defined by
$\varrho'(R):=\varrho(\delta/2,R+R_1)+\delta/2$ defines a repetitivity function. Note here that in the previous considerations $\delta$ and $R_1$ depend on $R$ but not on $h$.
\end{proof}

\begin{remark}
Note that if $\Lambda= (P,\alpha)$ is repetitive, then $\Lambda$ necessarily has FLC since $B_{\varrho(R)}(h)\cap P$ is uniformly bounded in $h\in G$. Furthermore, every repetitive $\Lambda$ is almost repetitive, which can be seen by setting $\varrho(\delta,R):=\varrho(R)$ for $R \geq 1$. In other words, for Delone sets we have the equivalence
\[
\Lambda \text{ is repetitive} \iff (\Lambda \text{ is almost repetitive}) \quad \text{and} \quad (\Lambda \text{ has FLC}).
\]
\end{remark}

\section{An abstract ergodic theorem} \label{ergodictheorem}

\subsection{Statement of the theorem}
Throughout this section $G$ will denote an amenable unimodular lcsc group. We fix a Haar measure $m_G$ on $G$. The goal of this section is to explain and prove the following theorem.
\begin{theorem}[Uniform sub-additive convergence theorem] \label{thm:abstr}
Assume that
\begin{itemize}
\item $\Lambda$ is a weighted Delone set in $G$ with hull $\mathcal{H}_\Lambda$;
\item $(T_m)$ is a strong F\o lner exhaustion sequence in $G$;
\item $\Lambda$ is almost tempered repetitive with respect to $(T_m)$;
\item $w$ is an admissible weight function on $G$ over $\mathcal{H}_\Lambda$.
\end{itemize}
Then there exists $I_w \in \RR$ (depending on $\Lambda$ and $w$, but independent of $(T_m)$) such that
\[
	 \lim_{m \to \infty} \sup_{\Pi \in \mathcal{H}_{\Lambda}} \Bigg| \frac{w(T_m, \Pi)}{m_G(T_m)} - I_w \Bigg| = 0.
\]
\end{theorem}
Theorem~\ref{thm:abstr} will be proved in Subsection~\ref{SecErgodicProof} after introducing all the required terminology. 
Tempered repetitivity with respect to a strong F\o lner exhaustion sequence will be defined in Subsection~\ref{sec2:amenability}, and admissible weight functions will be discussed in Subsection~\ref{SecWeight}.

\subsection{Strong F{\o}lner exhaustion sequences and tempered repetitivity} \label{sec2:amenability}
We recall that amenability of $G$ is equivalent to the existence of a F\o lner sequence, i.e.\ a sequence $(T_m)$ of compact subsets of $G$ of positive Haar measure such that for all compact subsets $K \subset G$,
\[
	\lim_{m\to\infty} \frac{m_G(T_m\triangle KT_m)}{m_G(T_m)}=0.
\] 
In fact, one can choose a F\o lner sequence with additional properties. Given relatively compact subsets $L, S \subset G$ we denote by
\[
\partial_L(S) := \{g \in G:\, Lg \cap S \neq \emptyset \, \wedge\, Lg \cap (G \setminus S) \neq \emptyset\}.
\]
the {\em $L$-boundary of $S$}, as defined in \cite{OW87}.  
\begin{definition}\label{DefStrongFES} A F\o lner sequence $(T_m)$ is called a \emph{strong F\o lner exhaustion sequence} if 
\begin{enumerate}[(i)]
\item $(T_m)$ is a \emph{strong F\o lner sequence} in the sense of \cite{PS16}, i.e.\ for all compact subsets $K \subset G$,
	\[
	\lim_{m \to \infty} \frac{m_G(\partial_{K}(T_m))}{m_G(T_m)} = 0.
	\]
\item $(T_m)$ is a \emph{strong exhaustion sequence}, i.e.\ $\{e\} \subset T_m \subseteq \mathring{T}_{m+1}$ and $\bigcup_{m} T_m = G$. 
\end{enumerate}
\end{definition}

\begin{remark}
	We point out that by \cite[Proposition~5.7]{PRS21}, the notion of a strong F{\o}lner sequence is equivalent to the concept of a {\em van Hove sequence}. The latter has been used prominently in the literature on aperiodic order in locally compact abelian groups. 
\end{remark}

\begin{proposition} \label{prop:strongfolnerexist} Every amenable unimodular group $G$ admits a strong F{\o}lner exhaustion sequence. 
\end{proposition}

\begin{proof} 
	Building on the uniform F{\o}lner condition proven in \cite{OW87}, one observes that $G$ admits a strong F{\o}lner sequence $(S_n)$, cf.\@ 
	\cite[Lemma~2.8]{PogorzelskiThesis14}. Moreover, the sequence $(S_n)$ can be turned into a strong F\o lner exhaustion sequence by the usual construction (cf. \cite{Gre73}): By $\sigma$-compactness of $G$ we can now choose an exhaustion of $G$ by compact subsets $(K_n)$ containing the identity. Set $T_1 := S_1$ and choose $m_1$ such that $T_1 \subseteq \mathring{K}_{m_1}$. Then choose $\ell \in \NN$ large enough such that $m_G(\partial_{K_{m_1}}(S_{\ell})) \leq \frac 1 2m_G(S_{\ell})$. Thus, there is an $h \in G$ such that $K_{m_1}h \subseteq S_{\ell}$. If we now set $T_2:= S_{\ell}h^{-1}$, then $T_1 \subseteq \mathring{K}_{m_1} \subseteq T_2$. One now proceeds by induction to construct the desired strong F\o lner exhaustion sequence $(T_m)$. 
\end{proof}

\medskip
From now on $\mathcal T = (T_m)$ denotes a strong F\o lner exhaustion sequence in $G$ and $\Lambda$ denotes a weighted Delone set in $G$. 

\begin{definition}\label{DefATR} The {\em repetitivity index for $\Lambda$ with respect to $\mathcal{T}$} is the function $  \mathcal{R}_{\Lambda}^{\mathcal{T}}: [0,1) \times \NN \to \NN \cup \{ +\infty\}$ given by
\[  \mathcal{R}_{\Lambda}^{\mathcal{T}}(\delta, m) := \inf\big\{ n \in \NN:\, \mbox{every } \mbox{pattern of } \Lambda \text{ of size } \mathring{T}_m^{-1}\,\delta\mbox{-occurs in } h\mathring{T}_n^{-1} \mbox{ for all } h \in G \big\},
 \]
 where, by convention, $\inf \emptyset = \infty$. We then define the {\em repetitivity portion of $\Lambda$ with respect to $\mathcal{T}$} as 
	\[
	\zeta:[0,1) \to [0,1],\,\,  \zeta(\delta):= \inf_{m \in \NN} \frac{m_G(T_m)}{m_G(T_{\mathcal{R}^{\mathcal{T}}_{\Lambda}(\delta, m)})},
	\]
where we use the convention that $\zeta(\delta) = 0$ if there is some $m \in \NN$ such that  $\mathcal{R}^{\mathcal{T}}_{\Lambda}(\delta, m) = \infty$. We say that $\Lambda$ is {\em almost tempered repetitive with respect to $\mathcal{T}$} if $\zeta(\delta) > 0$ for all $\delta > 0$. We say that $\Lambda$ is {\em tempered repetitive with respect to $\mathcal{T}$} if $\zeta(0) > 0$. 
\end{definition} 

\begin{remark} Some comments on the definition are in order:
\begin{enumerate}[(i)]
\item Since $(T_m)$ displays asymptotic invariance from the left, we use patterns arising from the {\em inverse} sequence $(T_m^{-1})$ in the above definitions. In certain situations of interest, one can choose the $T_m$ to be symmetric, for instance if they arise as  certain closed balls with respect to a suitable adapted metric, cf.\@ Section~\ref{sec:metrics}.

\item	Almost tempered repetitivity is the condition on Delone sets which will enable the proof of our ergodic theorem. The terminology ``tempered'' is chosen because in the literature it refers to a growth condition of F{\o}lner sequences used in the proofs of ergodic theorems.
\item Since the repetitivity index is monotonically decreasing in $\delta$, tempered repetitivity implies almost tempered repetitivity.
\item If $\Lambda$ is tempered repetitive, then $\mathcal{R}_{\Lambda}^{\mathcal{T}}(0,m) < \infty$ for all $m \in \NN$, and if $\Lambda$ is almost tempered repetitive, then  $\mathcal{R}_{\Lambda}^{\mathcal{T}}(\delta,m) < \infty$ for all $m \in \NN$ and $\delta>0$. This can be used to establish minimality of the associated Delone dynamical system.
\end{enumerate}
\end{remark}

\begin{proposition}[Minimality] \label{prop:minimalitytemp}
	A weighted Delone set $\Lambda$ is almost repetitive if and only if $\mathcal{R}_{\Lambda}^{\mathcal{T}}(\delta, m) < \infty$ for all $\delta > 0$ and all $m \in \NN$. In particular, every almost tempered repetitive Delone set is almost repetitive, and hence the associated Delone dynamical system is minimal.  
\end{proposition}

\begin{proof} We note that by the properties of a strong exhaustion sequence, we find for all $r > 0$ some $\ell(r) \in \NN$ such that $\overline{B}_r \subseteq \mathring{T}_{\ell(r)}$ and for each $s \in \NN$, we find some $L(s) \in \NN$ such that $T_s \subseteq B_{L(s)}$. Moreover, by the left-invariance of the metric $d_G$, we have $B_l = B_l^{-1}$ and $\overline{B}_l = \overline{B}_l^{-1}$ for all $l > 0$. Thus the condition $\mathcal{R}_{\Lambda}^{\mathcal{T}}(\delta, l) < \infty$ for all $\delta > 0$ and all $l \in \NN$ implies for a given $R \geq 1$, that every $B_R$-pattern $\delta$-occurs in each ball of radius $L\big( \mathcal{R}_{\Lambda}^{\mathcal{T}}(\delta, \ell(R)) \big)$. Conversely, for $m \in \NN$, we have that every $\mathring{T}_m^{-1}$-pattern must $\delta$-occur in every ball of radius $\varrho(\delta, L(m))$ (with $\varrho$ denoting an almost repetitivity function), whence also in each 
$h\mathring{T}_M^{-1}$ with $M = \ell\big( \varrho(\delta, L(m)) \big)$. 
\end{proof}

The following example shows that (almost) tempered repetitivity depends crucially on the underlying strong F\o lner exhaustion sequence.

\begin{example} 
\label{exa:crazygrowth}
	 Consider the Delone set $D = \ZZ$ (seen as weighted Delone set with constant weight $1$) in $G = \RR$. Then $D$ is tempered repetitive with respect to the sequence $\mathcal{T} = (\overline{B}_m)_m$, where $\overline{B}_l$ is the closed ball around $0$ with radius $l$ with respect to the Euclidean metric. Now consider the sequence $\mathcal{T}^{\prime} = (\overline{B}_{r_m})_m$ with $r_1 := 1$ and $r_{m+1} := 2^{r_m}$ for $m \geq 1$. Note that an open ball with integer radius $M \in \NN$ and integer center contains $2M-1$ points of $\ZZ$ while balls of the same size around a non-integer center contain $2M$ points of $\ZZ$.  This implies   $\mathcal{R}_D^{\mathcal{T}^{'}}(M) \geq M+1$ for all $M \in \NN$. However, we have $\lim_m \frac{2r_m - 1}{2r_{m+1} - 1} = 0$. Hence, $D$ is not (almost) tempered repetitive with respect to $\mathcal{T}^{\prime}$.
\end{example}

\subsection{Weight functions and \texorpdfstring{$\varepsilon$-}{epsilon-}quasi-tilings}
\label{SecWeight}

For the purpose of the following definition we denote by $ \mathcal{RK}(G)$ the set of all relatively compact subsets of $G$. We also denote by $X$ a compact metrizable space on which $G$ acts jointly continuously; in our applications $G \curvearrowright X$ will always be a Delone dynamical system. 

\begin{definition} 
\label{defi:weightfunction}
A function $w: \mathcal{RK}(G) \times X \to \RR$ is called an {\em almost sub-additive weight function} over $X$ if
\begin{itemize}
		\item[(W1)] $w(\emptyset, x) = 0$ for all $x \in X$ {\em (normalization)}.
			\item[(W2)] There is a compact $J \subseteq G$ and an $\eta > 0$
		such that 
		\[
		\big| w(L,x) - w(K,x)  \big| \leq \eta\cdot \big( m_G(L \setminus K) + m_G(\partial_J(L)) + m_G(\partial_J(K)) \big) 
		\] 
		for all $x \in X$ and all $K, L \in \mathcal{RK}(G)$ with $K \subseteq L$ 
		{\em (almost-monotonicity)}.
		\item[(W3)] There is a  compact subset $B \subset G$ and a $\theta \geq 0 $ such that for any finite collection of pairwise disjoint $K_i \in\mathcal{RK}(G)$ we have
		\[
		  w\Big( \bigsqcup_{i} K_i,\,x \Big)  \leq \sum_{i} w(K_i, x)   \, + \, \theta \cdot \sum_{i} m_G(\partial_B(K_i))		
		\]		
		for all $x \in X$ {\em (almost sub-additivity)}. 
		\item[(W4)] There is a compact subset $I \subset G$ and a $\vartheta \geq 0$ such that 
		for each $x \in X$, $K\in\mathcal{RK}(G)$ and $h\in G$, we have 
		\[
		\big| w(K,x) - w(Kh^{-1},hx) \big| \leq \vartheta\cdot m_G(\partial_I(K)) \quad\quad  \operatorname{(almost-equivariance)}.		
		\]		
	\end{itemize} 
\end{definition}

\begin{remark} \label{rem:boundedness}
If $w: \mathcal{RK}(G) \times X \to \RR$ is any function satisfying (W1) and (W2), then for every relatively compact subset $L \subset G$ and every $x \in X$ we have the boundedness property 
\[|w(L, x)| \leq \eta \cdot \big( m_G(L) + m_G(\partial_J(L)) \big)\]
\end{remark}

From this inequality we deduce that if $w: \mathcal{RK}(G) \times X \to \RR$ is an almost sub-additive weight function and $\mathcal{RK}(G)_+$ denotes the set of relatively compact sets of positive Haar measure, then we can define functions $w^{+}, w^{-}: \mathcal{RK}(G)_+ \times X \to \RR$ by
\begin{equation}\label{w+-}
w^{+}(S, x) := \sup_{g \in G} \frac{w(Sg, x)}{m_G(S)}, \quad w^{-}(S,x) := \inf_{g \in G} \frac{w(Sg, x)}{m_G(S)}.
\end{equation}

We will study these functions for a particular class of almost sub-additive weight functions over a Delone dynamical system.

\begin{definition} 
\label{defi:admissibleweight}
Let $\mathcal H_\Lambda$ be the hull of a weighted Delone set $\Lambda$ in $G$. Then an almost sub-additive weight function over $\mathcal H_\Lambda$ is called an \emph{admissible weight function} provided 
	\begin{itemize}
	\item[(W5)]  If $(T_m)$ is any  strong F\o lner exhaustion sequence 
		then for all $\varepsilon>0$, there is a $\delta > 0$ and $m_0\in\NN$ such that for all $m\geq m_0$ 
		\[
		d_{T_m^{-1}}(\Pi, \Phi) \leq \delta \quad\Longrightarrow\quad  \big|w(T_m, \Phi) -w(T_m, \Pi)\big|\leq \varepsilon \cdot m_G(T_m),
		\]		
		where $d_{T_m^{-1}}(\Pi, \Phi)$ is defined as in \eqref{eqn:formuladist}.
		\end{itemize}	
\end{definition}
\begin{remark}
\label{rem:W5-}
Let $(X,d)$ be a compact metric space on which $G$ acts continuously. Then we can study almost sub-additive weight functions that additionally satisfy the following.
\begin{itemize}
		\item[(W5$^\ast$)] If $(T_m)_{m\in\NN}$ is a  strong F\o lner exhaustion sequence and $d$ is a metric generating the topology of $X$, then
		then for all $\varepsilon>0$, there is an $m_0\in\NN$ such that for all $m\geq m_0$ one can find $\delta_m > 0$ such that 
		\[
		d(x,y) \leq \delta_m \quad\Longrightarrow\quad  \big|w(T_m, x) -w(T_m, y)\big|\leq \varepsilon m_G(T_m).
		\]		
	\end{itemize}
This condition has the advantage that it can be defined for general (i.e.\ not necessarily Delone) dynamical systems. While it is sufficient for the proof of the Ornstein-Weiss type lemma (Lemma~\ref{lemma:mainaux}), it is not sufficient for the proof of our main theorem. 

Let us sketch an argument how Axiom (W5) implies Axiom (W5$^\ast$) for $X:= \mathcal H_\Lambda$: Firstly, we can define a metric $d_{\ast}$ on $\mathcal{H}_\Lambda$ by
\[
d_{\ast}(\Phi,\Pi)
	:= \min\big\{ c_\ast,\, \inf\{\delta>0 :\, 
		\big|\delta_\Phi(B_\delta(y))-\delta_\Pi(B_\delta(y))\big|<\delta 
		\mbox{ for all } y\in \big(B_{1/\delta}\cap P\big)\cup \big(B_{1/\delta}\cap Q\big)		
		\}\big\}
\]
where $\Phi:=(P,\alpha),\Pi :=(Q,\beta)\in \mathcal{H}_\Lambda$ and $c_\ast>0$ is chosen in such a way that $B_{c_\ast}\subseteq U$ and $c_\ast<(2\sigma)^{-1}$. Now the metric $d_{\ast}$ generates the topology of $\mathcal H_\Lambda$ by Proposition~\ref{prop:neighborhood} in Appendix~\ref{AppendixB}, and Axiom (W5) implies (W5$^*$) for this specific choice of metric. On the other hand, a straightforward computation shows that if (W5$^\ast$) holds for some metric $d$ on $\mathcal{H}_\Lambda$ defining the topology, then it holds for any such metric. We thus conclude that indeed (W5) implies (W5$^\ast$) for $X:= \mathcal H_\Lambda$.
\end{remark}\medskip

At this point we have defined all the terms in the statement of  Theorem~\ref{thm:abstr}. We now turn to its proof.

\subsection{\texorpdfstring{$\varepsilon$-}{epsilon-}quasi-tilings and the proof of the abstract ergodic theorem}
\label{SecErgodicProof} 

Our proof of Theorem~\ref{thm:abstr} will follow the same strategy as the proofs in \cite{LaPl03, FrRi14}. The main difference is that we have to replace 
box tilings in Euclidean space with the $\varepsilon$-quasi tile machinery which was developed by Ornstein and Weiss \cite{OW87}. We use the formulation as in \cite{PS16} and recall the necessary definition. 
\begin{definition}\label{defi:qt}
Let $A \subset G$ be a relatively compact subset and $\varepsilon > 0$. Set $N(\varepsilon):=\big\lceil -\log(\frac{\varepsilon}{1-\varepsilon})\big\rceil$ and let $S_1, \dots, S_{N(\varepsilon)}$ be compact subsets of $G$ and let $C_1^A, \dots, C_{N(\varepsilon)}^A$ be finite subsets.

We say that $A$ is {\em $\varepsilon$-quasi tiled by the prototiles $S_i^{\varepsilon}$ with center sets $C_i^A$} if the following conditions are satisfied: 
\begin{itemize}
		\item[(T1)] $S_i^\varepsilon C_i^A \subseteq A$ for $1\leq i\leq N(\varepsilon)$;
		\item[(T2)] For each $1\leq i\leq N(\varepsilon)$ and every $c\in C_i^A$ there is a measurable set $\tS_i^\varepsilon(c)\subseteq S_i^\varepsilon c$ satisfying
		\begin{itemize}
			\item[$\bullet$] $(1-\varepsilon)m_G(S_i^\varepsilon c) \leq m_G(\tS_i^\varepsilon(c)) \leq m_G(S_i^\varepsilon c)$,
			\item[$\bullet$] $\bigcup_{c\in C_i^A}S_i^\varepsilon c = \bigsqcup_{c\in C_i^A} \tS_i^\varepsilon (c)$, where the latter union consists of pairwise disjoint sets;
		\end{itemize}
		\item[(T3)] $S_i^\varepsilon C_i^A \cap S_j^\varepsilon C_j^A =\emptyset$ for $1\leq i < j\leq N(\varepsilon)$;
		\item[(T4)] $m_G\left(\bigsqcup_{i=1}^{N(\varepsilon)} S_i^\varepsilon C_i^A\right) \geq (1-2\varepsilon) m_G(A)$. 
	\end{itemize}
\end{definition}
The following theorem guarantees the existence of certain $\varepsilon$-quasi-tilings. It is a weaker statement than Theorem~4.4~(a) in  \cite{PS16}
	and a slightly stronger version of Theorem~6 in \cite{OW87}. Given relatively compact subsets $S, T \subset G$ and $\delta >0$ we will say that $T$ is \emph{$(S,\delta)$-invariant} if
	\[
	{m_G(\partial_S(T))} < \delta \cdot m_G(T).
	\]
\begin{theorem}[Existence of $\varepsilon$-quasi tilings]
	\label{thm:tiling}
	 Let $0<\varepsilon<1/10$ and $N(\varepsilon):=\big\lceil -\log(\frac{\varepsilon}{1-\varepsilon})\big\rceil$. Then for every amenable unimodular lcsc group $G$, strong F\o lner exhaustion sequence $(S_l)_{l\in\NN}$ and natural number $n$ there exist sets $S_1^\varepsilon, \dots,  S_{N(\varepsilon)}^\varepsilon$ with the following properties:
\begin{enumerate}[(i)]
\item $S_i^\varepsilon\in\big\{S_l \,:\, l\geq \max\{i,n\}\big\}$;
\item $S_n\subseteq S_1^\varepsilon\subseteq\ldots\subseteq S_{N(\varepsilon)}^\varepsilon$; 
\item  there is some $\delta_0 > 0$ such that for every $\big( S^{\varepsilon}_{N(\varepsilon)} S^{\varepsilon\, -1}_{N(\varepsilon)}, \delta_0 \big)$-invariant compact set $A \subseteq G$, there are finite sets $C_i^A \subseteq A$ such that $A$ is $\varepsilon$-quasi tiled by the prototiles $S_i^{\varepsilon}$ with center sets $C_i^A$.
\end{enumerate}
\end{theorem}

We give some remarks on $\varepsilon$-quasi tilings.
\begin{remark}	\label{rem:tiling} 
	\begin{enumerate}[(i)]
		\item  According to (T3), the set $A$ can be $\varepsilon$-quasi tiled by $S_i^\varepsilon\,,\, 1\leq i\leq N(\varepsilon)\,$ such that tiles of different type do not overlap. On the other hand, tiles of the same type might overlap, see (T2). However, (T2) asserts that these sets are $\varepsilon$-disjoint which means one can remove from them portions of measure at most $\varepsilon$ to obtain disjoint sets. 
		
		\item In fact, these  trimmed  sets can be proven to maintain certain invariance conditions, cf.\@ part (b) of Theorem~4.4 in \cite{PS16}. 
		More precisely, for a given compact $B \subseteq G$ and $0 < \zeta < \varepsilon$ we can make sure that $m_G(\partial_B(S_i^{\varepsilon})) / m_G(S_i^{\varepsilon}) \leq \zeta^2$ and $m_G(\partial_B(\tilde{S}_i^{\varepsilon}(c))) / m_G(\tilde{S}_i^{\varepsilon}(c)) \leq 4\zeta$ 
		for all $1 \leq i \leq N(\varepsilon)$	and all $c \in C_i^m$. We will need this latter fact only once in this paper, namely for the proof of Lemma~\ref{lemma:tilingswf} which is given in the appendix.	
		\item Given any strong F{\o}lner exhaustion sequence $(A_m)$, the theorem implies that for a given $\varepsilon$ one finds sets  $S_i^{\varepsilon}$ ($1 \leq i \leq N(\varepsilon)$) extracted from a strong F{\o}lner exhaustion sequence (which possibly is $(A_m)$ itself), along with $M \in \NN$, such that for all $m \geq M$, one finds finite sets $C_i^m \subseteq A_m$ such that the conditions (T1)-(T4) hold true for $A= A_m$ and $C_i^A = C_i^m$.
		\item The assertion (T4) given above is weaker than condition (iv) in \cite[Definition~4.1]{PS16}, which provides precise quantitive information
		on the portion of $A$ covered by the sets $S_i^{\varepsilon} C_i^A$ for a fixed $1 \leq i \leq N(\varepsilon)$. 
		As shown in Remark~4.3 of \cite{PS16}, these estimates  result in~(T4).
	\end{enumerate}
\end{remark}

The proof of Theorem~\ref{thm:abstr} rests on two lemmas whose proofs build on known techniques used in the context of sub-additive and almost-additive Ornstein-Weiss type lemmas, see e.g.\@ \cite{OW87, Gro99, Kri07, CCK14}. For the sake of self-containment we give their proofs in Appendix~\ref{AppendixA}. The first lemma shows that weight functions are compatible with respect to $\varepsilon$-quasi tilings of
F{\o}lner sets.
\begin{lemma} \label{lemma:tilingswf}
	Let $G$ be an amenable unimodular lcsc group. Let $0 < \varepsilon < 1/10$ and a strong F{\o}lner exhaustion sequence $(S_n)$ be given.
	Suppose further that $v:\mathcal{RK}(G) \to \RR$
	is a function for which there are  $\eta(v), \theta(v) \geq 0$  as well as compact subsets $J, B \subset G$ such that 
	\begin{itemize}
		\item[(w1)]  $v(\emptyset) = 0$, 
			\item[(w2)] 	$\big| v(L) - v(K) \big| \leq \eta(v)\, \big( m_G(L \setminus K) + m_G(\partial_J(L)) + m_G(\partial_J(K)) \big) 	$
		for all $K, L \in \mathcal{RK}(G)$ with $K \subseteq L$,
		\item[(w3)] $ v(\sqcup_i K_i) \leq  \sum_i v(K_i) \, + \, \theta(v) \sum_i m_G(\partial_B(K_i)) $ for finitely many pairwise disjoint sets $K_i \in \mathcal{RK}(G)$.
	\end{itemize}
	Then for every compact subset $I  \subset G$ there is some $m_{I} \in \NN$ such that for every $n \geq m_{I}$ one finds $(I, \varepsilon)$-invariant sets $S_i^{\varepsilon} \in \{S_l:\, l \geq \max\{n,i\}\}$  ($1 \leq i \leq N(\varepsilon) = \lceil \frac{-\varepsilon}{\log(1-\varepsilon)} \rceil$)
	as well as $\delta_0 > 0$ such that each compact subset $A \subseteq G$ which is $\big( S_{N(\varepsilon)}^{\varepsilon}S_{N(\varepsilon)}^{\varepsilon\, -1}, \delta_0 \big)$-invariant and at the same time 
	$(L, \varepsilon)$-invariant for $L \in \{J,B,I\}$
	  can be $\varepsilon$-quasi tiled by prototiles $S_i^{\varepsilon}$ and finite center sets $C_i^A$
	for $1 \leq i \leq N(\varepsilon)$ such that in addition, one has
	\begin{eqnarray*}
		\frac{v(A)}{m_G(A)}  \leq \frac{1}{m_G(A)} \sum_{i=1}^{N(\varepsilon)} \sum_{c \in C_i^A} v\big( S_i^{\varepsilon}c \big) \, + \, \big( 8 \eta(v) +  2\theta(v)  \big) \cdot \varepsilon.
	\end{eqnarray*}
\end{lemma}

\medskip

The second lemma is about convergence of the functions $w^+$ from \eqref{w+-}.

\begin{lemma} \label{lemma:mainaux}
	Let $w: \mathcal{RK}(G) \times X \to \RR$ be an almost sub-additive weight function and $(T_m)$ be
	a strong F{\o}lner exhaustion sequence in $G$. 
	Then the limits
	\[
	I^+_w(x) = \lim_{m \to \infty} {w^+}(T_m, x) 
	\]
	exist for all $x \in X$. Moreover, the limits do not depend on the sequence $(T_m)$.
	If in addition, $w$ satisfies condition (W5$^\ast$) and if the action $G \curvearrowright X$ is minimal, then there is a value $I^+_w \in \RR$ such that 
	\[
	\lim_{m \to \infty} \sup_{x \in X} \big| {w^+}(T_m, x)  -  I_w^{+} \big| = 0.	
	\]
\end{lemma}

\begin{remark} \label{rem:mainaux}
If $-w$ instead of $w$ is a sub-additive weight function (in particular, if $w$ is actually additive), then the corresponding statements hold for $w^{-}$ instead of $w^+$.\end{remark}

\begin{proof}[Proof of Theorem~\ref{thm:abstr}]
Let $w: \mathcal{RK}(G) \times \mathcal{H}_{\Lambda} \to \RR$ be an admissible weight function and define $w^+$ and $w^-$ by \eqref{w+-}. \medskip

By Proposition~\ref{prop:minimalitytemp} the dynamical system $G \curvearrowright X$ is minimal. On the other hand, the function $w$ satisfies Axiom (W5) and therefore also Axiom (W5$^\ast$) by Remark~\ref{rem:W5-}, and hence the strong form of Lemma~\ref{lemma:mainaux} applies. We conclude that the limit $\nu^{+} = \lim_{n \to \infty} w^{+}(T_m,\Pi)$ exists uniformly over $\Pi$ and is independent of the choices of
$\Pi \in \mathcal{H}_{\Lambda}$ and the strong F\o lner exhaustion sequence.

\medskip

Moreover, we define $\nu^{-}(\Pi) := \liminf_{n \to \infty} w^{-}(T_n, \Pi)$
for $\Pi \in \mathcal{H}_{\Lambda}$. It follows from Lemma~\ref{lemma:minimality} in Appendix~\ref{AppendixA} that $\nu^{-} = \nu^{-}(\Pi)$
is independent of the choice of $\Pi \in \mathcal{H}_{\Lambda}$ and that the limit inferior is realized by one subsequence $(n_l)$
of the integers simultaneously for all $\Pi$. In light of that, it suffices to prove that 
\[
\nu^{+} = \lim_{n \to \infty} w^{+}(T_m,\Lambda) \leq \liminf_{n \to \infty} w^{-}(T_m, \Lambda) = \nu^{-}.
\]
To this end, we fix a subsequence $(m_l)$ of the integers such that
$\lim_{l \to \infty} w^{-}(T_{m_l}, \Lambda) = \nu^{-}$. For the sake of 
simpler notation, we set $T_l^{\prime} := T_{m_l}$ for all $l \in \NN$. We assume by contradiction that $\nu^{-} < \nu^{+}$.
Let $0 < \varsigma < \frac{\nu^{+} - \nu^{-}}{4}$ and choose $L_{\varsigma} \in \NN$ such that $|\nu^{-} - w^{-}(T^{\prime}_l,\, \Lambda)| < \varsigma$
for all $l \geq L_{\varsigma}$. Then for $l \geq L_{\varsigma}$, one finds $g_l \in G$ such that 
$\frac{w(T_l^{\prime} g_l, \Lambda)}{m_G(T_l^{\prime})} \leq w^{-}(T_l^{\prime},\Lambda) + \varsigma$. Combined
with the previous estimate and with the choice of $\varsigma$, we conclude from $4 \varsigma < \nu^{+} - \nu^{-}$ that
\begin{eqnarray} \label{eqn:startproof}
\frac{w\big(T_l^{\prime}g_l,\, \Lambda \big)}{m_G(T_l^{\prime})} \leq \nu^{-} + 2\varsigma  \leq \nu^{+} - 2 \varsigma
\end{eqnarray}
for all $l \geq L_{\varsigma}$. Writing $\Lambda = (P,\alpha)$ we 
define the $S^{-1}_l$-patches $q_l:= \Lambda_{S_l^{-1}}$
with inverse support $S_l := T_l^{\prime} g_l = T_{m_l} g_l$. 
Using Axiom~(W5) we find some $0< \delta < 1$ and an $m_0 \in \NN$ such that for all
$l \geq m_0$, the condition $d_{S^{-1}_{l}}(\Pi, \Phi) \leq \delta$ implies $|w(S_{l}, \Pi) - w(S_{l},\Phi)| < \varsigma m_G(S_{l})$
for $\Pi, \Phi \in \mathcal{H}_{\Lambda}$. 
We fix this $\delta$ and take an arbitrary $h \in G$. Since $\Lambda$ is almost tempered repetitive with respect to $\mathcal{T} = (T_m)$, 
the pattern arising from the patches $q_l$ must $\delta$-occur in $h^{-1}T^{-1}_{\mathcal{R}^{\mathcal{T}}_{\Lambda}(\delta, m_l)}$, where $\mathcal{R}^{\mathcal{T}}_{\Lambda}$
is the repetitivity index for $\Lambda$ with respect to $\mathcal{T}$. This means that there must be some element $\widetilde{g}_{l,h} \in G$ such that
$S_l \widetilde{g}_{l,h} \subseteq T_{\mathcal{R}^{\mathcal{T}}_{\Lambda}(\delta, m_l)}h$ and $d_{S_l^{-1}}\big( \Lambda,\, \widetilde{g}_{l,h}.\Lambda \big) \leq \delta$. 
This yields
\[
\big| w(S_l, \Lambda) - w(S_l,\, \widetilde{g}_{l,h} \Lambda) \big| \leq \varsigma m_G(S_l)
\]
for $l \geq m_0$.
Furthermore, increasing $m_0$ if necessary and invoking Axiom~(W4) we find that 
\begin{eqnarray} \label{eqn:mainaux1}
\big| w(S_l, \Lambda) - w(S_l\widetilde{g}_{l,h},\,  \Lambda) \big| \leq \frac{11}{10} \varsigma m_G(S_l)
\end{eqnarray}
for $l \geq m_0$. We point out that $m_0$ only depends on parameters given by the weight function $w$
and on $\varsigma$ but not on objects constructed in the proof. Given $h$ as above, 
define $A_{l,h}:= T_{\mathcal{R}^{\mathcal{T}}_{\Lambda}(\delta, m_l)}h \setminus S_l \widetilde{g}_{l,h} $. We will distinguish the following two cases. 
\begin{enumerate}[(A)]
	\item There is some $0 < \kappa_{\Lambda}(\delta) < 1$ such that $\frac{m_G(T_{m_l})}{m_G(T_{\mathcal{R}^{\mathcal{T}}_{\Lambda}(\delta, m_l)})} \leq \kappa_{\Lambda}(\delta)$ for all $l \in \NN$.
	\item There is some subsequence $(T_{m_{l_k}})_k$ such that $\lim_{k \to \infty} \frac{m_G(T_{m_{l_k}})}{m_G(T_{\mathcal{R}^{\mathcal{T}}_{\Lambda}(\delta, m_{l_k})})} = 1$.  
\end{enumerate}
We will cover both of these cases separately and in both cases we will obtain
$\nu^{+} \leq \nu^{+} - c(\varsigma)$ for some $c(\varsigma) > 0$ which is clearly a contradiction. We start with case~(A). 

\medskip

Case~(A). Invoking the general relations 
\[
\partial_L(C \setminus D) \subseteq \partial_L(C) \cup \partial_L(D), \quad \partial_L(C)g = \partial_L(Cg)
\]
for sets $C,D,L \in \mathcal{RK}(G)$ and $g \in G$, we find 
\begin{align} \label{eqn:Al}
\frac{m_G(\partial_K(A_{l,h}))}{m_G(A_{l,h})}
\quad \leq \quad &\frac{m_G(\partial_K(T_{\mathcal{R}^{\mathcal{T}}_{\Lambda}(\delta, m_l)}h) + m_G(\partial_K(S_l \tg_{l,h}))}{m_G(T_{\mathcal{R}^{\mathcal{T}}_{\Lambda}(\delta, m_l)})\, \big(1-\frac{m_G(T_{m_l})}{m_G(T_{\mathcal{R}^{\mathcal{T}}_{\Lambda}(\delta, m_l)})}\big)} \nonumber \\
\quad \leq \quad &\frac{1}{1-\kappa_{\Lambda}(\delta)}\left(\frac{m_G(\partial_K(T_{\mathcal{R}^{\mathcal{T}}_{\Lambda}(\delta, m_l)}))}{m_G(T_{\mathcal{R}^{\mathcal{T}}_{\Lambda}(\delta, m_l)})} + \frac{m_G(\partial_K(T'_l))}{m_G(T'_l)}\right)
\end{align}
for all relatively compact sets $K \subseteq G$.
This shows that the left-hand side of the above inequality tends to $0$  uniformly in $h$ if $l$ tends to infinity. In particular, the sequences $(A_{l,h})_l$ are strong F{\o}lner sequences for all $h \in G$. 
Now let $0 < \varepsilon < 1/10$ and choose some $n_{\varepsilon} \geq \max\{m_0, L_{\varsigma}\}$ such that 
\[
\big| \nu^{+} - w^{+}(T_n, \Lambda) \big| < \varepsilon
\] 
for all $n \geq n_{\varepsilon}$. We also make sure that $n_{\varepsilon}$ is chosen large enough such that
\begin{eqnarray} \label{eqn:inv11}
\frac{m_G(\partial_L(T_n))}{m_G(T_n)} < \varepsilon 
  \quad \mbox{for all } n \geq n_{\varepsilon},
\end{eqnarray}
where $L \in \{ J,B,I\}$ and  $J,B,I$ are the compact subsets of $G$ determined by $w$ as of Axioms~(W2), (W3) and~(W4). 
We fix $n \geq n_{\varepsilon}$ and we apply Lemma~\ref{lemma:tilingswf}
in order to find prototile sets 
\[
\{e \} \subseteq T_n \subseteq S_1^{\varepsilon} \subseteq \dots \subseteq S_{N(\varepsilon)}^{\varepsilon}, \quad S_i^{\varepsilon} \in \{ T_k: k \geq \max\{ i,n \} \},
\]
as well as $M \in \NN$ with $M \geq n$ such that for all $l \geq M$ we have $m_l \geq n$ and for all $h \in G$ we find a finite set 
$C_i^{l,h} \subseteq A_{l,h}$ for each $1 \leq i \leq N(\varepsilon)$ such that the properties~(T1)-(T4) from Definition~\ref{defi:qt} are satisfied and at the same time we obtain
\begin{equation*}
\frac{w\big(A_{l,h},\, \Lambda \big)}{m_G(A_{l,h})} \leq \sum_{i=1}^{N(\varepsilon)} \sum_{c \in C_i^{l,h}} \frac{ w\big( S_i^{\varepsilon}c, \Lambda \big)}{m_G(A_{l,h})}  \, +\, t_1 \cdot \varepsilon,
\end{equation*}
where the constant $t_1 \geq 0$ only depends on the parameters $\eta, \theta, \vartheta$ given by the 
definition of the weight function $w$. We emphasize at this point that the fact that the parameters $n, M$ (and hence also $l$) can be chosen independently of $h$ is justified by the observation that the validity of Lemma~\ref{lemma:tilingswf} depends on the invariance properties of $A_{l,h}$ with respect to the prototile sets, combined with the fact that $\lim_l m_G(\partial_{K}(A_{l,h}))/m_G(A_{l,h}) = 0$ uniformly over $h$ for all $K \in \mathcal{RK}(G)$.

\medskip

Exploiting the sub-additivity property~(W3) of $w$ and recalling 
$A_{l,h}:= T_{\mathcal{R}^{\mathcal{T}}_{\Lambda}(\delta, m_l)}h \setminus S_l \widetilde{g}_{l,h}$ we find with the estimates~\eqref{eqn:Al} and~\eqref{eqn:inv11}, as well as with $l \geq M$ that 
\[
\frac{w\big(  T_{\mathcal{R}^{\mathcal{T}}_{\Lambda}(\delta, m_l)}h, \Lambda \big)}{m_G(T_{\mathcal{R}^{\mathcal{T}}_{\Lambda}(\delta, m_l)})} \leq 
\frac{m_G(S_l)}{m_G(T_{\mathcal{R}^{\mathcal{T}}_{\Lambda}(\delta, m_l)})} \cdot \frac{w\big( S_l \widetilde{g}_{l,h}, \Lambda \big)}{m_G(S_l)}  +  \frac{m_G(A_{l,h})}{m_G(T_{\mathcal{R}^{\mathcal{T}}_{\Lambda}(\delta, m_l)})} \cdot \frac{w\big( A_{l,h}, \Lambda \big)}{m_G(A_{l,h})} + t_2 \cdot \varepsilon
\]
for some constant $t_2 \geq 0$ which only depends on the sub-additivity parameter $ \theta $ of $w$ and on $\delta$. Recall that $l \geq M \geq m_0$. Thus, 
we can combine the latter inequality with the Inequality~\eqref{eqn:mainaux1} and $m_G(S_l)=m_G(T'_l)$ in order to obtain
\[
\frac{w\big( T_{\mathcal{R}^{\mathcal{T}}_{\Lambda}(\delta, m_l)}h, \Lambda \big)}{m_G( T_{\mathcal{R}^{\mathcal{T}}_{\Lambda}(\delta, m_l)})} \leq 
\frac{m_G(T^{\prime}_l)}{m_G(T_{\mathcal{R}^{\mathcal{T}}_{\Lambda}(\delta, m_l)})} \cdot \Bigg( \frac{w(S_l, \Lambda)}{m_G(T_l^{\prime})} + \frac{11}{10} \varsigma \Bigg)
+  \frac{m_G(A_{l,h})}{m_G(T_{\mathcal{R}^{\mathcal{T}}_{\Lambda}(\delta, m_l)})} \cdot \frac{w\big( A_{l,h}, \Lambda \big)}{m_G(A_{l,h})} + t_2 \cdot \varepsilon.
\]
Invoking the sub-additivity induced by the above $\varepsilon$-quasi tilings and using the basic inequality $m_G(A_{l,h}) \leq m_G(T_{\mathcal{R}^{\mathcal{T}}_{\Lambda}(\delta, m_l)})$, the latter inequality transforms to 
\begin{eqnarray*}
	\frac{w\big(T_{\mathcal{R}^{\mathcal{T}}_{\Lambda}(\delta, m_l)}h, \Lambda \big)}{m_G(T_{\mathcal{R}^{\mathcal{T}}_{\Lambda}(\delta, m_l)})} &\leq& 
	\frac{m_G(T^{\prime}_l)}{m_G(T_{\mathcal{R}^{\mathcal{T}}_{\Lambda}(\delta, m_l)})} \cdot \Bigg(\frac{w(S_l, \Lambda)}{m_G(T_l^{\prime})} + \frac{11}{10} \varsigma \Bigg) \\
	&& \quad \quad +  \frac{m_G(A_{l,h})}{m_G(T_{\mathcal{R}^{\mathcal{T}}_{\Lambda}(\delta, m_l)})} \cdot  \sum_{i=1}^{N(\varepsilon)} \sum_{c \in C_i^{l,h}} \frac{m_G(S_i^{\varepsilon}c)}{m_G(A_{l,h})} \frac{ w\big( S_i^{\varepsilon}c, \Lambda \big)}{m_G(S_i^{\varepsilon}c)} \quad  +\, (t_1 + t_2) \cdot \varepsilon.
	\end{eqnarray*}
Note that the tiling properties (T1)-(T4) from Definition~\ref{defi:qt} yield
\[
(1- 2 \varepsilon) \leq \sum_{i=1}^{N(\varepsilon)} \sum_{c \in C_i^{l,h}} \frac{m_G(S_i^{\varepsilon}c)}{m_G(A_{l,h})} \leq \frac{1}{1-\varepsilon}. 
\]
Furthermore, the Inequality~\eqref{eqn:startproof} leads to
\[
\frac{m_G(T^{\prime}_l)}{m_G(T_{\mathcal{R}^{\mathcal{T}}_{\Lambda}(\delta, m_l)})} \cdot \Bigg( \frac{w(S_l, \Lambda)}{m_G(T_l^{\prime})} + \frac{11}{10} \varsigma \Bigg)  \leq
\frac{m_G(T^{\prime}_l)}{m_G(T_{\mathcal{R}^{\mathcal{T}}_{\Lambda}(\delta, m_l)})} \cdot \big( \nu^{+} - \frac{9}{10} \varsigma \big)
\]
for $l \geq M \geq L_{\varsigma}$. 
Recall that we have chosen $n \geq n_{\varepsilon}$ 
and $S_i^{\varepsilon} \in \{ T_k: k \geq \max\{ i,n \} \}$. 
Thus, 
\[
\frac{w\big(S_i^{\varepsilon}c,\, \Lambda \big)}{m_G(S_i^{\varepsilon}c)} 
	\leq w^+(S_i^{\varepsilon},\Lambda)
	\leq \nu^{+} + \varepsilon
\]
holds for all $1 \leq i \leq N(\varepsilon)$, all $l \geq M$, each $h \in G$ and each $c \in C_i^{l,h}$.

\medskip

Since $\Lambda$ is almost tempered repetitive the repetitivity portion $\zeta:[0,1) \to [0,1]$ of $\Lambda$ with respect to $\mathcal{T} = (T_m)$
is positive everywhere and in particular, we have
\[
\inf_{m \in \NN} \frac{m_G(T_m)}{m_G(T_{\mathcal{R}^{\mathcal{T}}_{\Lambda}(\delta, m)})} = \zeta(\delta) > 0.
\]
We set $a_{\varepsilon} := (1-2\varepsilon)$ if $\nu^+ \leq 0$ and $a_{\varepsilon} := (1-\varepsilon)^{-1}$ if $\nu^{+} > 0$ and $t_3 := t_1 + t_2$.   
Since $a_{\varepsilon} \nu^{+} -  \nu^{+}\geq 0$,
we finally derive from the above estimates
\begin{align*}
	\frac{w\big(  T_{\mathcal{R}^{\mathcal{T}}_{\Lambda}(\delta, m_l)}h, \Lambda \big)}{m_G(T_{\mathcal{R}^{\mathcal{T}}_{\Lambda}(\delta, m_l)})} \leq& 
	\frac{m_G(T^{\prime}_l)}{m_G(T_{\mathcal{R}^{\mathcal{T}}_{\Lambda}(\delta, m_l)})}  \big( \nu^{+} - \frac{9}{10} \varsigma \big)
	+ \frac{m_G(T_{\mathcal{R}^{\mathcal{T}}_{\Lambda}(\delta, m_l)}) - m_G(T_l^{\prime})}{m_G(T_{\mathcal{R}^{\mathcal{T}}_{\Lambda}(\delta, m_l)})} \Big( a_{\varepsilon} \nu^{+} + \frac{\varepsilon}{1-\varepsilon} \Big) + t_3 \varepsilon \\
	\leq& a_{\varepsilon}\nu^{+} - \frac{m_G(T^{\prime}_l)}{m_G(T_{\mathcal{R}^{\mathcal{T}}_{\Lambda}(\delta, m_l)})}\Bigg(a_{\varepsilon} \nu^{+} - \Big( \nu^{+} - \frac{9}{10}\varsigma \Big) \Bigg)  + (t_3 + 2)\varepsilon \\
	\leq& a_{\varepsilon}\nu^{+} - \zeta(\delta) \cdot \Big( a_{\varepsilon} \nu^{+} -  \nu^{+} + \frac{9}{10}\varsigma \Big)  + (t_3 + 2)\varepsilon,
\end{align*}
for all $l \geq M$ and all $h \in G$,
where the tempered repetitivity was used in the last step. Thus, taking the supremum over all $h \in G$ on the left hand side of the latter inequality and sending $l \to \infty$, we find that 
\[
\nu^{+} \leq  a_{\varepsilon}\nu^{+} - \zeta(\delta) \cdot \Big( a_{\varepsilon} \nu^{+} -  \nu^{+} + \frac{9}{10}\varsigma \Big)  + (t_3 + 2)\varepsilon.
\] 
By sending $\varepsilon \to 0$, we obtain
$\nu^{+} \leq \nu^{+} - 9/10 \zeta(\delta) \varsigma $ and since both $\varsigma$ and $\zeta(\delta)$ are positive numbers, we have arrived at a contradiction. Hence, we conclude $\nu^{+} = \nu^{-}$ in the case~(A).

\medskip

The case~(B) is much easier to handle. We assume with no loss of generality that 
$$
	\lim_{l \to \infty} \frac{m_G(T_{m_l})}{m_G(T_{\mathcal{R}^{\mathcal{T}}_{\Lambda}(\delta, m_l)})}~=~1.
$$
Then there is some $M_{\varsigma} \in \NN$ such that 
$\sup_{h \in G} m_G(A_{l,h}) \big(m_G\big(T_{\mathcal{R}^{\mathcal{T}}_{\Lambda}(\delta, m_l)}\big)\big)^{-1} < \frac{\varsigma}{5\eta}$
for all $l \geq M_{\varsigma}$, where $\eta > 0$ is the boundedness constant of $w$ as of the property~(W2). 
Increasing $M_{\varsigma}$ if necessary we can assume that 
$\sup_{h \in G} m_G( \partial_L(A_{l,h})) \big(m_G\big(T_{\mathcal{R}^{\mathcal{T}}_{\Lambda}(\delta, m_l)}\big)\big)^{-1} < \frac{1}{10}\varsigma (\max\{1,\eta,\theta \})^{-1}$
and $\frac{m_G(\partial_L(S_l))}{m_G(S_l)} < \frac{1}{10} \varsigma (\max\{1,\theta \})^{-1}$
for all $L \in \{ J,B,I \}$ and each $l \geq M_{\varsigma}$, where $J,B,I \subseteq G$ are compact and $\theta \geq 0$
depend on the weight function $w$. We now fix $l \geq M_{\varsigma}$.
 By the choice of $\widetilde{g}_{l,h}$, the Inequalities~\eqref{eqn:startproof} and~\eqref{eqn:mainaux1} together with $m_G(S_l)=m_G(T'_l)$
imply 
\[
\frac{w\big( S_l \widetilde{g}_{l,h} ,\, \Lambda \big)}{m_G(S_l)} \leq \frac{w\big(S_l,\, \Lambda \big)}{m_G(S_l)} + \frac{11}{10} \varsigma \leq \nu^{+} - \frac{9}{10}\varsigma
\]
for all $l \geq M_{\varsigma}$ and each $h \in G$. Then $m_G(\partial_L(C)) = m_G(\partial_L(Cg))$, the sub-additivity~(W3) with constant $\theta \geq 0$ (given by $w$) and the imposed invariance conditions imply
\begin{eqnarray*}
\frac{w\big( T_{\mathcal{R}^{\mathcal{T}}_{\Lambda}(\delta, m_l)}h ,\, \Lambda\big)}{m_G(T_{\mathcal{R}^{\mathcal{T}}_{\Lambda}(\delta, m_l)})}
&\leq& \frac{w\big(S_l \widetilde{g}_{l,h},\, \Lambda \big)}{m_G(T_{\mathcal{R}^{\mathcal{T}}_{\Lambda}(\delta, m_l)})} + \frac{w\big(A_{l,h},\, \Lambda \big)}{m_G(T_{\mathcal{R}^{\mathcal{T}}_{\Lambda}(\delta, m_l)})} + 
\theta \cdot \frac{m_G(\partial_{B}(S_l))}{m_G(S_l)} +  \theta \cdot \frac{m_G(\partial_B(A_{l,h}))}{m_G(T_{\mathcal{R}^{\mathcal{T}}_{\Lambda}(\delta, m_l)})} \\
&\leq& \frac{w\big(S_l \widetilde{g}_{l,h},\, \Lambda \big)}{m_G(T_{\mathcal{R}^{\mathcal{T}}_{\Lambda}(\delta, m_l)})} + \frac{w\big(A_{l,h},\, \Lambda \big)}{m_G(T_{\mathcal{R}^{\mathcal{T}}_{\Lambda}(\delta, m_l)})} + \frac{1}{5}\varsigma \\
&\leq& \frac{w\big(A_{l,h},\, \Lambda \big)}{m_G(T_{\mathcal{R}^{\mathcal{T}}_{\Lambda}(\delta, m_l)})} + \nu^{+} - \frac{9}{10}\varsigma + \frac{1}{5}\varsigma
\end{eqnarray*}
for all $l \geq M_{\varsigma}$ and all $h \in G$. 
Invoking Remark~\ref{rem:boundedness} with constant $\eta > 0$ applied to $w\big(A_{l,h},\, \Lambda \big)$, then the 
invariance condition on $A_{l,h}$ and the implications explained at the beginning of the proof of case~(B) yield
\begin{eqnarray*}
\frac{w\big( T_{\mathcal{R}^{\mathcal{T}}_{\Lambda}(\delta, m_l)}h ,\, \Lambda\big)}{m_G(T_{\mathcal{R}^{\mathcal{T}}_{\Lambda}(\delta, m_l)})}
&\leq& \eta \cdot \frac{m_G(A_{l,h})}{m_G(T_{\mathcal{R}^{\mathcal{T}}_{\Lambda}(\delta, m_l)})} + \eta \cdot  \frac{m_G(\partial_J(A_{l,h}))}{m_G(T_{\mathcal{R}^{\mathcal{T}}_{\Lambda}(\delta, m_l)})}  + \nu^{+} - \frac{9}{10}\varsigma + \frac{1}{5}\varsigma \\
&\leq& \frac{1}{5} \varsigma + \frac{1}{10} \varsigma + \nu^{+} - \frac{9}{10} \varsigma + \frac{1}{5} \varsigma = \nu^{+} - \frac{2}{5} \varsigma.
\end{eqnarray*}
for all $l \geq M_{\varsigma}$ and all $h \in G$. Thus, by taking first the supremum over all $h \in G$ and then sending $l \to \infty$ we arrive at 
\[
\nu^{+} \leq \nu^{+} - \frac{2}{5}\varsigma,
\]
which is a contradiction since $\varsigma > 0$.
\end{proof}

\section{Unique ergodicity} \label{UE}

\subsection{Unique ergodicity from tempered repetitivity}

We are now in a position to deduce Theorem~\ref{Intro2} from our abstract sub-additive convergence theorem, Theorem~\ref{thm:abstr}. The idea to connect ergodicity properties of dynamical systems $G \curvearrowright X$ to the existence of limits of the form $\lim_{m \to \infty} \frac{w(T_m,\Lambda)}{m_G(T_m)}$ for some class of  almost-additive functions $w:\mathcal{RK}(G) \times X \to \RR$ and strong F{\o}lner exhaustion sequence $(T_m)$, is well-established in the abelian case \cite{DaLe01, Len02b, LaPl03, FrRi14}. While the proof of the sub-additive convergence theorem was more complicated in the non-abelian case than in the abelian case, the application to the unique ergodicity statement is not much different, hence we will be brief.
 
 \medskip
 
For the remainder of this subsection we work in the setting of Theorem~\ref{Intro2}. Thus $\mathcal T = (T_m)$ is a strong F\o lner exhaustion sequence in an amenable unimodular lcsc group $G$ and $\Lambda$ is a weighted Delone set in $G$, which is almost tempered repetitive with respect to $\mathcal T$. We then denote by $\mathcal H_\Lambda$ the hull of $\Lambda$. In order to apply Theorem~\ref{thm:abstr} in this context, we need to construct suitable weight functions.\medskip

Given $f \in C(\mathcal{H}_{\Lambda})$, 
we define the mapping
\begin{equation} \label{SpecialWeights}
w_f: \mathcal{RK}(G) \times \mathcal{H}_{\Lambda} \to \RR, \,\, w(T, \Pi):= \int_T f(g.\Pi)\, dm_G(g).
\end{equation}
\begin{proposition} For every $f \in C(\mathcal{H}_{\Lambda})$, the function $w_f$ given by \eqref{SpecialWeights} is an {almost sub-additive weight function} over $\mathcal H_\Lambda$.
\end{proposition}
\begin{proof} We have to check Axioms (W1) -- (W4): (W1) holds by the convention of the empty integral being equal to $0$. Furthermore, (W2) holds with $J = \emptyset$ and $\eta = \| f\|_{\infty}$. The additivity of the integral yield (W3) (which holds as an equality) with $B = \emptyset$ and $\theta = 0$.  
 As for~(W4), a straight forward integral substitution allows us to choose $\vartheta = 0$ and $I = \emptyset$. 
\end{proof}\medskip

The key step is now to show that for $f \in C(\mathcal{H}_{\Lambda})$ the weight function $w_f$ is actually admissible, i.e.\ that it also satisfies Axiom (W5). We will use the fact that, by the same proof as in the abelian case \cite{MR13,FrRi14}, a metric generating the weak-$*$-topology on $\mathcal{H}_\Lambda$ is defined via
\[
d_{\ast}(\Phi,\Pi)
	:= \min\big\{ c_\ast,\, \inf\{\delta>0 :\, 
		\big|\delta_\Phi(B_\delta(y))-\delta_\Pi(B_\delta(y))\big|<\delta 
		\mbox{ for all } y\in \big(B_{1/\delta}\cap P\big)\cup \big(B_{1/\delta}\cap Q\big)		
		\}\big\}
\]
where $\Phi:=(P,\alpha),\Pi :=(Q,\beta)\in \mathcal{H}_\Lambda$ and $c_\ast>0$ are such that $B_{c_\ast}\subseteq U$ and $c_\ast<(2\sigma)^{-1}$. 
We will use this metric in the proof of the following lemma.

\begin{lemma} For every $f\in\mathcal{C}(\mathcal{H}_\Lambda)$ the weight function $w_f$ is admissible with respect to any strong F\o lner exhaustion sequence $(T_m)$.
\end{lemma}

\begin{proof}
	We fix $f\in\mathcal{C}(\mathcal{H}_\Lambda)$, a strong F\o lner exhaustion sequence $(T_m)$ and some $\varepsilon > 0$  throughout. We then have to show the existence of some $\delta>0$ and an $m_0\in\NN$ such that for all $m\geq m_0$, 
	we have the implication
	\begin{equation}\label{W5LemmaToShow}
	d_{T_m^{-1}}(\Pi,\Theta)\leq \delta 
	\quad\Longrightarrow\quad
	\big|w_f(T_m, \Pi) -w_f(T_m, \Theta)\big|\leq \varepsilon \cdot m_G(T_m)\,.
	\end{equation}
	Since $f$ is uniformly continuous on the compact space $\mathcal{H}_{\Lambda}$, there is a $0<\delta <c_\ast$ such that the condition $d_{\ast}(\Pi,\Theta) \leq \delta$ implies  $|f(\Pi) - f(\Theta)| \leq \frac{\varepsilon}{3}$.
	Since $(T_m)$ is a strong F{\o}lner exhaustion sequence and $B:=B_{1/\delta}\in\mathcal{RK}(G)$, we find some $m_0 \in \NN$ such that $\frac{m_G(\partial_B(T_m))}{m_G(T_m)}\leq\frac{\varepsilon}{3\|f\|_\infty}$ and $B \subseteq T_m$ for all $m\geq m_0$. 
	It follows that for all $\Theta \in \mathcal{H}_{\Lambda}$, 
	$$
	\left|\int_{T_m} f(g.\Theta)\, dg-\int_{T_m\setminus\partial_B(T_m)} f(g.\Theta)\, dg\right|
	\leq \int_{\partial_B(T_m)} |f(g.\Theta)|\, dg 
	\leq \|f\|_\infty m_G(\partial_B(T_m))
	\leq \frac{\varepsilon}{3} m_G(T_m).
	$$
	Note that since $e \in B$ and $B=B^{-1}$, we deduce $Bg \subseteq T_m$ for $g \in T_m \setminus\partial_B(T_m)$. Let $\Theta,\Pi\in\mathcal{H}_\Lambda$ such that $d_{T^{-1}_m}(\Pi,\Theta)\leq \delta$.
	Thus, we conclude $d_{g^{-1}B^{-1}}(\Pi,\Theta)\leq \delta$, and hence
	$d_{B}(g.\Pi, g.\Theta)=d_{B^{-1}}(g.\Pi, g.\Theta)\leq  \delta$ for all $g\in T_m\setminus\partial_B(T_m)$. 
	Since $\delta<c_\ast$, $d_{\ast}(g.\Pi,g.\Theta)<\delta$ follows by the definition of $d_\ast$ for all $g\in T_m\setminus\partial_B(T_m)$.
	Thus, $g\in T_m\setminus\partial_B(T_m)$ yields $|f(g.\Pi)-f(g.\Theta)|\leq\frac{\varepsilon}{3}$ and we arrive at
	\begin{eqnarray*}
		\left|\int_{T_m\setminus\partial_B(T_m)} f(g.\Pi)\, dg-\int_{T_m\setminus\partial_B(T_m)} f(g.\Theta)\, dg\right|
		&\leq& \int_{T_m\setminus\partial_B(T_m)} |f(g.\Pi)-f(g.\Theta)| \, dg \\
		&\leq& \frac{\varepsilon}{3}\, m_G(T_m\setminus\partial_B(T_m))
		\leq \frac{\varepsilon}{3}\, m_G(T_m)\,.
	\end{eqnarray*}
	Now \eqref{W5LemmaToShow} follows from the triangle inequality.	
\end{proof}

\medskip

We can now prove unique ergodicity for almost tempered repetitive Delone sets $\Lambda$ in $G$. 

\medskip

\begin{proof}[Proof of Theorem~\ref{Intro2}]
	Suppose that $\Lambda$ is almost tempered repetitive with respect to the strong F{\o}lner exhaustion sequence $(T_m)$.
 Let $f \in C(\mathcal{H}_{\Lambda})$. Since $w_f$ as considered above is an admissible weight function we deduce from Theorem~\ref{thm:abstr} that there is a number $I(f) \in \RR$ such that 
	\begin{eqnarray*} 
	I(f) = \lim_{n \to \infty} \frac{w_f(T_m, \Pi)}{m_G(T_m)} = \lim_{m \to \infty} \frac{1}{m_G(T_m)} \int_{T_m} f(g.\Pi)\, dg
	\end{eqnarray*}
	for all $\Pi \in \mathcal{H}_{\Lambda}$. We claim that for an arbitrary  $G$-invariant Borel probability measure $\nu$ on $\mathcal{H}_{\Lambda}$, we get $I(f) = \int_{\mathcal{H}_{\Lambda}} f\, d\nu$. Indeed, combining the latter limit relation with the dominated convergence theorem, Fubini's theorem and the $G$-invariance of $\nu$, we arrive at
	\begin{eqnarray*}
	I(f) &=& \int_{\mathcal{H}_{\Lambda}} \lim_{m \to \infty} \frac{1}{m_G(T_m)} \int_{T_m} f(g.\Pi)\, dg \, d\nu(\Pi) 
	= \lim_{m \to \infty} \frac{1}{m_G(T_m)} \int_{T_m} \int_{\mathcal{H}_{\Lambda}} f(g.\Pi)\, d\nu(\Pi)\, dg \\
	&=& \lim_{m \to \infty} \frac{1}{m_G(T_m)} \int_{T_m} \int_{\mathcal{H}_{\Lambda}} f(\Pi)\, d\nu(\Pi) \, dg = \int_{\mathcal{H}_{\Lambda}} f(\Pi)\, d\nu(\Pi).
	\end{eqnarray*}
	This shows that for any two $G$-invariant Borel probability measures $\nu_1$ and $\nu_2$ we obtain 
	\[
	\int_{\mathcal{H}_{\Lambda}} f(\Pi)\, d\nu_1(\Pi) = I(f) = \int_{\mathcal{H}_{\Lambda}} f(\Pi)\, d\nu_2(\Pi)
	\]
	for all $f \in C(\mathcal{H}_{\Lambda})$. This in turn is equivalent to $\nu_1 = \nu_2$. Therefore, the hull of $\Lambda$ is uniquely ergodic. It is also minimal by Proposition~\ref{prop:minimalitytemp}. 
\end{proof}

\subsection{Metrics of exact polynomial growth and linear repetitivity} \label{sec:metrics}
In the previous subsection we have established Theorem~\ref{Intro2} which asserts that every tempered repetitive weighted Delone set with respect to a strong F\o lner exhaustion sequence in an lcsc amenable group is uniquely ergodic and minimal. While the notion of a Delone set was originally defined using a (continuous) adapted metric on $G$, it does not depend on any choice of metric on $G$, but only on $G$ as a topological group. The same metric independence thus holds for the statement of the theorem. In the present subsection we would now like to relate tempered repetitivity to the more classical notion of linear repetitivity; this notion, unlike the previous ones, does depend on a choice of metric. To obtain a metric version of Theorem~\ref{Intro2} we will have to choose a metric with the additional property of exact polynomial volume growth. As we will see in Theorem~\ref{Bre2} below, such metrics exist on all compactly-generated groups of polynomial growth (e.g.\ word metrics with respect to compact generating sets), but it seems to be unknown in which generality they can be chosen to be continuous. For this reason we insist for the remainder of this section that $d$ is an adapted metric on $G$, which is not necessarily continuous. We still use the same notation as before concerning open and closed ball, but we emphasize that without the continuity assumption $B_t$ need not be open and $\overline{B}_t$ need not be closed in the topology of $G$. On the other hand, left-invariance of $d$ still implies that 
\begin{equation}
B_rB_s \subset B_{r+s}  \quad \text{and} \quad \overline{B}_r \overline{B}_s \subset \overline{B}_{r+s} \quad \text{for all }r,s >0.
\end{equation}
 Since $d$ is locally bounded we may choose $r_0$ such that $B_{r_0}$ contains an open identity neighbourhood. We claim that then
\begin{equation}\label{TF1}
\overline{B_r} \subset \mathring{\overline{B_{s}}} \quad \text{and} \quad \overline{B_r} \subset B_{r+ r_0} \quad \text{for all }r > 0 \text{ and }s \geq r+r_0.
\end{equation}
Indeed, the former follows from the fact that every $x \in \overline{B_r}$ admits an open neighbourhood which is contained in $xB_{r_0} \subset  \overline{B_r} \overline{B_{r_0}} \subseteq \overline{B_{r+r_0}} \subset \overline{B_s}$. For the latter, if $x \in \overline{B_r}$, then there exist $x_n \in B_r$ such that $x_n \to x$ and hence $x_n^{-1}x \in B_{r_0}$ for some sufficiently large $n_0$, i.e. $x \in x_{n_0}^{-1}B_{r_0} \subset B_rB_{r_0} \subset B_{r+r_0}$.

\medskip

If $d$ is continuous, then $B_t$ is open and $\overline{B}_t$ is closed and we have
\begin{equation}\label{TF2}
B_s \subset \overline{B_s} \subset \overline{B}_s 
	\subset \mathring{\overline{B}_t} \subset \overline{B}_t \quad \text{for all}\; t>s>0.
\end{equation}
Also, if $K \subset G$ is compact and $d$ is continuous, then there exists $r_K > 0$ such that $K \subset B_{r_K}$. This latter property is invariant under coarse equivalence, hence holds for all adapted pseudo-metrics on $G$, even if they are discontinuous by Proposition~\ref{AdaptedMetrics}. We will be interested in metrics with the following special property (see \cite[Section~4.4]{Nev06} for background information)
\begin{definition}[Groups with exact polynomial growth with respect to a metric] \label{defi:exactgrowth}
	Let $G$ be an lcsc group and let $d$ be an adapted metric on $G$. We say that $G$ {\em has exact polynomial growth with respect to $d$} 
	if there are constants $C > 0$ and $q \geq 0$ such that 
	\[
	\lim_{t \to \infty} \frac{m_G(B_t)}{C t^{q}} = 1. 
	\]
\end{definition}
Note that exact polynomial growth is a property of the pair $(G,d)$, not just of the (topological) group $G$. It is also not invariant under replacing $d$ by a coarsely equivalent metric, hence $G$ may have exact polynomial growth with respect to some adapted metrics, but not for others. Note that exact polynomial growth implies that for all $r>0$ we have
\begin{equation}\label{EPGStupid}
\lim_{t \to \infty} \frac{m_G(B_{t+r})}{m_G(B_{t})} = \lim_{t \to \infty} \frac{m_G(B_{t+r})}{m_G(B_{t})} \frac{Ct^q}{C(t+r)^q} = 1.
\end{equation}
The following observation is crucial for our purposes. Though it is well-known to experts (see e.g.\@ \cite[Proposition~4.13]{Nev06} or \cite[Corollary~1.10]{Bre14} for a slightly different formulation), we give a proof for the sake of being self-contained.
	\begin{proposition}[Balls as F\o lner sequences] \label{prop:folner!}
		Let $G$ be an lcsc group that has exact polynomial growth with respect to some adapted metric $d$. Assume that either
	\begin{enumerate}[(a)]
	\item $d$ is continuous and $T_n := \overline{B}_{r_n}$, where $(r_n)_{n \in \mathbb N}$ is any strictly monotone sequence with $\lim_n r_n = \infty$;
	\item or $T_n := \overline{B_{r_n}}$, where $r_{n+1} - r_n > r_0$ for all $n \in \mathbb N$, where $r_0>0$ is chosen such that $B_{r_0}$ contains an open identity neighbourhood
	\end{enumerate}	
	Then $(T_n)_{n \in \mathbb N}$ is a strong F{\o}lner exhaustion sequence.
	\end{proposition}
	
\begin{proof} To deal with both cases simultaneously we set $r_0 :=0$ in Case (a). Note first that in either case the sets $T_n$ are compact by continuity, respectively properness of $d$ (whereas closed balls need not be topologically closed, let alone compact, in Case (b)). Moreover we have $T_n \subset \mathring{T}_{n+1}$ by \eqref{TF2} and \eqref{TF1} respectively, and hence $(T_n)$ is a strong exhaustion sequence in either case.

Moreover, by the previous remarks we can find for every compact subset $K \subset G$ some $r_K > 0$ with $K \subset B_{r_K}$. This yields
\[
\partial_KT_n  =  K^{-1} T_n \cap K^{-1}\big( G \setminus T_n \big)  \subset B_{r_K}T_n \cap B_{r_K}\big( G \setminus T_n \big).
\]
We claim that if $x \in B_{r_n-r_K}$, then $x \not\in \partial_K T_n$. Indeed, otherwise we would find $y \in B_{r_K}$ and $z \in G \setminus T_n$ such that $x = yz$ and hence $z = y^{-1}x \in B_{r_n} \subset T_n$. On the other hand we deduce that in Case (a) we have $\partial_KT_n \subset B_{r_K} \overline{B}_{r_n} \subset \overline{B}_{r_K + r_n} \subset B_{r_n + r_K + 1}$. In Case (b) we can use \eqref{TF1} to similarly deduce
\[
\partial_KT_n \subset B_{r_K} \overline{B_{r_n}} \subset \overline{B_{r_K + r_n}} \subset B_{r_K + r_n + r_0}
\]
Thus in either case we have
\[
\partial_KT_n  \subset B_{r_n + r_K + r_0 + 1} \setminus B_{r_n -r_K}.
\]
Using \eqref{EPGStupid} we obtain
	\begin{align*}
	\limsup_{n \to \infty} \frac{m_G(\partial_KT_n)}{m_G(T_n)} &\leq \limsup_{n \to \infty} \frac{m_G(B_{r_n + r_K + r_0 + 1} \setminus B_{r_n -r_K})}{m_G({B}_{r_n})}  \\  
	&\leq \lim_{n \to \infty} \frac{m_G(B_{r_n+(r_K+r_0+1)})}{m_G(B_{r_n})}- \lim_{n \to \infty} \frac{m_G(B_{r_n-r_K})}{m_G(B_{(r_n-r_K) + r_K})} = 1 - 1 = 0. 
	\end{align*}
	This finishes the proof.
\end{proof}

We will see that if one restricts attention to strong F{\o}lner exhaustion sequences which arise from metrics of exact polynomial growth as above, then our notion of (almost) tempered repetitivity reduces to the more classical and purely metric notion of (almost) linear repetitivity. Let us recall the relevant definitions:

\begin{definition} Let $G$ be an lcsc group and let $d$ be an adapted metric on $G$. A weighted Delone set $\Lambda$ in $G$ is called {\em almost linearly repetitive with respect to $d$} if it is almost repetitive and there exists some function $c_{\Lambda}:(0,\infty) \to (0,\infty)$ such that 
$$
\varrho(\delta, R) := c_{\Lambda}(\delta) R, \qquad (\delta,R)\in(0,\infty)\times [1,\infty),
$$ 
is an almost repetitivity function.
\end{definition}

For weighted Delone sets with FLC it makes sense to work with the following definition. 

\begin{definition}\label{DefLinRep}
A weighted Delone set $\Lambda$ on $G$ is called {\em linearly repetitive} if it is repetitive and there exists a constant $c_{\Lambda}$ such that $\varrho(R) := c_{\Lambda} R,\, R\in [1,\infty),$ is a repetitivity function.
\end{definition}

We emphasize that (almost) linear repetitivity is a metric notion, whereas (almost) tempered repetitivity is a notion relative to a given strong F\o lner exhaustion sequence. However, we have the following relation.

\begin{proposition}[Linear repetitivity vs.\ tempered repetitivity]
\label{prop:Temp-LR} Assume that $G$ is an lcsc group that has exact polynomial growth with respect to some adapted metric $d$ and let $r_0 > 0$ such that $B_{r_0}$ is an identity neighbourhood. Then for a weighted Delone set $\Lambda$ in $G$ the following statements are equivalent:
\begin{enumerate}[(i)]
\item $\Lambda$ is almost linearly repetitive with respect to $d$.
\item $\Lambda$ is almost tempered repetitive with respect to every strong F\o lner exhaustion sequence of the form
$\mathcal{T} = (\overline{B_{r_n}})$ with the sequence $(r_n)$ satisfying 
 \[
 r_0 \leq \inf_n (r_{n+1} - r_n) \leq \sup_{n}( r_{n+1} - r_n) < \infty.
 \]
\item $\Lambda$ is almost tempered repetitive with respect to the strong F\o lner exhaustion sequence $(\overline{B_{nr_0}})$.
\end{enumerate}
If $d$ is continuous, then we can replace closures of balls by closed balls, the condition in (ii) by
 \[0< \inf_n (r_{n+1} - r_n) \leq \sup_{n}( r_{n+1} - r_n) < \infty,\]
 and the strong F\o lner exhaustion sequence in (iii) by $(\overline{B}_{n})$. \\ 
 If $\Lambda$ has FLC, then we can drop the word ``almost'' from any of the above equivalences.
\end{proposition}

\begin{remark}
Note that due to {Part~(b) of} Proposition~\ref{prop:folner!}, the sequences appearing in the Assertions (ii) and (iii) of Proposition~\ref{prop:Temp-LR} are indeed strong F{\o}lner exhaustion sequences. {In the case where the metric is additionally continuous, we find that the closed balls give rise to strong F{\o}lner exhaustion sequences by Part~(a) of Proposition~\ref{prop:folner!}.}  Note that the condition on the sequence~$(r_n)$ given in Assertion~(ii) implies that $r_n < r_{n+1}$ for all $n \in \NN$, as well as $r_n \nearrow \infty$ as $n \to \infty$. 
However, even if $d$ is continuous these two monotonicity conditions are not sufficient for $\Lambda$ being tempered repetitive with respect to $\mathcal{T}$, as can be inferred by Example~\ref{exa:crazygrowth}.
\end{remark}
\begin{proof} Assume that $G$ has exact polynomial growth with respect to $d$ with parameters say $C > 0$ and $q \geq 0$ and let $(t_m)$ be any strictly increasing sequence such that $t_m \nearrow \infty$ as $m \to \infty$. We recall from \eqref{EPGStupid} that for all $c > 0$ there exists a constant $\gamma > 1$ such that, for all $m \in \NN$
\[
m_G(B_{ct_m + r_0}) =  \frac{ m_G(B_{ct_m + r_0})}{m_G(B_{ct_m})} \cdot m_G(B_{ct_m}) \leq \gamma m_G(B_{ct_m}).
\]

Then for every $c^{\prime} > 0$ we find a constant $\beta \geq 1$ such that for all $m \in \NN$ and $c \in \{1, c'\}$,
\[
	\beta^{-1} C (ct_{m})^{q} \leq m_G(B_{ct_m}) \leq m_G(\overline{B_{c t_{m}}}) \leq m_G(B_{ct_m + r_0})\leq \gamma m_G(B_{ct_m}) \leq \beta C (ct_{m})^q.\]
The same argument also holds for $\overline{B}_{ct_m}$ instead of $\overline{B_{c t_{m}}}$ (one can set $r_0 = 1$ in this situation).
\medskip

	(i) $\Longrightarrow$ (ii): Assume that $\Lambda$ is almost linearly repetitive with almost repetitivity function $\varrho(\delta,R) = c_{\Lambda}(\delta) \cdot R$ (where $c_{\Lambda} = c_{\Lambda}(0)$ in the FLC case), and let $\mathcal T = (T_m)$, where $T_m:= \overline{B_{r_m}}$. We set $a:= \inf_m (r_{m+1} - r_m) > 0$, as well as $A := \sup_m (r_{m+1} - r_m) < \infty$. Then by the given almost linear repetitivity we get $\mathcal{R}_{\Lambda}^{\mathcal{T}}(\delta, m) \leq a^{-1} c_{\Lambda}(\delta) r_{m}$ for all $m \in \NN$. 
	Further, we have that $T_{\mathcal{R}_{\Lambda}^{\mathcal{T}}(\delta, m)} \subseteq \overline{B_{A \mathcal{R}_{\Lambda}^{\mathcal{T}}(\delta, m)}} $ for all $m \in \NN$.
	In line with the remark made at the beginning of the proof (applied with $t_m = r_m$ and with $c^{\prime} = A a^{-1} c_{\Lambda}(\delta)$) we find $\beta = \beta(\delta) \geq 1$ such that 
\begin{align*}
\inf_{m \in \NN} \frac{m_G(T_m)}{m_G(T_{\mathcal{R}_{\Lambda}^{\mathcal{T}}(\delta, m)})} \geq \inf_{m \in \NN} \frac{\beta(\delta)^{-1} \cdot C r_m^{q} }{m_G(\overline{B_{A a^{-1} c_{\Lambda}(\delta)\cdot r_m}})} \geq  \beta(\delta)^{-2} (Aa^{-1})^{-q} c_{\Lambda}(\delta)^{-q} > 0.
\end{align*}
Hence, the repetitivity portion of $\Lambda$ is bounded from below by $\zeta(\delta) \geq \beta(\delta)^{-2} \cdot (Aa^{-1})^{-q} c_{\Lambda}(\delta)^{-q}$, which is positive.
This shows that $\Lambda$ is almost tempered repetitive with respect to\@ $\mathcal{T}$. In the continuous case, the same argument works with $T_m := \overline{B}_{r_m}$ instead of $T_m:= \overline{B_{r_m}}$.

\medskip

(ii)  $\Longrightarrow$ (iii): This is obvious since the sequence $(r_n)$ with $r_n = r_0n$ (or $r_n = n$  in the continuous case) for all $n \in \NN$ satisfies the growth condition given in~(ii). 

\medskip

(iii) $\Longrightarrow$ (i): Since the statements are invariant under rescaling the metric, we may assume that $r_0 = 1$ to simplify notation. Assume that  $\Lambda$ is almost tempered repetitive with respect to $\mathcal{T} =  (\overline{B_n}) $.  Now let $R \geq 1$ and $m := \lceil R \rceil + 1$ such that $m \in \NN$ with $m \geq 2$ and $m-2 < R \leq m - 1$. By assumption, there is a function $\zeta:[0,1) \to [0,1]$ with $\zeta(\delta) >0$ for all $\delta > 0$ (and in the FLC case $\zeta(0) > 0$)  such that 
$$
0 < \zeta(\delta) \leq \inf_{m \in \NN} \frac{m_G(T_m)}{m_G(T_{\mathcal{R}^{\mathcal{T}}_{\Lambda}(\delta, m)})}. 
$$
 Using the considerations made at the beginning of the proof (applied with $t_m =m$ and $c^{\prime} = 1$) we obtain $\beta(\delta) \geq 1$ such that for all $m \in \NN$,
 \[
 {\mathcal{R}^{\mathcal{T}}_{\Lambda}(\delta, m)} \leq \left(\frac{\beta(\delta)^2}{\zeta(\delta)} \right)^{1/q} \cdot m. 
 \]
 Note that $m = (m-2) + 2 \leq R + 2R = 3R$. By the first inclusion given in~\eqref{TF1} and $R+1 \leq m$ we get $B_R \subseteq \mathring{\overline{B_{R+1}}} \subseteq \mathring{\overline{B_m}}$.  
It follows from the definition of the repetitivity index that every $B_R$-pattern $\delta$-occurs in $h \mathring{T}^{-1}_{\mathcal{R}^{\mathcal{T}}_{\Lambda}(\delta, m)} = \mathring{\overline{B_{{\mathcal{R}^{\mathcal{T}}_{\Lambda}(\delta, m)}}(h)}} \subseteq B_{ \mathcal{R}^{\mathcal{T}}_{\Lambda}(\delta, m) + 1}(h) $ for all $h \in G$, where the latter inclusion follows from the second inclusion in~\eqref{TF1}. We can now choose $\kappa(\delta) > 0$ large enough such that $\big( \beta(\delta)^2/\zeta(\delta)\big)^{1/q} 3R + 1 \leq \kappa(\delta)R$ for all $R \geq 1$,
 and then $\varrho(\delta, R):= \kappa(\delta) R$ is an almost repetitivity function for $\Lambda$. The special case where $d$ is a continuous metric follows along  similar lines, using the inclusions given in~\eqref{TF2}.  
\end{proof}

\medskip

This now allows us to restate Theorem~\ref{Intro2} in purely metric terms and thereby establish Theorem~\ref{Intro1a} from the introduction:

\begin{theorem}\label{LRConvenient} 
Assume that $G$ is an lcsc group that has exact polynomial growth with respect to an adapted metric $d$. If a weighted Delone set $\Lambda$ in $G$ is almost linearly repetitive with respect to $d$, then its hull $\mathcal H_\Lambda$ is minimal and uniquely ergodic.
\end{theorem}

\begin{proof}
The unique ergodicity follows from combining Proposition~\ref{prop:Temp-LR} with Theorem~\ref{Intro2}. Furthermore, $\mathcal{H}_\Lambda$ is minimal by Proposition~\ref{prop:minmin} since $\Lambda$ is almost (linearly) repetitive.
\end{proof}

\medskip

\subsection{Examples of groups with metrics of exact polynomial growth}
We have seen in the previous subsection that adapted metrics of exact polynomial growth give rise to strong F\o lner exhaustion sequences and allow us to translate the abstract notion of tempered repetitivity into the more concrete notion of linear repetitivity. We now provide some explicit examples of such metrics. Our starting point is the Cygan-Kor{\'a}nyi metric on the Heisenberg group:

\begin{example}\label{SayMyName}
Consider the $3$-dimensional Heisenberg group $G = \CC \times \RR$ with group law given by
\[
(z,t) \cdot (w,s) = \Big( z + w, t+s + \frac{1}{2} \operatorname{Im} \overline{z}w \Big)
\] 
and note that for $z \in \CC$ and $t \in \RR$ we have
\[
(z,t)^{-1} = (-z,-t).
\]
The {\em Cygan-Kor{\'a}nyi norm} $\| \cdot \|: G \to [0, \infty)$
\[
\big\| (z,t) \big\| := \sqrt[4]{|z|^4 + 16t^2}
\]
defines a group norm on the Heisenberg group, cf.\@ \cite[p.\@ 18, Inequality~(2.12)]{CDPT07}, which means that the formula
\[
d_G(g,h) = \| g^{-1}h\|.
\]
defines a left-invariant metric $d_G$ on $G$, called thee \emph{Cygan-Kor{\'a}nyi metric}. Since the Cygan-Kor{\'a}nyi norm is bounded above and below by suitable powers of the Euclidean norm on $\CC \times \RR \cong \RR^3$, the metric $d_G$ is proper and defines the group topology. Thus, $d_G$ is continuous and adapted. Using that the Haar measure $m_G$ on $G$ is the $3$-dimensional Lebesgue measure, a simple integration via polar coordinates using the substitution $s = r^2$ yields
\begin{eqnarray*}
m_G(B_R) &=& \underset{ x^2 + y^2 \leq R^2}{\int \int} \int\limits_{-\frac{1}{4} \sqrt{R^4 - (x^2 +y^2)^2}}^{{\frac{1}{4} \sqrt{R^4 - (x^2 +y^2)^2}}} \quad  1\, ds\, dx\, dy
\quad = \quad  \int_0^R r \Bigg( \int_{0}^{2\pi} \Bigg( \int\limits_{-\frac{1}{4}\sqrt{R^4 - r^4}}^{\frac{1}{4}\sqrt{R^4 - r^4}} \quad 1 \, ds \Bigg) \, d\theta \Bigg)\, dr \\
&=& 2\pi \cdot \int_0^R r \cdot \frac{1}{2} \sqrt{R^4 - r^4}\, dr \quad = \quad  \pi \cdot \int_0^R r \cdot \sqrt{R^4 - r^4}\, dr \quad = \quad  \frac{\pi}{2} \int_0^{R^2} \sqrt{R^4 - s^2}\, ds.
\end{eqnarray*}
Since  the above integral is equal to the area of the quarter circle with radius $R^2$, which is $(R^2)^2 \pi / 4$, we obtain
\[
m_G(B_R) =  \frac{\pi}{2} \cdot \frac{(R^2)^2 \cdot \pi}{4} = \frac{\pi^2}{8} R^4,
\]
and hence
\[	 \lim_{t \to \infty} \frac{m_G(B_{t+r} )}{m_G(B_t)} =  \lim_{t \to \infty} \frac{(t+r)^4}{t^4} = 1 \quad\quad \mbox{for all } r > 0. 
\]
Therefore, the Heisenberg group has exact polynomial growth with respect to the Cygan-Kor{\'a}nyi metric $d_G$.
\end{example}
The previous example can be generalized to cover the whole class of homogeneous Lie groups: As explained in \cite[Chapter~3]{FR16}, each homogeneous Lie group $G$ comes equipped with a class of continuous adapted metrics, called homogeneous metrics, which are mutually quasi-isometric, and a natural family of dilation automorphisms $(D_\lambda)_{\lambda>0}$ such that $B_t = D_t(B_1)$ for balls with respect to any homogeneous metric $d$. By \cite[p.100, Equation (3.6)]{FR16} this implies that  \begin{equation}\label{HomExPG}m_G(B_t) = m_G(B_1) t^{\kappa},\end{equation} where $\kappa$ is a constant depending on $G$ (called the \emph{homogeneous dimension} of the homogeneous group). We have seen in the case of the Heisenberg group that $\kappa = 4$. It follows immediately from \eqref{HomExPG} that homogeneous groups have exact polynomial groups with respect to homogeneous metrics, and this property also carries over to lattices in such group (see \cite[Proposition~3.36]{BHP21-prim}):

\begin{proposition}\label{HomogeneousExact} Let $G$ be a homogeneous Lie group. If $d$ is a homogeneous metric on $G$, then $G$ has exact polynomial growth with respect to $d$. Similarly, if $\Gamma$ is a lattice in $G$, then $\Gamma$ has exact polynomial growth with respect to $d|_{\Gamma \times \Gamma}$.\qed
\end{proposition}

At this point we can easily deduce Theorem~\ref{Intro1} from the introduction.

\begin{proof}[Proof of Theorem~\ref{Intro1}] In view of Theorem~\ref{LRConvenient} and the fact that linear repetitivity implies almost linear repetitivity we only have to show that every homogeneous metric $d$ is adapted and that $G$ has exact polynomial growth with respect to $d$.  The former follows from  \cite[Proposition~3.1.37]{FR16}, and the latter follows from 
Proposition~\ref{HomogeneousExact}.
\end{proof}

Note that Proposition~\ref{HomogeneousExact} generalizes Example~\ref{SayMyName}, since the Cygan-Kor{\'a}nyi metric is a homogeneous metric on the Heisenberg group. Our next goal is to characterize all groups which admit an adapted metric of exact polynomial growth.

In geometric group theory, a finitely generated group $\Gamma$ is said to have \emph{polynomial growth} provided for some (hence any) finite generating set $S$ there exist constants $C\geq 1$ and $k \in \mathbb N$ such that $|S^n| \leq C \cdot n^k$ for all $n \in \mathbb N$. More generally, if $G$ is a compactly-generated lcsc group, then we say that $G$ has \emph{polynomial growth} if for some (hence any) compact symmetric generating set $\Omega$ and Haar measure $m_G$ on $G$ there exist constants $C\geq 1$ and $k \in \mathbb N$ such that
\[
m_G(\Omega^n) \leq C \cdot n^k.
\]
We warn the reader that while finitely-generated discrete groups of polynomial growth are virtually nilpotent by a celebrated theorem of Gromov, this need not be the case in the compactly-generated lcsc groups of polynomial growth. However, such groups are still virtually compact-by-solvable, and hence amenable. A more precise structure theory for compactly-generated lcsc groups of polynomial growth was developed only fairly recently by Breuillard \cite{Bre14}. One consequence of this structure theory is the following volume growth formula, which is a special case of \cite[Theorem~1.1]{Bre14}:

\begin{theorem}[Breuillard's volume growth formula, {\cite[Theorem~1.1]{Bre14}}] \label{Breuillard-thm} Let $G$ be an lcsc group of polynomial growth. Then there exists $d(G) \in \mathbb N_0$ and
for every compact symmetric generating set $\Omega$ there exists a constant $c(\Omega) > 0$ such that
\begin{equation}\label{Breuillard}
\lim_{n \to \infty} \frac{m_G(\Omega^n)}{n^{d(G)}} = c(\Omega).
\end{equation}
\end{theorem}
Note that polynomial growth is a property of a (compactly-generated) lcsc group $G$, whereas the definition of exact polynomial growth also involves a choice of metric $d$. Nevertheless the two notions are closely connected:

\begin{theorem}[Polynomial growth vs.\ exact polynomial growth, cf.~\cite{Bre14}]
\label{Bre2} For a compactly-generated lcsc group $G$, the following conditions are equivalent:
\begin{enumerate}[(i)]
\item $G$ has polynomial growth.
\item $G$ admits an adapted metric $d$ of exact polynomial growth.
\end{enumerate}
\end{theorem}

\begin{proof} (ii) $\implies$ (i): Since the balls $B_n$ exhaust $G$ and $G$ is compactly-generated, it is generated by some ball $B_N$; rescaling the metric by a positive constant if necessary, we may thus assume that $G$ is generated by $B_1$. Now let $q>0$ be chosen such that $d$-balls satisfy the condition $\lim_{t \to \infty} \frac{m_G(B_t)}{C t^{q}} = 1$ and pick an integer $k > q$. Then we can find $C' > 0$ such that $m_G(B_t) < C' t^{k}$ for all $t\geq 1$ and hence
\[
m_G(B_1^n) \leq m_G(B_n) < C' n^k,
\]
which shows that $G$ has polynomial growth.
\medskip

(i) $\implies$ (ii): Let $\Omega$ be a compact generating set of $G$. It then follows from Theorem~\ref{Breuillard-thm}, that the word metric $d_\Omega$ with respect to $\Omega$ has exact polynomial growth, and this metric is obviously adapted.
\end{proof}

Note that the above proof of the implication (i) $\implies$ (ii)  does not produce a \emph{continuous} adapted metric of exact polynomial growth, unless $G$ is discrete to begin with. The remainder of this subsection is devoted to the question, for which groups of polynomial growth one can actually produce  a \emph{continuous} adapted metric of exact polynomial growth and will not be needed in the sequel. We start our discussion with a simple example:

If $G$ is a connected Lie group of polynomial growth, then any choice of inner product on the Lie algebra of $G$ will induce a left-invariant Riemannian metric on $G$, and the corresponding metric $d$ will be continuous and adapted. This metric will have the additional property of being \emph{geodesic} in the sense that for all $x,x' \in G$ there exists an isometry $\gamma$ from an interval $I = [a,b] \subset \mathbb R$ to a path in $G$ such that $\gamma(a) = x$ and $\gamma(b) = x'$ (by a suitable version of the Hopf--Rinow theorem). For an arbitrary compactly-generated lcsc group $G$ one only knows that there exists a continuous adapted metric $d$ on $G$, which is \emph{large-scale geodesic} in the sense of \cite[p.10]{CdlH}, but it is unclear under which conditions this metric can be chosen to be geodesic in the strict sense.

\begin{theorem}[Exact polynomial growth with respect to a continuous metric]\label{PGImpliesEPG} Assume that a compactly-generated lcsc group $G$ of polynomial growth admits a geodesic (rather than just a large-scale geodesic) continuous adapted metric $d$. Then $G$ has exact polynomial growth with respect to any such metric. In particular, a connected Lie group admits a continuous adapted metric of exact polynomial growth if and only if it has polynomial growth.
\end{theorem}

For the remainder of this section we assume that $d$ is a geodesic continuous adapted metric on a compactly-generated lcsc group $G$ of polynomial growth. The fact that $d$ is geodesic has the following consequence concerning balls in $G$.

\begin{lemma}
\label{BallPowers} 
The equality $\overline{B}_n = \overline{B}_1^n$ holds for all $n \in \mathbb N$.
\end{lemma}
\begin{proof} Let $x \in  \overline{B}_n$ with  $\alpha := \|x\|$ and $m := \lceil \alpha \rceil - 1 \leq n-1$. Choose a geodesic $\gamma$ in $G$ with $\gamma(0) = e$ and $\gamma(\alpha) = x$, and for $k = 1, \dots, n$ define $y_k := \gamma(k-1)^{-1}\gamma(k)$. Then by the left-invariance of the metric $d$ and the fact that 
	$d(e, y_k) = d(\gamma(k-1),\gamma(k)) = 1$, 
	we get 
	\[x = y_1 \cdots y_m (\gamma(m)^{-1}x) \in  \overline{B}_1^{m+1} \subset  \overline{B}_1^n,\] and hence $ \overline{B}_n \subset  \overline{B}_1^n$. The converse inclusion is immediate from left-invariance and the triangle inequality.
\end{proof}

In view of Breuillard's volume growth formula, this implies the theorem:

\begin{proof}[Proof of Theorem~\ref{PGImpliesEPG}]
	We apply \eqref{Breuillard} to $\Omega := \overline{B}_1$ and use that $m_G(\overline{B}_r)= m_G(B_r)$ for all $r > 0$. 
	For $q=d(G)$ and $c=c(\Omega)$ as in Theorem~\ref{Breuillard-thm} we obtain with Lemma~\ref{BallPowers} that 
	\[
	\lim_{n \to \infty} \frac{m_G(B_n)}{c \cdot n^q} = \lim_{n \to \infty} \frac{m_G(\overline{B}_n)}{c \cdot n^q} = \lim_{n \to \infty} \frac{m_G(\Omega^n)}{c \cdot n^q} = 1.
	\]
	Now for an arbitrary sequence $(t_n)$ with $\lim_{n \to \infty} t_n = \infty$, we have 
	\[
	1 = \lim_{n \to \infty} \frac{m_G(B_{\lfloor t_n \rfloor})}{c \cdot \big( \lfloor t_n \rfloor \big)^q} \cdot \lim_{n \to \infty} \frac{c \cdot \big( \lfloor t_n \rfloor \big)^q}{c \cdot \big( \lceil t_n \rceil \big)^q} \leq \liminf_{n \to \infty} \frac{m_G(B_{t_n})}{c \cdot t_n^q}.
	\]
	The converse inequality follows analogously.  
\end{proof}

\begin{remark}[Beyond connected Lie groups]
The simple argument used to establish Theorem~\ref{PGImpliesEPG} still carries through if $d$ is merely \emph{asymptotically geodesic} (see \cite[Thm.\ 14.3]{Nev06}), but it no longer works if $d$ is only assumed to be large-scale geodesic in the sense of \cite{CdlH}. In this case we can still find positive integers $a,b,c$ such that  for all $x,x' \in G$ there exist $k \in \mathbb N$ with $n \leq a d(x,x') + b$ and $x_0, \dots, x_k \in G$ with $x = x_0$, $x' = x_n$ and $d(x_{i-1}, x_i) \leq c$ for all $i \in \{1, \dots, k\}$, but this yields only the weaker statement
\[
\overline{B}_1^n \subset \overline{B}_n \subset \overline{B}_c^{an+b} \subset \overline{B}_c^{(a+b)n} \text{ for all }n \in \mathbb N
\]
instead of Lemma~\ref{BallPowers}. In the notation of Theorem~\ref{Breuillard-thm} this leads to the conclusion that
\[
c(\overline{B}_1) \leq \liminf_{n \to \infty} \frac{m_G(B_n)}{n^q} \leq \limsup_{n \to \infty} \frac{m_G(B_n)}{n^q} \leq c(\overline{B}_c^{(a+b)}),
\]
which in the language of \cite{Nev06} means that $G$ has \emph{strict} polynomial growth with respect to $d$. Our argument thus shows that every compactly-generated lcsc group $G$ of polynomial growth admits a \emph{continuous} adapted metric $d$ such that $G$ has \emph{strict} polynomial growth with respect to $d$. It seems to be an open problem 
for which classes of compactly-generated lcsc group of polynomial growth one can find a metric, which is both continuous and of exact polynomial growth, see the discussion in [Nev06, Section 4.4] and the references therein, in particular [Bre14] and [Pan83].
\end{remark}

\medskip

\section{Symbolic systems} \label{sec:symbolic}

In this section we describe how the previous considerations carry over to symbolic systems over a fintite alphabet and we give the proof of Theorem~\ref{thm:LRsymbolic}. 

\medskip

In the sequel, $\Gamma$ will denote an amenable, countable, discrete group. For the Haar measure on $\Gamma$ we fix the normalized counting measure $m_{\Gamma}$ which satisfies 
$m_{\Gamma}(\{e \}) = 1$ for the identity element $e \in \Gamma$. We will additionally assume that $\Gamma$ is a uniform lattice in an amenable lcsc group $H$, i.e.\@ $\Gamma$ is a discrete subgroup of $H$ whose quotient $H/\Gamma$ is compact in the quotient topology. This point of view is useful for constructing Delone sets in the ambient group $H$ via symbolic considerations over $\Gamma$. Moreover we do not lose any generality because we can always consider $\Gamma$ as a lattice in itself. 
We will also assume that $H$ carries adapted metric $d_H$ and its restriction to $\Gamma$ gives rise to an adapted $d_{\Gamma}$ on $\Gamma$ (which is automatically continuous since $\Gamma$ is discrete).

\medskip

We further fix a finite set $\mathcal{A}$, called  {\em alphabet}, which carries the discrete topology. Then $\mathcal{A}^{\Gamma}$ becomes a compact metrizable space when endowed with the product topology. We will call the elements in $\mathcal{A}^{\Gamma}$ {\em colorings} and the elements in $\mathcal{A}$ are called {\em colors}. 
 For any coloring $\mathcal{C} \in \mathcal{A}^{\Gamma}$ we can define its translation by $\gamma \in \Gamma$ as $(\gamma. \mathcal{C})(x) = \mathcal{C}(\gamma^{-1}x)$, for all $x \in \Gamma$. This leads to the notion of the  {\em hull of $\mathcal{C}$}, defined as  
\[
\Omega_{\mathcal{C}}:= \overline{\big\{ \gamma.\mathcal{C}:\, \gamma \in \Gamma \big\}},
\]
where the closure is taken in the product topology. Clearly, the hull is a compact metrizable space as well and $\Gamma$ acts on it via translations. We say that the hull $\Omega_{\mathcal{C}}$ is {\em uniquely ergodic} if it admits a unique $\Gamma$-invariant probability measure. Every $\mathcal{C} \in \mathcal{A}^{\Gamma}$ gives rise to canonical definition of (colored) patches and patterns. Writing $\mathcal{F}(\Gamma)$ for the collection of all finite subsets of $\Gamma$, for each $S \in \mathcal{F}(\Gamma)$ we denote the {\em $S$-patch of $\mathcal{C}$} as the restriction $\mathcal{C}_{|S}$ of the map $\mathcal{C}$ to the set $S$. We say that for $S,T \in \mathcal{F}(\Gamma)$, the corresponding patches of $\mathcal{C}$ are {\em equivalent} if there is some $\gamma \in \Gamma$ such that $\gamma S = T$ and $\mathcal{C}(\gamma s) = \mathcal{C}(s)$ for all $s \in S$. Given an $S$-patch of $\mathcal{C}$ we denote its equivalence class as the {\em pattern} of that $S$-patch. We write $[S]_{\mathcal{C}}$ for the collection of all patterns arising from all $\gamma S$-patches in $\mathcal{C}$. For $p \in [S]_{\mathcal{C}}$ and $T \in \mathcal{F}(\Gamma)$ we say that $p$ {\em occurs in $T$} if there are some representative $\mathcal{C}_{|xS}$ of $p$ and  $\gamma \in \Gamma$ such that $\gamma xS \subseteq T $ and $\mathcal{C}(\gamma xs) = \mathcal{C}(xs)$ for all $s \in S$.  
We point out that $[S]_{\mathcal{C}}$ is finite for some $\mathcal{C} \in \mathcal{A}^{\Gamma}$ since $\mathcal{A}$ is finite, i.e.\@ $\mathcal{C}$ is of finite local complexity with respect to\@ the above equivalence relation of patterns.

\medskip

	We will now relate symbolic colorings with weighted Delone sets in the ambient group $H$. To this end, we fix $\sigma \geq 1$ along with an injective map $\iota: \mathcal{A} \to [\sigma^{-1}, \sigma]$ and define the map 
	\[
	\mathcal{I}: \mathcal{A}^{\Gamma} \to \mathrm{Del}(U,K,\sigma), \quad \mathcal{C} \mapsto \sum_{\gamma \in \Gamma} \iota\big( \mathcal{C}(\gamma) \big) \cdot \delta_{\gamma}, 
	\]
	where $U$ open and $K$ are chosen such that $\Gamma$ is $(U,K)$-Delone in $H$.

	\begin{lemma} \label{lem:equivariant}
		The map $\mathcal{I}$ is a continuous, $\Gamma$-equivariant embedding. 
	\end{lemma}

	\begin{proof}
		The injectivity of $\mathcal{I}$ is clear from the injectivity of $\iota$ and uniform discreteness of $\Gamma$ in $H$. As for the $\Gamma$-equivariance, take $\gamma_0 \in \Gamma$, $\mathcal{C} \in \mathcal{A}^{\Gamma}$ and $B \subseteq H$ Borel. Then
		\begin{align*}
		\gamma_0.\mathcal{I}(\mathcal{C})(B) 
			&= \mathcal{I}(\mathcal{C})(\gamma_0^{-1}B) 
			= \sum_{\gamma \in \Gamma} \iota \big( \mathcal{C}(\gamma) \big) \delta_{\gamma}(\gamma^{-1}_0 B) 
			= \sum_{\gamma \in \Gamma} \iota \big( \mathcal{C}(\gamma) \big)\delta_{\gamma_0 \gamma}(B) \\
			&= \sum_{\gamma \in \Gamma} \iota \big( \mathcal{C}( \gamma_0^{-1}\gamma) \big) \delta_{\gamma}(B) 
			= \sum_{\gamma \in \Gamma} \iota((\gamma_0. \mathcal{C})(\gamma)) \delta_{\gamma}(B) 
			= \mathcal{I}(\gamma_0 .\mathcal{C})(B). 
		\end{align*}
		It remains to show that $\mathcal{I}$ is continuous. To this end, suppose that $(\mathcal{C}_n)\subseteq \mathcal{A}^\Gamma$ is a sequence and $\mathcal{C}$ is an element in $\mathcal{A}^\Gamma$ such that $\lim_{n \to \infty} \mathcal{C}_n = \mathcal{C}$. Fix a compact neighborhood $W$ of $\{e\}$ in $H$ with $W \cap \Gamma = \{e\}$.  
		Since all values of $\mathcal{I}$ are measures supported on $\Gamma$ it suffices to show that 
		$\lim_{n \to \infty} \mathcal{I}(\mathcal{C}_n)(\varphi) = \mathcal{I}(\mathcal{C})(\varphi)$ for all $\varphi \in C_c(H)$ with the support of $\varphi$ being contained in $KW$, where $K \subseteq \Gamma$ is finite. Indeed, given a finite set $K$ and fixing such $\varphi$, we find $n_K \in \NN$ such that  we have ${\mathcal{C}_n}_{|K} = \mathcal{C}_{|K}$ holds. Therefore, $\mathcal{I}(\mathcal{C}_n)(\varphi) = \mathcal{I}(\mathcal{C})(\varphi)$ follows
		for all $n \geq n_K$. This finishes the proof. 
		\end{proof}
	
		\begin{proposition} \label{prop:isomorphic}
			For any $\mathcal{C} \in \mathcal{A}^{\Gamma}$ the dynamical systems $\Gamma \curvearrowright \Omega_{\mathcal{C}}$ and $\Gamma \curvearrowright \mathcal{H}_{\mathcal{I}(\mathcal{C})}$ are topologically isomorphic. 
		\end{proposition}
	
		\begin{proof}
			It suffices to show that the restriction $\tau$ of $\mathcal{I}$ to $\Omega_{\mathcal{C}}$ gives rise to a $\Gamma$-equivariant homeomorphism onto $\mathcal{H}_{\mathcal{I}(\mathcal{C})}$. It follows from Lemma~\ref{lem:equivariant} (with $H=\Gamma$) that $\tau$ is an injective $\Gamma$-map. Moreover, $\tau$ is continuous as a restriction of the map $\mathcal{I}$ from Lemma~\ref{lem:equivariant}. The restriction of  $\tau$ to the orbit $\{\gamma.\mathcal{C}:\, \gamma \in \Gamma\} \subseteq \Omega_{\mathcal{C}}$ is bijective onto its image $\{\gamma.\mathcal{I}(\mathcal{C}):\, \gamma \in \Gamma \} \subseteq \mathcal{H}_{\mathcal{I}(\mathcal{C})}$. Since $\tau(\Omega_{\mathcal{C}})$ must be compact and hence closed as an image of a compact set under a continuous map with values in a  Hausdorff space, we must have $\tau(\Omega_{\mathcal{C}}) = \mathcal{H}_{\mathcal{I}(\mathcal{C})}$. So $\tau$ is surjective as well. Since $\tau$ is a bijective continuous mapping between two compact Hausdorff spaces, $\tau$ is indeed a homeomorphism.
		\end{proof}
	
	\medskip
	
	We transfer the repetitivity concepts from Definitions~\ref{defi:rep} and~\ref{DefATR} to symbolic systems. Since colorings $\mathcal{C} \in \mathcal{A}^{\Gamma}$ are of finite local complexity with respect to\@ the pattern equivalence relation described above we can find symbolic analogues of (linear) repetitivity and tempered repetitivity. 
	Recall that $H$ carries 
	an adapted metric $d_H$ that restricts to an adapted metric $d_{\Gamma}$ on $\Gamma$.
	
	For $r  > 0$ we denote by $B^{H}_r(h)$ the collection of points in $H$ with $d_{H}$-distance from the identity $h \in H$ being less or equal than $r$. We analogously define $B^{\Gamma}_r(\gamma)$ for $\gamma \in \Gamma$. If $h=e$ or $\gamma = e$ we just write $B^{H}_r$ and $B^{\Gamma}_r$, respectively. 
	
	\begin{definition}[(Linear) repetitivity for symbolic systems]
	\label{Def-SymbLR}
		In the situation above we say that $\mathcal{C} \in \mathcal{A}^{\Gamma}$ is {\em symbolically repetitive with respect to $d_{\Gamma}$} if for every $r \geq 1$ there is some $R= R(r) > 0$ such that every pattern $p \in [B_r^{\Gamma}]_{\mathcal{C}}$ occurs in $B_{R}(\gamma)$ for all $\gamma \in \Gamma$. We say that $\mathcal{C}$ is {\em symbolically linearly repetitive} if it is symbolically repetitive and there is a $C \geq 1$ such that one can choose $R(r) = Cr$ for all $r\geq 1$.
	\end{definition}
	
	 We now fix a {\em F{\o}lner exhaustion sequence}, i.e.\@ a F{\o}lner sequence $\mathcal{T} = (T_m)$ in $\Gamma$ with $e \in T_m \subsetneq T_{m+1}$ for all $m \in \NN$. Then $\mathcal{T}$ is also a strong F{\o}lner sequence by \cite[Lemma~2.7~(d)]{PS16}, hence a strong F{\o}lner exhaustion sequence in $\Gamma$.  
	
	\begin{definition}[Tempered repetitivity of symbolic systems] \label{defi:symbolicallytemperedrep}
		In the situation described above we define the {\em repetitivity index $\mathcal{R}^{\mathcal{T}}_{\mathcal{C}}: \NN \to \NN \cup \{+ \infty\}$ for $\mathcal{C}$ with respect to $\mathcal{T}$}  as 
		\[
		\mathcal{R}^{\mathcal{T}}_{\mathcal{C}}(m) := \inf\big\{n \in \NN:\,\, \mbox{every pattern } p \in [T_m^{-1}]_{\mathcal{C}} \mbox{ occurs in } \gamma T_n^{-1} \mbox{ for all } \gamma \in \Gamma\big\}.  
		\] 	
		We say that $\mathcal{C}$ is {\em symbolically tempered repetitive with respect to\@ $\mathcal{T}$} if 
		\[
		\inf_{m \in \NN} \frac{m_{\Gamma}(T_m)}{m_{\Gamma}\big( T_{\mathcal{R}^{\mathcal{T}}_{\mathcal{C}}(m)} \Big)} > 0. 
		\]
	\end{definition}
	
	\begin{proposition} \label{prop:shift-rep}
		For $\mathcal{C} \in \mathcal{A}^{\Gamma}$ we have the following assertions. 
		\begin{enumerate}[(a)]
			\item If $\mathcal{C}$ is symbolically (linearly) repetitive with respect to\@ $d_{\Gamma}$, then $\mathcal{I}(\mathcal{C})$ is (linearly) repetitive with respect to\@ $d_{H}$.  
			\item Let $H= \Gamma$ and let $\mathcal{T}$ be a F{\o}lner exhaustion sequence. Then
			$\mathcal{C}$ is symbolically tempered repetitive with respect to\@  $\mathcal{T}$ if and only if  $\mathcal{I}(\mathcal{C})$ is tempered repetitive with respect to\@ $\mathcal{T}$. 
		\end{enumerate}
	\end{proposition}

	\begin{remark}
		It is not difficult to come up with a more general statement for~(b) covering also cases where $H \neq \Gamma$. This is however a technical point that we do not need for the purposes of this work.
	\end{remark}

	\begin{proof}
		We first prove statement~(a). To this end we fix some 
		$r \geq 1$ along with some $r$-patch $P_{\mathcal{C}} := \big(\mu,\, B^H_r(g) \big)$ of $\mathcal{I}(\mathcal{C})$, where $g \in H$ and $\mu = \sum_{\gamma \in \Gamma \cap B_r^H(g)} \iota (\mathcal{C}(\gamma)) \delta_{\gamma}$, where $\iota:\mathcal{A} \to (0,\infty)$ is as defined in Lemma~\ref{lem:equivariant}. Let $h \in H$. Since $\Gamma \leq H$ is cocompact and $d_H$ is locally bounded there is an $r_0 > 0$ such that $\Gamma B^H_{r_0} = H$.  Hence there are $\gamma_g,\gamma_{h} \in \Gamma$ such that $g \in B^H_{r_0}(\gamma_g)$ and $h \in B^H_{r_0}(\gamma_h)$.  We write $q_{\mathcal{C}}:= \mathcal{C}_{|B^H_r(g) \cap \Gamma} \in \mathcal{A}^{B^H_r(g) \cap \Gamma}$ for the corresponding patch in $\mathcal{C}$. Then $q_{\mathcal{C}}$ occurs in 
		$$
		\tilde{q}_{\mathcal{C}}:=\mathcal{C}_{|B^H_{r+r_0}(\gamma_g) \cap \Gamma} \in \mathcal{A}^{B^H_{r+r_0}(\gamma_g) \cap \Gamma}.
		$$ 
		Note that we also used here the fact that $B^H_r B^H_s \subseteq B^H_{r+s}$ for $r,s \geq 0$ by the left-invariance of the metric $d_H$.
		Since $d_{\Gamma}$ is the restriction of $d_H$ and since $\mathcal{C}$ is symbolically (linearly) repetitive there is some $R= R(r + r_0) > 0$ (which in the case of linear repetitivity can be chosen as $R(r) = Cr$ for some universal $C \geq 1$) such that the pattern of $\tilde{q}_{\mathcal{C}}$ occurs in $B^{\Gamma}_{R(r + r_0)}(\gamma_h)$. 
		Thus, we find by virtue of Lemma~\ref{lem:equivariant} that the pattern of $P_{\mathcal{C}}$ must occur in $B^H_{R(r + r_0) + r_0}(h)$. We now set $R^{\prime}(r) = r_0 + R(r + r_0)$ (and in the linearly repetitive situation we have $C^{\prime}:=r_0+C(1+r_0) \geq C$ such that $r_0 + C(r + r_0) \leq C^{\prime}r$ for all $r \geq 1$ and define $R^{\prime}(r) = C^{\prime} r$). In any case we see that the pattern of $P_{\mathcal{C}}$ occurs in $B^H_{R^{\prime}(r)}(h)$. This finishes the proof. 
	
		The validity of assertion~(b) is clear from the definitions and from Lemma~\ref{lem:equivariant}.
	\end{proof}

	\medskip

	We are now in position to prove the following
	which contains all assertions claimed in
	 Theorem~\ref{thm:LRsymbolic} as subcases.
	
	\begin{theorem} \label{thm:mainsymb}
		Let $\Gamma$ be a countable amenable group that is co-compactly embedded as a lattice in an lcsc group $H$ carrying an adapted metric $d_H$.
		Let $\mathcal{A}$ be a finite set and assume that $\mathcal{C} \in \mathcal{A}^{\Gamma}$. Then the following assertions hold. 
		\begin{enumerate}[(a)]
			\item If $\mathcal{C}$ is symbolically tempered repetitive with respect to some F{\o}lner exhaustion sequence, then the hull $\Omega_{\mathcal{C}}$ is minimal and uniquely ergodic. 
			\item If $\mathcal{C}$ is symbolically linearly repetitive with respect to the metric $d_{\Gamma}$ induced by $d_H$ and if $H$ has exact polynomial growth with respect to $d_H$, then the hull
			\[ 
			\mathcal{H}_{\mathcal{I}(\mathcal{C})} = \overline{\big\{ h.\mathcal{I}(\mathcal{C}):\, h \in H \big\}}^{w*}
			\]
			is minimal and uniquely ergodic, where $\mathcal{I}(\mathcal{C})$ is defined according to the map of Lemma~\ref{lem:equivariant}. 
			\item If $\mathcal{C}$ is symbolically linearly repetitive with respect to some adapted metric $d_{\Gamma}$ on $\Gamma$, and if $\Gamma$ has exact polynomial growth with respect to $d_{\Gamma}$, then the hull ${\Omega}_{\mathcal{C}}$ is minimal and uniquely ergodic. 
		\end{enumerate}
	\end{theorem}
	
	\begin{proof}
		We first prove the assertion~(a). By Proposition~\ref{prop:shift-rep}~(b) we find that $\mathcal{I}(\mathcal{C})$ is tempered repetitive as a weighted Delone set over $\Gamma$ with respect to some F{\o}lner exhaustion sequence. It follows from Theorem~\ref{Intro2} that the dynamical system $\Gamma \curvearrowright \mathcal{H}_{\mathcal{I}(\mathcal{C})}$ is minimal and uniquely ergodic. By Proposition~\ref{prop:isomorphic} these features also carry over to the system $\Gamma \curvearrowright \Omega_{\mathcal{C}}$. 
		
		In order to prove the assertion~(b) we note first that $\mathcal{I}(\mathcal{C})$ is linearly repetitive with respect to\@ $d_H$ as a weighted Delone set in $H$ by Proposition~\ref{prop:shift-rep}~(a). Since $H$ has exact polynomial growth with respect to\@ $d_H$ we infer from Theorem~\ref{LRConvenient} that the system $H \curvearrowright \mathcal{H}_{\mathcal{I}(\mathcal{C})}$ is minimal and uniquely ergodic. 
		
		Assertion~(c) follows from the assertion~(b) in the case $H = \Gamma$, while observing with  Proposition~\ref{prop:isomorphic} that the dynamical systems $\Gamma \curvearrowright \Omega_{\mathcal{C}}$ and $\Gamma \curvearrowright \mathcal{H}_{\mathcal{I}(\mathcal{C})}$ are topologically isomorphic.
	\end{proof}

\begin{proof}[Proof of Theorem~\ref{thm:LRsymbolic}]
The first statement is just assertion~(a) of the above theorem. The ``in particular''-statement of Theorem~\ref{thm:LRsymbolic} is exactly part~(c) of the above theorem.
\end{proof}

\section{Two applications} 
\label{sec:applications}

This section is devoted to the proofs of Corollaries~\ref{cor:Banachdensities} and~\ref{cor:IDS}.

\medskip

If $(T_m)$ is a (left) F\o lner sequence, then $\lim_{m}\frac{m_G(\partial_K(T_m))}{m_G(T_m)}=0$ for every compact $K \subseteq G$. In order to prove Corollary~\ref{cor:Banachdensities} and Corollary~\ref{cor:IDS}, we need that for a suitable boundary notion a similar fraction goes to zero for $(T_m^{-1})$ which will be a (right) F\o lner sequence. 
In light of this define the {\em (right)-$K$-boundary} of $S$ by
$$
\tilde{\partial}_K(T) := T K^{-1} \cap (G \setminus T) K^{-1}.
$$

The following is a simple modification of \cite[Lemma~2.3]{PogorzelskiThesis14} for $\partial_K$ replaced by $\tilde{\partial}_K$.

\begin{lemma}\label{lem:RightFolner}
Let $K,L,S,T\subseteq G$. Then the following assertion hold true.
\begin{itemize}
\item[(a)] $\tilde{\partial}_K(S\cup T) \subseteq \tilde{\partial}_K(T) \cup \tilde{\partial}_K(S)$,
\item[(b)] $\tilde{\partial}_K(S\setminus T) \subseteq \tilde{\partial}_K(T) \cup \tilde{\partial}_K(S)$,
\item[(c)] $\tilde{\partial}_K(T)\subseteq \tilde{\partial}_L(T)$ if $K\subseteq L$,
\item[(d)] {$\tilde{\partial}_K\big(\tilde{\partial}_L(T)\big) \subseteq \tilde{\partial}_{KL}(T)$},
\item[(e)] if $e\in K$, then $xK\subseteq \tilde{\partial}_{K^{-1}}(T)\cup T$ holds for each $x\in T$.
\end{itemize}
Furthermore, if $(T_m)$ is a strong F\o lner exhaustion sequence then $(T_m^{-1})$ is an exhaustion sequence and for every  compact $K \subseteq G$,
$$
\lim_{m\to\infty}\frac{m_G(\tilde{\partial}_K(T_m^{-1}))}{m_G(T_m^{-1})} = 0.
$$
\end{lemma}

\begin{proof}
Assertion (a)-(c) are immediate from the definition of $\tilde{\partial}_K(T)$. As for assertion~(d) we note that 
$$
\tilde{\partial}_K\big(\tilde{\partial}_L(T)\big)
	\subseteq \big\{ g\in G :\, \mbox{there is an } x\in K \mbox{ such that } gx\in \tilde{\partial}_L(T) \big\} 
	= \tilde{\partial}_L(T) K^{-1}.	
$$
The claim now follows from
\begin{eqnarray*}
\tilde{\partial}_L(T) K^{-1} &=& \big( TL^{-1} \cap (G \setminus T)L^{-1} \big) K^{-1} \subseteq TL^{-1}K^{-1} \cap (G \setminus T) L^{-1}K^{-1} \\
&=& T \big( KL \big)^{-1} \cap \big( G \setminus T \big)\big( KL \big)^{-1}.
\end{eqnarray*}
Assertion (e) is proven as follows. Let $x\in T$ and $g\in xK$. If $g\in T$ there is nothing to prove. Otherwise $g\in G\setminus T$ and so $g\in (G\setminus T)K$ as $e\in K$. Since $g\in TK$, we derive $g\in \partial_{K^{-1}}(T)$.

\medskip

Suppose $(T_m)$ is a strong F\o lner exhaustion sequence. Then clearly $(T_m^{-1})$ is an exhaustion sequence. Since $G$ is unimodular, $m_G(S)=m_G(S^{-1})$ holds for all $S\in\mathcal{RK}(G)$. Thus, 
$$
m_G\big(\tilde{\partial}_K(T_m^{-1}) \big)  = m_G\big( T^{-1}_m K^{-1} \cap (G \setminus T^{-1}_m) K^{-1} \big)
    =  m_G\big( KT_m  \cap K (G \setminus T_m)\big)
    = m_G\big( \partial_{K^{-1}}(T_m)\big)
$$
holds for each compact $K \subseteq G$ and $m \in \NN$. Hence, $\lim_{m\to\infty}\frac{m_G(\tilde{\partial}_K(T_m^{-1}))}{m_G(T_m^{-1})} = 0$ follows as $(T_m)$ is a strong F\o lner sequence and $m_G(T_m)=m_G(T_m^{-1})$.
\end{proof}

\medskip

\subsection*{Equality of upper and lower Banach density}

The upper and lower Banach density for a weighted Delone set $\Pi$ along a F\o lner sequence $(T_m)$ are defined by
$$
\overline{d}(\Pi) := \limsup_{m\to\infty} \sup_{g\in G} \frac{\delta_{\Pi}(gT_m^{-1})}{m_G(T_m)} \quad\mbox{ and }\quad
\underline{d}(\Pi) := \liminf_{m\to\infty} \inf_{g\in G} \frac{\delta_{\Pi}(gT_m^{-1})}{m_G(T_m)}.
$$
The above limit inferior and limit superior are actually limits by Lemma~\ref{lemma:mainaux} and Remark~\ref{rem:mainaux} as the almost sub-additive weight function $w(S,\Pi):=\delta_{\Pi}(S^{-1})$ is actually additive, see below. Invoking Theorem~\ref{thm:abstr}, we show in Corollary~\ref{cor:Banachdensities} that the upper Banach density and the lower Banach density coincide if $\Pi$ is an almost tempered repetitive weighted Delone set.

\medskip

\begin{proof}[Proof of Corollary~\ref{cor:Banachdensities}]
Let $\Lambda$ be a weighted Delone set in $G$ that is almost tempered repetitive with respect to a strong F\o lner exhaustion sequence $(T_m)$.
We seek to prove that	
\[
b_{\Lambda} := \lim_{m \to \infty} \frac{\delta_{\Pi}(T_m^{-1})}{m_G(T_m)} 
\]
exists uniformly in $\Pi \in \mathcal{H}_\Lambda$. Due to Theorem~\ref{thm:abstr} it suffices to show that $w:\mathcal{RK}(G)\times \mathcal{H}_\Lambda\to\mathbb{R},\, w(K,\Pi):= \delta_{\Pi}(K^{-1}),$ is an admissible (almost sub-)additive weight function. Recall that $\Pi$ is a Radon measure on $G$. Thus, $w$ satisfies (W1) and (W3) with  $B = \emptyset$ and $\theta=0$. Recall that $h_*\delta_{\Pi}(S):=\delta_{\Pi}(h^{-1}S)$ holds for any $h\in G$ and $S\in\mathcal{RK}(G)$. Let $K\in\mathcal{RK}(G)$, $h\in G$ and $\Pi\in\mathcal{H}_\Lambda$, then 
$$
w(Kh^{-1},h.\Pi) = h_*\delta_{\Pi}\big((Kh^{-1})^{-1}\big) = h_*\delta_{\Pi}\big(hK^{-1}\big) = \delta_{\Pi}(K^{-1}) = w(K,\Pi)
$$
holds implying (W4) with $I = \emptyset$ and $\vartheta=0$. 

\medskip

Recall that for each $\Pi=(Q,\beta)\in\mathcal{H}_\Lambda$, $Q$ is a $(U,K)$-Delone set where $e\in U$ holds without loss of generality. Let $Q$ be an arbitrary left-$U$-uniformly discrete subset of $G$. It is straightforward to show that there exists a $V\subseteq G$ open (only depending on $U$) satisfying $e\in V$, $\overline{V}^{-1}=\overline{V}$ and $x\overline{V}\cap y\overline{V}= \emptyset$ for all $x,y\in Q$ and every left-$U$-uniformly discrete set $Q$. 
By the assertion~(e) of 
 Lemma~\ref{lem:RightFolner} one obtains that for each
$S \in \mathcal{RK}(G)$ and $x\in S$ we have $x\overline{V}\subseteq \tilde{\partial}_{\overline{V}}(S)\cup S$. Thus,
$$
\sharp (S\cap Q) = \sum_{x\in S\cap Q} \frac{m_G(x\overline{V})}{m_G(\overline{V})} 
	= \frac{1}{m_G(\overline{V})}  m_G\left(\bigsqcup\nolimits_{x\in S\cap Q}x\overline{V} \right)
	\leq \frac{1}{m_G(\overline{V})} \big( m_G(\tilde{\partial}_{\overline{V}}(S)) + m_G(S)\big)
$$
follows.

\medskip

Let $L,K\in\mathcal{RK}(G)$ with $K\subseteq L$ and $\Pi=(Q,\beta)\in\mathcal{H}_\Lambda$. Since $\delta_{\Pi}(S)= \sum_{x\in Q} \beta(x)\delta_x(S)$, we obtain
\begin{eqnarray*}
\big|w(L,\Pi)-w(K,\Pi)\big| &=& \delta_{\Pi}\big((L\setminus K)^{-1}\big)
	\leq \sigma \cdot \sharp \big((L\setminus K)^{-1}\cap Q\big) \\
	&\leq& \frac{\sigma}{m_G(\overline{V})} \Big( m_G(\tilde{\partial}_{\overline{V}}(L^{-1})) + m_G(\tilde{\partial}_{\overline{V}}(K^{-1})) + m_G(L\setminus K) \Big),
\end{eqnarray*}
where we used $\beta(x)\leq \sigma$ and $\tilde{\partial}_{\overline{V}}(L^{-1}\setminus K^{-1}) \subseteq \tilde{\partial}_{\overline{V}}(L^{-1}) \cup \tilde{\partial}_{\overline{V}}(K^{-1})$.
By the symmetry of $\overline{V}$ and the unimodularity of $G$, we observe that 
\[
m_G\big(  \tilde{\partial}_{\overline{V}}(L^{-1}) \big) = m_G\big( L^{-1}\overline{V} \cap (G\setminus L^{-1})\overline{V} \big) =  m_G\big(\overline{V}L \cap \overline{V}(G\setminus L) \big) = m_G\big( \partial_{\overline{V}}(L) \big)
\]
and likewise for $K$.
Hence, $w$ satisfies (W2) with $\eta:= \frac{\sigma}{m_G(\overline{V})}$ and $J:=\overline{V}$. It remains to prove the Condition~(W5). 

\medskip

Let $\varepsilon>0$ and suppose that $V$ is chosen as before. Fix $0 < \delta < 1$ such that $B_{2\delta}\subseteq V$ and $\delta<\min\big\{\frac{\sigma^{-1}}{2},\frac{m_G(\overline{V})}{2 + 4\sigma}\varepsilon\big\}$. Due to Lemma~\ref{lem:RightFolner}, there is an $m_0\in\mathbb{N}$ such that $m_G(\tilde{\partial}_{\overline{V}^3}(T_m^{-1}))\leq \delta m_G(T_m^{-1})$ holds for $m\geq m_0$. We use the properties (a)-(d) subsequently.
Let $m\geq m_0$ and $\Pi=(Q,\beta), \Phi=(P,\alpha)\in \mathcal{H}_\Lambda$ be such that $d_{T_m^{-1}}(\Pi,\Phi)<\delta$ where as defined in the Equalities~\eqref{eqn:formuladistMeasure} and~\eqref{eqn:formuladist}, we have
$$
d_{T_m^{-1}}(\Phi,\Pi) = \inf\big\{ \delta > 0:\, \big| \delta_{\Phi}\big( B_{\delta}(x) \big) - \delta_{\Pi}\big( B_{\delta}(x) \big)  \big| < \delta  \mbox{ for all } x \in (T_m^{-1} \cap P) \cup (T_m^{-1} \cap Q) \big\}.
$$
Set $S_1:=T_m^{-1}\setminus \tilde{\partial}_{\overline{V}}(T_m^{-1})$. Thus, for each $x\in S_1 \cap P$, $x\overline{V} \subseteq T_m^{-1}$ follows. Let $x\in S_1 \cap P$. Since $Q$ is $U$-uniformly discrete, $B_{\delta}\subseteq V$ and
$$
\left|\sum_{y\in Q\cap B_{2\delta}(x)} \beta(y) - \alpha(x)\right|
	= \big|\delta_{\Phi}\big(B_{\delta}(x)\big) - \delta_{\Pi}\big(B_{\delta}(x)\big)\big|
	<\delta
	<\sigma^{-1},
$$
we conclude that there is a unique $y_x\in T_m^{-1} \cap Q$ such that $y_x\in B_{\delta}(x)$ and $|\alpha(x)-\beta(y_x)|<\delta$. 
Let $S_2:=T_m^{-1}\setminus \tilde{\partial}_{\overline{V}^2}(T_m^{-1})$. Then $y\in S_2 \cap Q$ implies $y\overline{V}\subseteq S_1$. Thus, $|\delta_{\Pi}(B_\delta(y))-\delta_{\Phi}(B_\delta(y))|<\delta<\sigma^{-1}$ implies as before that there is a unique $x_y\in P\cap S_1$ satisfying $x_y\in B_\delta(y)$ and $|\alpha(x_y)-\beta(y)|<\delta$. Consequently, $S_2 \cap Q\subseteq \{y_x:\, x\in S_1 \cap P\}$. Hence, 
$$
(T_m^{-1} \cap Q) \subseteq \{y_x:\, x\in S_1 \cap P\} \cup (\tilde{\partial}_{\overline{V}^2}(T_m^{-1}) \cap Q)
$$
follows.
With this at hand, we derive
\begin{align*}
\big| w(T_m,\Phi) - w(T_m,\Pi)\big| 
	&= \left| \sum_{x\in T_m^{-1}\cap P} \alpha(x) - \sum_{y\in T_m^{-1}\cap Q} \beta(x) \right|\\
	&\leq \sum_{x\in S_1\cap P} |\alpha(x)-\beta(y_x)| + \sigma\Big( \sharp \big(\tilde{\partial}_{\overline{V}}(T_m^{-1}) \cap P\big) + \sharp \big(\tilde{\partial}_{\overline{V}^2}(T_m^{-1})\cap Q\big) \Big).
\end{align*}
By Lemma~\ref{lem:RightFolner}~(c) and the choice $m_0$, we have $m_G(\tilde{\partial}_{\overline{V}}(T_m^{-1}))\leq \delta m_G(T_m^{-1})$.
Since $\sharp (S\cap D )\leq \frac{1}{m_G(\overline{V})}\big(m_G(\tilde{\partial}_{\overline{V}}(S)) + m_G(S)\big)$ for any $U$-uniformly set $D$ and any $S\in\mathcal{RK}(G)$, the previous considerations lead to
\begin{align*}
\big| w(T_m,\Phi) - w(T_m,\Pi)\big| 
	\leq &\frac{\delta}{m_G(\overline{V})} \big(m_G(\tilde{\partial}_{\overline{V}}(T_m^{-1})) + m_G(T_m^{-1})\big) 
	+ \frac{\sigma}{m_G(\overline{V})}\Big( m_G\big( \tilde{\partial}_{\overline{V}}\big(\tilde{\partial}_{\overline{V}}(T_m^{-1})\big)\big) \\
	&+ m_G(\tilde{\partial}_{\overline{V}}(T_m^{-1})) + m_G\big( \tilde{\partial}_{\overline{V}}\big(\tilde{\partial}_{\overline{V}^2}(T_m^{-1})\big)\big) + m_G\big(\tilde{\partial}_{\overline{V}^2}(T_m^{-1})\big) \Big).
\end{align*}
Since $\tilde{\partial}_I(\tilde{\partial}_J(S))\subseteq \tilde{\partial}_{IJ}(S)$ holds by Lemma~\ref{lem:RightFolner}~(d), the estimate
$$
\big| w(T_m,\Phi) - w(T_m,\Pi)\big| 
	\leq \frac{2+4\sigma}{m_G(\overline{V})} \delta m_G(T_m^{-1})
	\leq \varepsilon m_G(T_m)
$$
is derived using $m_G(T_m^{-1})=m_G(T_m)$ which holds as $G$ is unimodular. Thus, $w$ satisfies (W5).
\end{proof}

\subsection*{Uniform convergence of the IDS}

We briefly sketch the model outlined in \cite{LSV11} and in \cite{PS16}. 
In order to be consistent with the previous sections we will consider Cayley graphs with a left-invariant metric as well as pattern equivalence defined by left translations. (In the mentioned papers the authors consider translations from the right but also deal with left F{\o}lner sequences.)

\medskip

We suppose that $\Gamma$ is an amenable countable group generated by a finite symmetric set $S \subset \Gamma$. 
For the Haar measure on $\Gamma$ we fix the normalized counting measure $m_{\Gamma}$ which satisfies 
$m_{\Gamma}(\{e \}) = 1$ for the identity element $e \in \Gamma$.

\medskip

Recall for finitely generated groups, F{\o}lner sequences $(F_n)$ consist of non-empty finite sets such that for all $s \in S$, one has 
\[
\lim_{n \to \infty} \frac{m_{\Gamma}\big( F_n\, \triangle \, sF_n \big)}{m_{\Gamma}(F_n)} = 0.
\]
This definition is equivalent to the notion of strong F{\o}lner sequences in the above sense by \cite[Lemma~2.7~(d)]{PS16}. The generating set $S$ comes with a natural left-invariant word metric. Precisely, for $g,h \in \Gamma$, one defines
\[
d_S(g,h) = \min\{L \in \NN:\, \mbox{there are } s_1, s_2, \dots, s_L \in S \mbox{ s.t. } s_1s_2 \cdots s_L = g^{-1}h \}.
\]
In the following we write 
\[
B^S_n(g) := \{h \in \Gamma:\, d_S(g,h) < n\}
\]
for $g \in \Gamma$ and $n \in \NN$. 

\medskip

Each group $\Gamma$ generated by a symmetric set $S$ comes with a Cayley graph with vertices given by $\Gamma$ and two elements $g,h \in \Gamma$ are connected by an edge if there exists some $s \in S$ such that $gs = h$. We write $\operatorname{Cay}(\Gamma, S)$ for this graph. 

\medskip

We consider a {coloring} $\mathcal{C}\in\mathcal{A}^\Gamma$ of the group $\Gamma$ which is nothing but a map $\mathcal{C}: \Gamma \to \mathcal{A}$, where $\mathcal{A}$ is a finite set whose elements are called colors.  
  As demonstrated in Section~\ref{sec:symbolic}, by assigning to each color in $a \in \mathcal{A}$ a different positive number $\iota(a) > 0$ we can identify $\mathcal{C}$ with the weighted Delone set $I(\mathcal{C}):= (\Gamma, \delta_{\mathcal{C}})$,
  where $\delta_{\mathcal{C}} := \sum_{x \in \Gamma} \iota(\mathcal{C}(x))  \delta_{\{ x\}} $on $\Gamma$. 
  Invoking Lemma~\ref{lem:equivariant} and Proposition~\ref{prop:isomorphic} we consider a coloring $\mathcal{C}$ both as mappings $\Gamma \to \mathcal{A}$ and as weighted Delone sets in the above form. 
Given $\mathcal{C}$, we define its hull by
\[
\Omega_{\mathcal{C}} := \overline{\{g.\mathcal{C}:\, g \in \Gamma\}},
\]
where the closure is taken with respect to the product topology in $\mathcal{A}^\Gamma$ induced from the discrete topology on $\mathcal{A}$. 
We now consider a coloring $\mathcal{C}$ that is symbolically tempered repetitive with respect to a strong F{\o}lner exhaustion sequence as of Definition~\ref{defi:symbolicallytemperedrep}.
It follows from Theorem~\ref{thm:mainsymb}~(a) that $\Gamma \curvearrowright \Omega_{\mathcal{C}}$ is uniquely ergodic.
Given a finite set $E \in \mathcal{F}(\Gamma)$ we define the $E$-patch for $\mathcal{C}$ as in Section~\ref{sec:symbolic} above. 

Since characteristic functions over the cylinder sets associated with patches are continuous we derive from Theorem~\ref{thm:abstr} (in combination with Proposition~\ref{prop:isomorphic}) that for every colored patch $\mathcal{E} := \mathcal{C}_{|E}$ with $E \subset \Gamma$ finite, the limits
\[
\nu(\mathcal{E}) := \lim_{m \to \infty} \frac{1}{m_{\Gamma}(T_m)} \sum_{x \in T_m} 1_{\mathcal{E}}(x.\mathcal{D})
\]
exist for all $\mathcal{D} \in \Omega_{\mathcal{C}}$, where for $\mathcal{D}^{\prime} \in \Omega_{\mathcal{C}}$, one has   
$1_{\mathcal{E}}(\mathcal{D}^{\prime}) = 1$ if and only if $ \mathcal{D}_{|E}^{\prime}= \mathcal{E}$ and $1_{\mathcal{E}}(\mathcal{D}^{\prime}) = 0$ otherwise. This yields
\begin{eqnarray*}
	\nu(\mathcal{E}) &:=& \lim_{m \to \infty} \frac{1}{m_{\Gamma}(T_m)} \sum_{x \in T_m} 1_{\mathcal{E}}(x.\mathcal{C}) = \lim_{m \to \infty} \frac{m_{\Gamma}\big(\{x \in T_m:\, (x.\mathcal{C})_{|E} = \mathcal{E}\}\big)}{m_{\Gamma}(T_m)} \\
	&=& \lim_{m \to \infty} \frac{m_{\Gamma}\big(\{ x \in T_m:\, \mathcal{C}_{|x^{-1}E} \cong \mathcal{E} \}\big)}{m_{\Gamma}(T_m)},
\end{eqnarray*}
where we write $\mathcal{C}_{|x^{-1}E} \cong \mathcal{E}$ if the patches $\mathcal{C}_{|x^{-1}E}$ and $\mathcal{E}$ are equivalent.  
With no loss of generality we will assume that the identity $e$ is contained in $E$ since otherwise we can pick $a \in E$ and repeat the above argument with the coloring $\mathcal{D} := a.\mathcal{C}$ and the patch $\mathcal{C}_{|a^{-1}E}$ instead of $\mathcal{C}$ and $\mathcal{E}$. 
 Furthermore, the sequence $(T_m^{-1})$ is a right-F{\o}lner sequence, see Lemma~\ref{lem:RightFolner}. Thus,
\[
\lim_{m \to \infty} \frac{m_{\Gamma}\big(\{ y \in \Gamma: y^{-1}E \subseteq T^{-1}_m \}\big)}{m_{\Gamma}(T_m)} = \lim_{m \to \infty} \frac{m_{\Gamma}\big(\{ y \in T_m: y^{-1}E \subseteq T^{-1}_m \}\big)}{m_{\Gamma}(T_m)} = 1.
\]
Consequently, we arrive at 
\[
\nu(\mathcal{E}) = \lim_{m \to \infty} \frac{m_{\Gamma}\big(\{ z \in \Gamma:\, z^{-1}E \subseteq T_m^{-1} \, \wedge\, \mathcal{C}_{|z^{-1}E} \cong \mathcal{E} \}\big)}{m_{\Gamma}(T_m)}
\]
for every colored patch $\mathcal{E} := \mathcal{C}_{|E} $ with $E \subset \Gamma$ finite.
In the setting of \cite[Section~5]{PS16} with right translations of patches replaced by left translations of patches the latter condition is equivalent to 
\begin{eqnarray} \label{eqn:frequences}
\nu(\mathcal{E}) = \lim_{m \to \infty} \frac{\#^l_\mathcal{E}(T_m^{-1})}{m_{\Gamma}(T_m^{-1})}
\end{eqnarray}
for all patches  $\mathcal{E} := \mathcal{C}_{|E} $ with $E \subset \Gamma$ finite, where $\#^l_\mathcal{E}(T_m^{-1})$ counts the number of occurrences of the pattern of  $\mathcal{E}$ in $T_m^{-1}$. Recall that $(T^{-1}_m)$ is a right F{\o}lner sequence. The setting in \cite{PS16} is reversed: there, one deals with right occurrences of colored patterns, with a right-invariant word metric and with left F{\o}lner sequences.  

\medskip

We turn to operators on the given left Cayley graph.

\medskip

For a subset $W \subseteq \Gamma$ we write $\ell^2(W, \mathcal{S})$ for the space of functions 
$u: W \to \mathcal{S}$ satisfying $\sum_{W \in \Gamma} \|u(g)\|^2_{\mathcal{S}} < \infty$, where $\mathcal{S}$ is a finite 
dimensional Hilbert space with norm $\| \cdot \|_{\mathcal{S}}$.
For a finite set $L \subset \Gamma$ and an operator $H: \ell^2(\Gamma, \mathcal{S}) \to \ell^2(\Gamma, \mathcal{S})$ we set
\[
H[L]: \ell^2(L, \mathcal{S}) \to \ell^2(L, \mathcal{S}),\quad  u \mapsto p_L H i_L\, u,
\]
where $i_L: \ell^2(L, \mathcal{S}) \to \ell^2(\Gamma, \mathcal{S})$ and $p_L: \ell^2(\Gamma, \mathcal{S}) \to \ell^2(L, \mathcal{S}) $ are the canonical inclusion and projection operators defined as 
\[
i_L(u)(g) = \begin{cases}
u(g) & \mbox{if } g \in L, \\
0 & \mbox{otherwise}
\end{cases}
,
\quad\quad\quad \mbox{and} \quad\quad  p_L(v)(g) = v(g) \quad  \mbox{ for all } g \in L.
\]
For $g,h \in \Gamma$ we write $H(g,h) := p_{\{g\}}H i_{\{h\}}$.

\medskip

On $\ell^2(\Gamma, \mathcal{S})$ we now define the class of operators which is of interest to us.

\begin{definition}\label{defi:operators}
	Let $H:\ell^2(\Gamma, \mathcal{S}) \to \ell^2(\Gamma, \mathcal{S})$ be a self-adjoint operator with $\Gamma, S$ and $\mathcal{S}$ as above. Further suppose that $\mathcal{C}: \Gamma \to \mathcal{A}$ is a coloring. We say that $H$ {\em has finite hopping range} if there is some $M \in \NN$ such that $H(g,h) = 0$ whenever $d_S(g,h) \geq M$. We say that $H$ is {\em $\mathcal{C}$-invariant} if there is some $N \in \NN$ such that $H(xg, xh) = H(g,h)$ for $x,g,h \in \Gamma$ that satisfy 
	\[
	x \big( \mathcal{C}_{|B_N(g) \cup B_N(h)} \big) = \mathcal{C}_{|B_N(xg) \cup B_N(xh)}.
	\]
	If $H$ has finite hopping range and is $\mathcal{C}$-invariant then we call $R:= \max\{M;N\}$ the {\em overall range of} $H$.
	\end{definition}

\medskip

We now define the eigenvalue counting function for a self-adjoint operator $A:\mathcal{S} \to \mathcal{S}$ defined on finite dimensional Hilbert space $\mathcal{S}$. 
Specifically, the {\em cumulative eigenvalue counting function of $A$} is defined by 
\[
\operatorname{ev}(A)(E) := \sum_{\lambda \leq E} \operatorname{mult}(\lambda),
\]
where the sum is taken over the set of all eigenvalues of $A$ less or equal than $E$ and $\operatorname{mult}(\lambda)$ is the multiplicity of the eigenvalue $\lambda$. 

\medskip

In the sequel, we fix once and for all a self-adjoint operator $H:\ell^2(\Gamma, \mathcal{S}) \to \ell^2(\Gamma, \mathcal{S})$ that has finite hopping range and is $\mathcal{C}$-invariant with respect to a coloring $\mathcal{C}: \Gamma \to \mathcal{A}$ with overall range $R \in \NN$.

\medskip

 Given a finite non-empty set $F \subset \Gamma$ and $E \in \RR$ we define
\[
H_F := H[F^R] \quad\quad \mbox{and} \quad\quad \overline{N}(F ): \RR \to \NN, \quad \overline{N}\big(F\big)(E) := \operatorname{ev}\big( H_F \big)(E),
\] 
where $F^R := \{x \in F:\, B^S_{R+1}(x) \subseteq F\}$. Then the functions
\begin{eqnarray} \label{eqn:empiricaldist}
N_H(F): \RR \to [0,1], \quad E \mapsto \frac{\overline{N}(F)(E)}{m_{\Gamma}(F)\operatorname{dim}(\mathcal{S})}
\end{eqnarray}
are elements in the Banach space of bounded right continuous functions on $\RR$ with supremum norm, denoted by $\big(\operatorname{BRC}(\RR),\, \| \cdot \|_{\infty} \big)$. They denote the empirical eigenvalue distribution of the operators $H_F$ for non-empty finite sets $F \subset \Gamma$.
 Note that 
dividing by $\operatorname{dim}(\mathcal{S})$ gives the correct normalization to obtain numbers in $[0,1]$ since 
\[
\operatorname{dim} \big( \ell^2(F,\mathcal{S}) \big) = m_{\Gamma}(F) \cdot \operatorname{dim}(\mathcal{S})
\]
for finite $F \subset \Gamma$.

\medskip

\begin{proof}[Proof of Corollary~\ref{cor:IDS}]
For the proof of Corollary~\ref{cor:IDS}, suppose that the coloring $\mathcal{C}:\Gamma \to \mathcal{A}$ is symbolically tempered repetitive with respect to $(T_m)$. Let $H:\ell^2(\Gamma, \mathcal{S}) \to \ell^2(\Gamma, \mathcal{S}) $ be a self-adjoint, $\mathcal{C}$-invariant finite hopping range operator with overall range $R \in \NN$. Due to the existence of the frequencies as displayed in~\eqref{eqn:frequences} and taking into account the reversed model as in Sections~5 and~7 of \cite{PS16}, we obtain from Theorem~7.4 in \cite{PS16} that there is some $N_H \in \operatorname{BRC}(\RR)$ such that 
\[
\lim_{m \to \infty} \Big\| N_{H_m}  -N_H \Big\|_{\infty} = 0,
\]
where  $H_m = H[(T_m^{-1})^R]$ and the $N_{H_m}$ are the empirical eigenvalue distributions 
as of~\eqref{eqn:empiricaldist}.
\end{proof}

\medskip

\begin{remark}
	Theorem~7.4 in \cite{PS16} also shows that $N_H$ is the distribution function of a probability measure $\mu_H$ on $\RR$. 
	\end{remark}

\appendix

\section{Weight functions and convergence}\label{AppendixA}

In this appendix we give the proofs of Lemma~\ref{lemma:tilingswf} and of Lemma~\ref{lemma:mainaux}. 

\medskip

\begin{proof}[Proof of Lemma~\ref{lemma:tilingswf}]
	Fix $0 < \varepsilon < 1/10$, a compact $I \subseteq G$ as well as $\zeta = \varepsilon / 16$. Fix a strong F{\o}lner exhaustion sequence $(S_l)$. 
	We first choose $m_I \in \NN$ large enough	such that $m_G(\partial_I(S_k)) \leq \varepsilon m_G(S_k)$ for all $k \geq m_I$. Fix $n \in \NN$ with $n \geq m_I$. Using Theorem~\ref{thm:tiling}, we find prototiles $S_i^{\varepsilon} \in \{ S_l:\, l \geq \max\{i, n \} \}$
	 for $1 \leq i \leq N(\varepsilon)$ and $\delta_0 > 0$ such that every 
	$\big( S_{N(\varepsilon)}^{\varepsilon} S_{N(\varepsilon)}^{\varepsilon\, -1},\, \delta_0 \big)$-invariant compact subset $A \subseteq G$ can be $\varepsilon$-quasi tiled while satisfying the properties (T1) to (T4) given in the statement of Theorem~\ref{thm:tiling}. So in the following we assume that $A$ is  $\big( S_{N(\varepsilon)}^{\varepsilon} S_{N(\varepsilon)}^{\varepsilon\, -1},\, \delta_0 \big)$-invariant and also 
	$(L, \varepsilon)$-invariant, where $L \in \{J,B,I\}$ and we fix 
	finite center sets $C_i^A \subseteq A$ giving rise to an $\varepsilon$-quasi tiling. By Remark~\ref{rem:tiling}~(ii) we can also make sure that 
	\begin{equation} \label{eqn:inv}
	\frac{m_G(\partial_L(S_i^{\varepsilon}))}{m_G(S_i^{\varepsilon})} < \zeta^2, \quad \frac{m_G(\partial_L(\tilde{S}_i^{\varepsilon}(c)))}{m_G(S_i^{\varepsilon})} < 4 \zeta
	\end{equation}	
	for $L \in \{J,B\}$ and for all $1 \leq i \leq N(\varepsilon)$ and $c \in C_i^A$.
	
	We define 
	\[
	\Delta = \frac{1}{m_G(A)}\Bigg(v(A) - \sum_{i=1}^{N(\varepsilon)} \sum_{c \in C_i^A} v\big( S_i^{\varepsilon}c \big)  \Bigg),
	\]
	as well as $A_{\varepsilon} := A \setminus \bigcup_{i=1}^{N(\varepsilon)} \bigcup_{c \in C_i^A} \tilde{S}_i^{\varepsilon}(c)$.
	We obtain
	\begin{eqnarray*}
	m_G(A) \cdot \Delta &=& \Bigg( v(A) - \sum_{i,c} v(\tilde{S}^{\varepsilon}_i(c)) - v(A_{\varepsilon}) \Bigg) - \sum_{i,c} \big( v(S_i^{\varepsilon}c) -  v(\tilde{S}^{\varepsilon}_i(c)) \big) + v(A_\varepsilon).
\end{eqnarray*}		
	Furthermore, we use the sub-additivity property~(w3) of $v$, combined with the general inclusions
	\[
	\partial_E(C \cup D) \subseteq \partial_E(C)\, \cup \, \partial_E(D), \quad \partial_E(C \setminus D) \subseteq \partial_E(C)\, \cup \, \partial_E(D)
	\] 
	for general sets $C,D,E \in \mathcal{RK}(G)$
	and the triangle inequality 
	in order to  obtain 
	\begin{eqnarray*}
	m_G(A) \cdot \Delta &\leq& \theta(v) \Bigg( 2 \cdot \sum_{i,c} m_G\big( \partial_B(\tilde{S}^{\varepsilon}_i(c)) \big) + m_G(\partial_{B}(A))   \Bigg)  
	 + \sum_{i, c} \big| v(S_i^{\varepsilon}c) - v(\tilde{S}^{\varepsilon}_i(c) ) \big|  +  v(A_\varepsilon).
	\end{eqnarray*}
Next, we use the almost monotonicity condition~(w2) to find that 
\begin{align*}
	m_G(A) \cdot \Delta \leq &\theta(v) \Bigg( 2 \cdot \sum_{i,c} m_G\big( \partial_B(\tilde{S}^{\varepsilon}_i(c)) \big) + m_G(\partial_{B}(A))   \Bigg)  +  v(A_\varepsilon)  \\
	&+ \eta(v) \Big( \sum_{i, c} m_G\big( S_i^{\varepsilon}c \setminus  \tilde{S}^{\varepsilon}_i(c) \big) \Big)  + \eta(v) \sum_{i, c} m_G\big( \partial_J(\tilde{S}^{\varepsilon}_i(c)) \big) + \eta(v) \sum_{i, c} m_G\big( \partial_J({S}^{\varepsilon}_i c) \big).
\end{align*}

	In view of the conditions (w1) and (w2) (and the above set inclusions for boundaries), we also have
	\[
	v(A_{\varepsilon}) \leq |v(A_{\varepsilon})| \leq  \eta(v) \Bigg( m_G(A_{\varepsilon}) + m_G(\partial_J(A)) + \sum_{i,c} m_G\big( \partial_J(\tilde{S}^{\varepsilon}_i(c) \big) \Bigg).
	\]
	Plugging this into the above inequality and using $m_G(S_i^{\varepsilon}c \setminus \tilde{S}_i^{\varepsilon}(c)) < \varepsilon$ along with 
	the invariance conditions on the prototiles, cf.\@ Equality~\eqref{eqn:inv} and $\zeta^2 \leq 4\zeta$, we arrive at
\begin{align*}
m_G(A) \cdot \Delta \leq& \theta(v)  8 \zeta \sum_{i,c} m_G(S_i^{\varepsilon}c) +   \theta(v)m_G(\partial_B(A)) 
	+\eta(v) m_G(A_{\varepsilon})   +  \eta(v) m_G(\partial_J(A))\\
	&+ \eta(v) 4\zeta  \sum_{i,c} m_G(S_i^{\varepsilon}c)
	+ \eta(v) \varepsilon \sum_{i,c} m_G(S_i^{\varepsilon}c)
	+ \eta(v)4\zeta \sum_{i,c} m_G(S_i^{\varepsilon}c)\\
	&+ \eta(v)4\zeta \sum_{i,c} m_G(S_i^{\varepsilon}c). 
\end{align*}	
By the tiling property~(T4) of Theorem~\ref{thm:tiling} and using the fact that $\sum_i \sum_c m_G(S_i^{\varepsilon}c) \leq 2m_G(A)$, we get 
\begin{eqnarray*}
\Delta &\leq&  16 \theta(v) \zeta  + \theta(v) \frac{m_G(\partial_B(A))}{m_G(A)}  +  2 \eta(v) \varepsilon + \eta (v)\frac{m_G(\partial_J(A))}{m_G(A)} +  8 \eta(v) \zeta  \\
&& \quad  + 2 \eta(v) \varepsilon + 8\eta(v)\zeta + 8 \eta(v) \zeta.
 \end{eqnarray*}
We finally use the additional invariance assumptions on $A$, i.e.\@ $m_G(\partial_J(A)) \leq \varepsilon m_G(A)$ and $m_G(\partial_B(A)) \leq \varepsilon m_G(A)$. Hence, the choice $\zeta = \varepsilon / 16$ finally leads to 
 \[
 \Delta \leq 2 \theta(v) \cdot \varepsilon + 8 \eta(v) \cdot \varepsilon.
 \]
\end{proof}

\medskip

For the proof of Lemma~\ref{lemma:mainaux}
we need another lemma first. 

\medskip

\begin{lemma} \label{lemma:minimality}
	Let $w:\mathcal{RK}(G) \times X \to \RR$ be an almost sub-additive weight function satisfying (W5$^\ast$) and assume that $G \curvearrowright X$ is minimal. 
	Then for every strong F{\o}lner exhaustion sequence $(T_m)$ and for each $\varepsilon> 0$, there is some $m_0 \in \NN$ such that 
	\[
	\big| w^{+}(T_m, x) - w^{+}(T_m, y)   \big| < \varepsilon, \quad \big| w^{-}(T_m, x) - w^{-}(T_m, y)  \big| < \varepsilon
	\]
	for all $x,y \in X$ and for all $m \geq m_0$.
\end{lemma}

\begin{proof}
	Let $\varepsilon > 0$. We now choose  $m_0 \in \NN$ such that property~(W5$^\ast$) of Remark~\ref{rem:W5-} holds applied with $\varepsilon / 3$ instead of $\varepsilon$. 
		If necessary we increase $m_0$ such that $m_G(\partial_I(T_m))/m_G(T_m) \leq \varepsilon/(6 \vartheta)$ holds for all $m \geq m_0$, where $I \subseteq G$ is the compact subset of property~(W4) that only depends on $w$. 
	Let $m \geq m_0$ and fix $x,y \in X$. 
	By definition of $w^{+}$ and the property~(W4) of Definition~\ref{defi:weightfunction}, we find and element $h^m_x \in G$ such that
	\[
	\Bigg| \frac{w(T_m, h^m_x x)}{m_G(T_m)} - w^{+}(T_m,x)  \Bigg| < \frac{\varepsilon}{3} + \vartheta\cdot \frac{m_G(\partial_I(T_m))}{m_G(T_m)}. 
	\]
	Moreover, since the action $G \curvearrowright X$ is minimal, we find $j_m \in G$ such that 
	$d_{X}(h^m_x x, j_m y ) < \delta_m$ (with $\delta_m > 0$ chosen as in~(W5$^\ast$)), and by property~(W5$^\ast$) we arrive at
	\[
	\big| w(T_m, h^m_x x) - w(T_m, j_m  y)\big| < \varepsilon/3 m_G(T_m).
	\]
	Putting everything together and using again the property~(W4), we observe
	\begin{eqnarray*}
		w^{+}(T_m, x) &\leq&  \frac{w(T_m, h^m_x x)}{m_G(T_m)}  + \frac{\varepsilon}{3} + \vartheta \cdot \frac{m_G(\partial_I(T_m))}{m_G(T_m)} \\
		&\leq&  \frac{w(T_m j_m^{-1},  y)}{m_G(T_m)}  + \frac{2\varepsilon}{3} + 2 \vartheta \cdot \frac{m_G(\partial_I(T_m))}{m_G(T_m)}.
	\end{eqnarray*}
	By definition of $w^+$ and $m_G(\partial_I(T_m))/m_G(T_m) < \varepsilon (6\vartheta)^{-1}$, we conclude $w^{+}(T_m, x) \leq w^{+}(T_m, y) + \varepsilon$
	for all $m \geq m_0$.
	Interchanging the roles of $x$ and $y$ yields the statement for $w^{+}$. The assertion for $w^{-}$
	can be proven in the very same manner and we leave the details to the reader.
\end{proof}

\medskip

We are ready to give the proof of Lemma~\ref{lemma:mainaux}.

\medskip

\begin{proof}[Proof of Lemma~\ref{lemma:mainaux}]
 It suffices to show
	\[
	\limsup_{m \to \infty} {w^{+}}(T_m, x) \leq \liminf_{m \to \infty} {w^{+}}(T_m, x)
	\]
	for all $x\in X$. 
	Invoking Remark~\ref{rem:boundedness}, we see that both the above limsup and the liminf are contained in the interval $[-\eta, \eta]$ for all $x \in X$, where $\eta \geq 0$ is the constant as of (W2) in Definition~\ref{defi:weightfunction}. 
	So fix $x \in X$. We find a subsequence 
	$(S_l)$ of $(T_m)$ such that 
	\[
	\liminf_{m\to\infty}w^+(T_m,x) = \lim_{l\to\infty} w^+(S_l,x)\,.
	\]
	By property~(W4) there is a compact $I \subseteq G$ and $\vartheta \geq 0$ such that the almost-equivariance property for $w$ is satisfied.  
	Let  $0<\varepsilon<1/10$ and $N(\varepsilon):=\big\lceil -\log(\frac{\varepsilon}{1-\varepsilon})\big\rceil$. 
	We apply Lemma~\ref{lemma:tilingswf} with $B=B$, $I=I$, $J=J$, $\eta(v) = \eta$ and $\theta(v) = \theta$, 
	where $J$ and $\eta$ describe the almost monotonicity property~(W2) and $B$ and $\theta$ describe the almost sub-additivity property~(W3) of $w$. We choose $\ell \in \NN$ large enough such that $m_G(\partial_I(S_l)) \leq \varepsilon m_G(S_l)$
	for all $l \geq \ell$. 
	Hence, we find a collection of prototiles
	\[
	\{e\}\subseteq S_{n}\subseteq S_1^\varepsilon\subseteq\ldots\subseteq S_{N(\varepsilon)}^\varepsilon
	\,,\qquad 
	S_i^\varepsilon\in\big\{
	S_k \,:\, k\geq \max\{i, \ell\}
	\big\}
	\]
	taken from the sequence $(S_l)$ and there is some $M \in \NN$ such that for each $m \geq M$, the set $T_m$ can
	be $\varepsilon$-quasi tiled by the prototiles $S_i^{\varepsilon}$ with center sets $C_i^m$ for 
	$1 \leq i \leq N(\varepsilon)$ and
	\[
	\frac{w(T_m,y) }{m_G(T_m)} 
	\leq \frac{1}{m_G(T_m)} 
	\sum_{i=1}^{N(\varepsilon)} \sum_{c \in C_i^m} w(S_i^{\varepsilon}c, y) + \big(8 \eta  + 2 \theta\big)\,  \varepsilon
	\]
	for all $y \in X$. Further, combining the previous estimate 
	and property~(W4) of a weight function gives
	\begin{eqnarray*}
		\frac{w(T_m h,x)}{m_G(T_m)}
		&\leq & \sum_{i=1}^{N(\varepsilon)}\sum_{c\in C_i^m} \frac{m_G(S_i^\varepsilon c)}{m_G(T_m)}\,
		\frac{w(S_i^\varepsilon c h,x)}{m_G(S_i^\varepsilon c)}
		+ \vartheta \sum_{i=1}^{N(\varepsilon)} \sum_{c \in C_i^m} \frac{m_G(\partial_I(S_i^{\varepsilon}))}{m_G(T_m)} \\
		& &+ \vartheta \frac{m_G(\partial_I(T_m))}{m_G(T_m)} + \big(8 \eta  + 2 \theta\big)\,  \varepsilon
	\end{eqnarray*}
	for all $h \in G$ and $m \geq M$. 
	Increasing $M$ if necessary and using $m_G(\partial_I(S_l)) \leq \varepsilon m_G(S_l)$ for $l\geq \ell$ together with $\sum_i \sum_c m_G(S_i^{\varepsilon}c) \leq 2m_G(T_m)$, we arrive at  
	\begin{eqnarray*}
			\frac{w(T_m h,x)}{m_G(T_m)} &\leq &
			 \sum_{i=1}^{N(\varepsilon)}\sum_{c\in C_i^m} \frac{m_G(S_i^\varepsilon c)}{m_G(T_m)}\,
			\frac{w(S_i^\varepsilon c h,x)}{m_G(S_i^\varepsilon c)}
		 + \big( 8 \eta  + 2 \theta + 3 \vartheta \big) \varepsilon
		\end{eqnarray*}
		for all $h \in G$ and $m \geq M$.
	
	\medskip
	
	Let $S^\varepsilon:=S_{i_0}^\varepsilon$ be chosen such that $w^+(S^\varepsilon,x)=\max_{1\leq i\leq N(\varepsilon)}w^+(S_i^\varepsilon,x)$. Then 
	$$
	\frac{w(S_i^\varepsilon c h,x)}{m_G(S_i^\varepsilon c)} 
	\leq w^+(S^\varepsilon,x)\,
	$$
	for all $1 \leq i \leq N(\varepsilon)$.
	In addition, the tiling properties (T1)-(T4) listed in Theorem~\ref{thm:tiling} lead to 
	$$
	(1-2\varepsilon) m_G(T_m) 
		\leq \sum_{i=1}^{N(\varepsilon)}\sum_{c\in C_i^m} m_G(\widetilde{S}_i^\varepsilon (c)) 
		\leq  \sum_{i=1}^{N(\varepsilon)}\sum_{c\in C_i^m} m_G(S_i^\varepsilon c)
		\leq \sum_{i=1}^{N(\varepsilon)} \sum_{c \in C_i^m} \frac{m_G(\widetilde{S}_i^{\varepsilon}(c))}{1-\varepsilon} 
		\leq \frac{m_G(T_m)}{1-\varepsilon} .
	$$
	Combined with the previous estimate, this yields
	$$
	\frac{w(T_m h,x)}{m_G(T_m)}
	\leq  \max\Big\{ \frac{1}{1-\varepsilon} w^+(S^\varepsilon,x),\, (1-2\varepsilon) w^+(S^\varepsilon,x) \Big\}
	+ \big( 8 \eta  + 2 \theta + 3 \vartheta \big) \varepsilon
	$$
	for all $m \geq M$ and $h \in G$. (Note that the sign of $w^{+}(S^{\varepsilon}, x)$ determines the factor in front of it.)
	Taking the supremum over $h\in G$ and the limsup in $m$, we derive
	$$
	\limsup_{m\to\infty}w^+(T_m,x)
	\leq \max\Big\{ \frac{1}{1-\varepsilon} w^+(S^\varepsilon,x),\, (1-2\varepsilon) w^+(S^\varepsilon,x) \Big\}
	+ \big( 8 \eta  + 2 \theta + 3 \vartheta \big) \varepsilon.
	$$
	We recall that by definition, we have that $S^{\varepsilon} =  S_l$ for some $l \geq \ell$. Since $\ell$
	was chosen arbitrarily but large enough 
	and since $(S_l)$ is a subsequence of $(T_m)$, where $(w^{+}(T_m, x))_m$ attains 
	its limit inferior, we can send $\ell$ to infinity in order to see 
	\begin{eqnarray} \label{eqn:sl}
		\limsup_{m\to\infty}w^+(T_m,x)
		&\leq& \max\Big\{ \frac{1}{1-\varepsilon} \lim_{l\to\infty} w^+(S_l,x),\, (1-2\varepsilon)\lim_{l\to\infty} w^+(S_l,x) \Big\} \nonumber \\
		&& \quad
		+\big( 8 \eta  + 2 \theta + 3 \vartheta \big) \varepsilon\\
		&=& \max\Big\{ \frac{1}{1-\varepsilon} \liminf_{m\to\infty} w^+(T_m,x),\, (1-2\varepsilon) \liminf_{m\to\infty} w^+(T_m,x)\Big\} \nonumber \\
		 && \quad + \big( 8 \eta  + 2 \theta + 3 \vartheta \big) \varepsilon. \nonumber
	\end{eqnarray}
	Sending $\varepsilon \to 0$ gives what was claimed above.
	
	\medskip

	We show next that the limits do not depend on the choice of the F{\o}lner sequence. To this end, fix two nested F{\o}lner sequences $(T_m)$ and $(T^{\prime}_l)$. Then for each $\varepsilon < 1/10$, each $\ell \in \NN$,
	and large $m$, the set $T_m$ can be $\varepsilon$-quasi tiled by the same means as above by prototiles taken from $(T^{\prime}_l)_{l \geq \ell}$. Repeating exactly the same arguments as before leads to the 
	Inequality~\eqref{eqn:sl} with $(S_l)$ replaced by $T^{\prime}_{l}$. Sending $\varepsilon \to 0$ shows that the limit with respect to one sequence is less or equal than the limit with respect to the 
	other sequence. By symmetry, the independence follows. 
	
	\medskip

	If $w$ is even  a topological weight function, $(T_m)$ is a strong F{\o}lner exhaustion sequence and if the action $G \curvearrowright X$ is additionally minimal, then the limits $\lim_{m \to \infty} {w^{+}}(T_m, x)$ 
	must  coincide for all $x \in X$ and the convergence must be uniform by Lemma~\ref{lemma:minimality}. 
\end{proof}

\section{Topology of weighted Delone sets}\label{AppendixB}
This section is devoted to a particular useful neighborhood basis for the weak-$*$-topology on spaces of weighted Delone sets.

\medskip

We fix an lcsc group $G$, an open subset $U \subset G$, compact subset $K \subset G$ and $\sigma\geq 1$. We then consider weighted Delone sets in 
$\operatorname{Del}(U,K,\sigma)$, i.e.\ subsets of $G$ which are left-$U$-uniformly discrete and left-$K$-syndetic with weights in the interval $[\sigma^{-1}, \sigma]$.
 
\medskip

Let $S \in  \mathcal{RK}(G)$. Recall the notion
\begin{eqnarray*}
d_S\big( \Lambda, \Pi \big) &:=& \inf\big\{ \delta > 0:\, 
|\delta_\Lambda\big(B_{\delta}(y)\big) - \delta_\Pi\big(B_{\delta}(y)\big)| < \delta \mbox{ for all } y \in (P \cap S) \cup (Q \cap S)  \}.
\end{eqnarray*}
for two weighted Delone sets $\Lambda = (P,\alpha)$ and $\Pi = (Q,\beta)$.

\begin{proposition} \label{prop:neighborhood}
	Let $\Lambda$ be a weighted Delone set in $G$. Then for all $S \in \mathcal{RK}(G)$ and $\delta > 0$, the set 
	\[
	\mathcal{U}_{S,\delta}(\Lambda) := \big\{ \Pi\in \operatorname{Del}(U,K,\sigma) :\, d_S(\Lambda, \Pi) < \delta \big\}
	\]
	is a weak-$*$-neighborhood of $\Lambda$. Moreover, $\{ \mathcal{U}_{S,\delta}(\Lambda): S \in \mathcal{RK}(G),\,\delta > 0  \}$  defines a weak-$*$-neighborhood basis
	of $\Lambda$.
\end{proposition}

\begin{proof}
Let $0 < \delta < \sigma^{-1}/2$, $S\in \mathcal{RK}(G)$ and $\Lambda=(P,\alpha)\in \operatorname{Del}(U,K,\sigma)$. Define $S':=S.B_1\in\mathcal{RK}(G)$. Since $\sigma\geq 1$, $B_{2\delta}(x)\subseteq S'$  holds for all $x\in S$. 
There is no loss in generality in assuming that $e\in U$ and $\delta>0$ is small enough such that  $B_{2\delta}(x) \cap B_{2\delta}(y)=\emptyset$ for any $x,y\in D$ where $D$ is some $U$-uniformly discrete set.

\medskip

Set $K:=S'\setminus \bigcup_{x\in P\cap S'} B_{\delta}(x)$. By Urysohn's lemma, there is a $\phi\in\Cc_c(G)$ such that $0\leq \phi\leq 1$, $\phi(x)=0$ for all $x\in P$ and $\phi|_K = 1$. 
Furthermore, for each $x\in P\cap S'$, there are $\phi_x\in \Cc_c(G)$ such that $0\leq \phi_x\leq 1$, $\phi_x$ is supported in $B_{2\delta}(x)$ and $\phi_x|_{B_\delta(x)}\equiv 1$.
Then the set
$$
\mathcal{V} := 
	\big\{ 
		\Pi\in \operatorname{Del}(U,K,\sigma) :\, 
		|\delta_{\Lambda(}\psi)-\delta_{\Pi}(\psi)|<\delta \mbox{ for all } \psi\in\{\phi\}\cup\{ \phi_x :\, x\in P\cap S'\}
	\}
$$
is a weak-$*$-neighborhood of $\Lambda$ in $\operatorname{Del}(U,K,\sigma)$. We will show first that $\mathcal{V}\subseteq\mathcal{U}_{S,\delta}(\Lambda)$
showing that $\mathcal{U}_{S,\delta}(\Lambda)$ is indeed a weak-$\ast$-neighborhood of $\Lambda$.

\medskip

Let $\Pi=(Q,\beta)\in\mathcal{V}$. By definition of $\phi$, we have $\delta_{\Lambda}(\phi)=0$. Thus, $|\delta_{\Pi}(\phi)| = |\delta_{\Lambda}(\phi)-\delta_{\Pi}(\phi)|<\delta$ implies 
$
Q\cap S' \subseteq \bigcup_{x\in P\cap S'} B_\delta(x)
$
as $\delta<\sigma^{-1}$ and $\phi_{|K} =1 $.

\medskip

By the choice of $\phi_x$, we have $\delta_{\Lambda}(\phi_x)=\alpha(x)$ for all $x \in P \cap S^{\prime}$. 
In particular, for each $x\in P\cap S$, $|\delta_{\Lambda}(\phi_x)-\delta_{\Pi}(\phi_x)|<\delta$ and the condition
$B_{2\delta}(x) \cap B_{2\delta}(y) = \emptyset$ for $x,y \in Q$ with $x \neq y$ implies that there is a unique $z_x\in Q\cap B_{2\delta}(x) \subseteq Q\cap S'$. 
Since $Q\cap S' \subseteq \bigcup_{x\in P\cap S'} B_\delta(x)$, we conclude $z_x\in Q\cap B_{\delta}(x)$.
Furthermore, $\phi_x|_{B_\delta(x)}\equiv 1$ yields $\phi_x(x) = 1 = \phi_x(z_x)$. Thus,
$$
\big|\delta_{\Lambda}\big(B_{\delta}(x)\big)-\delta_{\Pi}\big(B_{\delta}(x)\big)\big| 
	= \big|\phi_x(x)\alpha(x) - \phi_x(z_x)\beta(z_x)\big|
	= \big|\delta_{\Lambda}\big(\phi_x\big)-\delta_{\Pi}\big(\phi_x\big)\big| 
	< \delta
$$
holds for each $x\in P\cap S$. If $y\in Q\cap S$ there is a unique $x\in P\cap S'$ such that $y \in B_{\delta}(x)$.
Thus, $\delta_{\Lambda}(B_{\delta}(x)) = \delta_{\Lambda}(B_{\delta}(y))$ and $\delta_{\Pi}(B_{\delta}(x)) = \delta_{\Pi}(B_{\delta}(y))$ for these choices of $x$ and $y$.
Hence,
\[
 \big|\delta_{\Lambda}\big(B_{\delta}(y)\big)-\delta_{\Pi}\big(B_{\delta}(y)\big)\big| = \big|\delta_{\Lambda}\big(B_{\delta}(x)\big)-\delta_{\Pi}\big(B_{\delta}(x)\big)\big| = \big| \delta_{\Lambda}(\phi_x) - \delta_{\Pi}(\phi_x) \big| < \delta 
\]
for all $y \in Q \cap S$ and the unique $x \in P \cap S^{\prime}$ with $y \in B_{\delta}(x)$. 
 Consequently, we have proven that $\Pi\in\mathcal{V}$ implies $d_S(\Lambda,\Pi)< \delta$. Thus, $\mathcal{U}_{S,\delta}(\Lambda)$ is a weak-$\ast$-neighborhood of $\Lambda$.

\medskip

Let $\Lambda=(P,\alpha)\in \operatorname{Del}(U,K,\sigma)$, $\phi \in\Cc_c(G)$ and $\varepsilon>0$. In order to show that $\{ \mathcal{U}_{S,\delta}(\Lambda): S \in \mathcal{RK}(G),\, \delta > 0 \} $  
defines a neighborhood basis, it suffices to show that there is an $S\in\mathcal{RK}(G)$ and a $\delta >0$ such that 
$$
\mathcal{U}_{S,\delta}(\Lambda)\subseteq \{\Pi\in \operatorname{Del}(U,K,\sigma) :\, |\delta_{\Lambda}(\phi)-\delta_{\Pi}(\phi)|<\varepsilon\}.
$$
Without loss of generality suppose that $e\in U$ and $B_{2\varepsilon}(x) \cap B_{2\varepsilon}(y) = \emptyset$ for all $x,y \in D$ with $x \neq y$ and for all 
$U$-uniformly discrete set $D$. 
Let $S\in\mathcal{RK}(G)$ be chosen such that it contains the compact support of $\phi$ and $S\cap P = S.B_{2\varepsilon} \cap P$. Set $N:=\sharp (S.B_\varepsilon\cap P)$. Since $\phi \in\Cc_c(G)$, 
there is an  $\delta>0$ such that $\delta<\min\big\{\varepsilon,\sigma^{-1},\frac{\varepsilon}{2N\|\phi\|_{\infty}}\big\}$ and $|\phi(x)-\phi(y)|<\frac{\varepsilon}{2N\sigma}$ holds for all $x,y\in G$ with $d_G(x,y)<\delta$.

\medskip

Let $\Pi=(Q,\beta)\in\mathcal{U}_{S.B_\varepsilon,\delta}(\Lambda)$. By the choice of $S$ each point $x \in P \cap S.B_{\varepsilon}$ is necessarily
contained in $P \cap S$. Hence for each $x\in P\cap S.B_\varepsilon = P \cap S$, the condition $|\delta_{\Lambda}\big(B_\delta(x)\big) - \delta_{\Pi}\big(B_\delta(x)\big)|<\delta\leq \varepsilon$ 
implies that there is a unique $y_x\in Q\cap S.B_{\varepsilon}\cap B_\delta(x)$ such that $|\alpha(x)-\beta(y_x)|<\delta \leq \frac{\varepsilon}{2N\| \phi \|_{\infty}}$. 
In particular, we have shown that $\{y_x\,: x\in P\cap S.B_\varepsilon\} \subseteq Q\cap S.B_\varepsilon$. We claim the latter two sets are actually equal.  

For indeed, if there was an $y\in Q\cap S.B_\varepsilon$ that is not equal to some $y_x$, then $B_\delta(y)\cap P\subseteq S.B_{2\varepsilon}\cap P$ and this would lead to $B_\varepsilon(y)\cap P=\emptyset$. Thus, 
$$|\beta(y)| = |\delta_{\Lambda}(B_\delta(y))-\delta_{\Pi}(B_\delta(y))| <\delta< \sigma^{-1}
$$ 
as $\Pi\in \mathcal{U}_{S.B_\varepsilon,\delta}(\Lambda)$, a contradiction to $\Pi\in \operatorname{Del}(U,K,\sigma)$. Furthermore, $\Pi\in \mathcal{U}_{S.B_\varepsilon,\delta}(\Lambda)$ yields
$$
|\beta(y_x)-\alpha(x)| = \big|\delta_{\Lambda}\big(B_\delta(y_x)\big) - \delta_{\Pi}\big(B_\delta(y_x)\big)\big| <\delta
$$
for all $x \in P \cap S.B_{\varepsilon}$.
Thus,
\begin{align*}
|\delta_{\Lambda}(\phi)-\delta_{\Pi}(\phi)| \leq &\sum_{x\in P\cap SB_\varepsilon} |\alpha(x) \phi(x) - \beta(y_x) \phi(y_x)|\\
	\leq &\sum_{x\in P\cap S.B_\varepsilon} \Big(|\alpha(x) - \beta(y_x)|\, |\phi(x)| + |\beta(y_x)|\, |\phi(x) - \phi(y_x)|\Big)
	<\varepsilon,
\end{align*}
using $d_G(x,y_x)<\delta$.
\end{proof}

\bibliographystyle{amsalpha}
\bibliography{references}
\end{document}